\documentclass[12pt]{amsart}

\usepackage{amsfonts,amsthm,amsmath,amssymb,amscd,mathrsfs}
\usepackage{graphics}
\usepackage{indentfirst}
\usepackage{color}
\usepackage{cite}
\usepackage{latexsym}
\usepackage[dvips]{epsfig}

\newtheorem{theorem}{Theorem}[section]
\newtheorem{remark}{Remark}[section]
\newtheorem{definition}{Definition}[section]
\newtheorem{lemma}[theorem]{Lemma}
\newtheorem{pro}[theorem]{Proposition}

\newenvironment{pf}{{\noindent \it \bf Proof:}}{{\hfill$\Box$}\\}

\renewcommand{\div}{{\rm div }}

\newcommand{\dis}{\displaystyle}

\newcommand{\bt}{\begin{theorem}}
\newcommand{\bl}{\begin{lemma}}
\newcommand{\el}{\end{lemma}}
\newcommand{\et}{\end{theorem}}

\newcommand{\la}{\label}

\newcommand{\bn}{\begin{eqnarray}}
\newcommand{\en}{\end{eqnarray}}
\newcommand{\bnn}{\begin{eqnarray*}}
\newcommand{\enn}{\end{eqnarray*}}

\newcommand{\ba}{\begin{aligned}}
\newcommand{\ea}{\end{aligned}}
\newcommand{\be}{\begin{equation}}
\newcommand{\ee}{\end{equation}}

\renewcommand{\la}{\label}

\newcommand{\Bv}{{\boldsymbol{v}}}

\newcommand{\Bn}{{\boldsymbol{n}}}

\newcommand{\Bu}{{\boldsymbol{u}}}
\newcommand{\Be}{{\boldsymbol{e}}}
\newcommand{\BF}{{\boldsymbol{F}}}

\newcommand{\bBU}{\bar{{\boldsymbol{U}}}}

\newcommand{\Bo}{{\boldsymbol{\omega}}}
\newcommand{\Bp}{{\boldsymbol{\phi}}}

\newcommand{\mcE}{\mathcal{E}}
\newcommand{\mcA}{\mathcal{A}}
\newcommand{\mcH}{\mathcal{H}}

\begin{document}

\title[Stability of Poiseuille Flows]
{Uniqueness and uniform structural stability of Poiseuille flows in a periodic pipe with Navier boundary conditions}

\author{Yun Wang}
\address{School of Mathematical Sciences, Center for dynamical systems and differential equations, Soochow University, Suzhou, China}
\email{ywang3@suda.edu.cn}

\author{Chunjing Xie}
\address{School of mathematical Sciences, Institute of Natural Sciences, and
Ministry of Education Key Laboratory of Scientific and Engineering Computing,
 Shanghai Jiao Tong University, 800 Dongchuan Road, Shanghai, China}
\email{cjxie@sjtu.edu.cn}

\begin{abstract}
In this paper, we prove the uniqueness and structural stability of Poiseuille flows for axisymmetric solutions of steady Navier-Stokes system supplemented with Navier boundary conditions in a periodic pipe. Moreover, the stability is uniform with respect to both the flux and the slip coefficient of Navier boundary conditions. It is also showed that the non-zero frequency part of the velocity is bounded by a power function of the flux with negative power as long as the flux is suitably large.  One of the key ingredients of the analysis is to prove the uniform linear structural stability, where the analysis for the boundary layers and the swirl velocity corresponding to the flux and the slip coefficients in different regimes plays a crucial role.

 \end{abstract}



\maketitle

\section{Introduction and Main Results}
A fundamental problem in fluid dynamics is to study steady flows in a domain with solid boundary. The typical governing equations for the incompressible  viscous flows are the following Navier-Stokes system
\begin{equation}\label{SNS}
\left\{
\begin{aligned}
& \Bu\cdot \nabla \Bu +\nabla p=\Delta \Bu+\BF,\\
&\div~\Bu=0,
\end{aligned}
\right.
\end{equation}
where $\Bu=(u^x, u^y, u^z)$ and $\BF=(F^x, F^y, F^z)$ are the velocity field and external force, respectively. Here we  formally put the viscosity coefficient to be the unity. When the domain $\Omega$ is bounded, the existence of solutions for the equation \eqref{SNS} supplemented with the Dirichlet boundary conditions for the velocity was first established by Leray in \cite{Leray} for the flows in simply connected domains, and was later obtained by Korobkov, Pileckas, and Russo in \cite{KPR} for the flows in multiply connected domains. However, the uniqueness is a very challenging problem. It was even showed in \cite{Yudovich} that there are multiple solutions for the problem on certain axisymmetric domain. For more references on the existence and uniqueness of solutions for steady Navier-Stokes system in bounded domains with Dirichlet boundary conditions, one may refer to \cite{Galdi}.

In 1933, Leray (\cite{Leray}) also proposed the problem on the existence of solutions for the Navier-Stokes system in a general nozzle supplemented with the no slip conditions on the nozzle wall and the asymptotic convergence to Poiseuille flows at far fields where the nozzle tends to be straight. The problem is called Leray problem nowadays (\cite{Galdi}). The major progress for this problem was made in \cite{Amick1, LS}, where the Leray problem was solved when the flux of the flow is small. In fact, Ladyzhenskaya and Solonnikov\cite{LS} even proved the existence of solutions with arbitrary flux  in general nozzles of bounded cross section, however, the uniqueness and far field behavior of the solutions with large fluxes are not clear.  Therefore, in order to solve the Leray problem, one needs only  to establish the uniqueness and   far field behavior of the solutions. There are lots of studies on the far field behavior for the solutions of Navier-Stokes system in a nozzle, one may refer to \cite{AmickF, MF, Horgan, HW, LS, AP,Pileckas}, etc. For more references on steady solutions of the Navier-Stokes equation in nozzles or other type of domains, please refer to the book by Galdi (\cite{Galdi}). However, all these results are under the assumption on the smallness of the flux.
A significant open problem posed in \cite[p. 19]{Galdi} is global well-posedness for Leray problem in a general nozzle when the flux $\Phi$ is large.

As discussed in \cite{LS}, in order to get the global well-posedness for  Leray problem in a general nozzle tending to a pipe,  a key step is to  prove global uniqueness of Poiseuille flow in a pipe with the no slip boundary conditions. As a first step to prove global uniqueness of Poiseuille flows, the local uniqueness was addressed in \cite{WX1,WX2}.  In fact,  the uniform structural stability of Poiseuille flows was established in \cite{WX1}, and it was even proved in \cite{WX2} that the solution is unique in a suitably large set when the flux is large. Furthermore, the solutions tend to the Hagen-Poiseuille flows exponentially fast as long as the external force tends to zero exponentially fast. For the Poiseuille flows in two dimensional infinitely long strip, the uniqueness of the solutions in the class of symmetric flows  was obtained in \cite{Rabier1}, while the uniqueness of the solutions in the class of general two dimensional flows was obtained only for the case with small flux.  The uniqueness and  uniform structural stability of two dimensional Poiseuille flows with arbitrary large flux in a periodic strip was achieved in \cite{SWX}.

For incompressible Newtonian fluid near the solid boundary,
the more general boundary conditions  are the following Navier boundary conditions,
\begin{equation}\label{slipBC}
\Bu\cdot \Bn=0,\quad 2\Bn \cdot D(\Bu) \cdot \boldsymbol{\tau} + \alpha \Bu\cdot \boldsymbol{\tau} = 0 \quad \text{on}\,\, \partial \Omega,
\end{equation}
where $D(\Bu)=\frac{1}{2}(\nabla \Bu+\nabla \Bu^t)$
is the symmetric part of the stress tensor, $\alpha\geq 0$ is the scalar coefficient which measures the tendency of a fluid to slip over the boundary and called the slip coefficient. The
boundary condition means
that the rate of strain on the boundary is proportional to the tangential slip velocity, which
was first proposed by Navier (\cite{Navier}) and derived for gases by Maxwell (\cite{Maxwell}).  This Navier boundary conditions were also justified rigorously in \cite{Jager} as an effective boundary
condition for flows over rough boundaries.
Formally, as $\alpha\to \infty$, the Navier boundary conditions become the Dirichlet  boundary conditions (no slip boundary conditions). The existence of steady solutions of Navier-Stokes system \eqref{SNS} supplemented with boundary conditions \eqref{slipBC} in simply connected bounded domains was obtained in \cite{AACG}. The existence of two dimensional and three dimensional axisymmetric solution for  Navier-Stokes system with Navier boundary conditions in a nozzle was obtained in \cite{Mucha1, Mucha2, Mucha3}  when the slip coefficient satisfies certain constraints.
 An important problem is to study the generalized Leray problem where the Navier conditions \eqref{slipBC} are prescribed on the solid boundary  and the far field behavior of the solutions corresponds to the shear flows satisfying the same boundary conditions.
 If  $\BF=0$, the straightforward computations show that  the Navier-Stokes system \eqref{SNS} in a pipe $\hat{\Omega} = B_1(0) \times \mathbb{R}$ supplemented with  the boundary conditions \eqref{slipBC}
and the flux constraint
\begin{equation}\label{fluxBC}
\int_{B_1(0)} u^z(x,y, z)\, dxdy = \Phi
\end{equation}
 has an explicit solution $\Bu=(0,0, \bar{u}^z)$ with
\be \label{1.8}
 \bar{u}^z = \bar{U}(r) = \frac{4 + 2\alpha}{4 + \alpha} \left(1 - \frac{2\alpha}{4+ 2\alpha}r^2\right)\frac{\Phi}{\pi},
\ee
where $r=\sqrt{x^2+y^2}$ and $\Phi\in \mathbb{R}$ is called the flux of the flow.
The solutions
$\bBU =\bar{U}(r) \Be_z$ are also called Poiseuille flows.
Clearly, if  $\alpha\to \infty$, the flow $\bBU$
tends to the Hagen-Poiseuille flow  $\bBU_{\infty} = \frac{2\Phi}{\pi}(1-r^2) \Be_z$, which is the shear flow of the Navier-Stokes system in a pipe with no slip boundary condition. If $\alpha =0$, then flow $\bBU_0 = \frac{\Phi}{\pi} \Be_z $  is a constant flow.  As same as the case for the classical Leray problem, the global uniqueness of Poiseuille flows \eqref{1.8} in a pipe plays an important role in studying the well-posedness for the generalized Leray problem with Navier boundary conditions.  It was mentioned in \cite{QWS} that many interesting physical phenomena show that the slip
coefficient should depend on viscosity or Reynolds number in general. Therefore, it should be important to prove the uniqueness of the solutions when both the flux and slip coefficients vary.

In this paper,
we focus on the uniqueness of Poiseuille flows \eqref{1.8} in periodic pipes which can be regarded as multiply connected domains. Without loss of generality, we assume that $\Omega=B_1(0)\times \mathbb{T}_{2\pi} $, i.e., the flows are periodic in the axial direction with period $2\pi$, and $\Phi$ is nonnegative. Let $\Bv= \Bu - \bBU$.
Consider  the following linearized perturbation  system
\be \label{2-0-1} \left\{ \ba
&\bBU \cdot \nabla \Bv + \Bv \cdot \nabla \bBU - \Delta \Bv + \nabla P = \BF \ \ \ \mbox{in}\ \Omega, \\
& {\rm div}~\Bv = 0\ \ \ \mbox{in}\ \Omega,
\ea
\right.
\ee
supplemented with the boundary conditions and flux constraint
\begin{equation}\label{slipBC1}
\Bv\cdot \Bn=0,\quad \ 2 \Bn \cdot D(\Bv) \cdot \boldsymbol{\tau}  + \alpha \Bu\cdot \boldsymbol{\tau}=0 \quad \text{on}\,\, \partial \Omega,
\ \ \ \ \int_{\Sigma} \Bv \cdot \Bn \, dS = 0.
\end{equation}


Our first main result is the following uniform linear structural stability of Poiseuille flows.

\bt \label{thm1}
Assume that $\BF= \BF(r, z) \in L^2(\Omega)$  is axisymmetric, and that $\BF$ satisfies
\be \label{compatibility1}
\int_{-\pi}^{\pi} \int_0^1 F^\theta(r, z) r^2 dr dz = 0.
 \ee
 The linear system  \eqref{2-0-1} supplemented with boundary conditions  and the flux constraint \eqref{slipBC1}  has a unique axisymmetric solution $\Bv$ which satisfies
\be \label{estuniformlinear}
\|\Bv\|_{H^{\frac32}(\Omega)} \leq C \|\BF\|_{L^2 (\Omega)}
\ee
and
\be \label{estimatelinear}
\|\Bv\|_{ H^2 (\Omega)} \leq C (1 + \Phi^{\frac14} ) \|\BF\|_{L^2 (\Omega)},
\ee
where $C$ is a uniform constant independent of $\BF$,  $\Phi$, and $\alpha$.
\et

There are several remarks on Theorem \ref{thm1}.

\begin{remark}
The key point of Theorem \ref{thm1} is that the constant $C$ in \eqref{estuniformlinear} does not depend on $\Phi$ and $\alpha$, which provides a uniform estimate for the solutions of \eqref{2-0-1}-\eqref{slipBC1} even when the coefficient of the system is very large and the slip coefficient is arbitrary.
\end{remark}

\begin{remark}
The condition \eqref{compatibility1} is a compatibility condition, which is used to guarantee  the existence of solutions in the case $\alpha =0$, but not necessary for the case $\alpha >  0$.
\end{remark}

\begin{remark}
The period of the solutions in this paper is $2\pi$. In fact, all the estimates and results in this paper hold for any solution which is periodic in $z$-direction. Certainly, the constant $C$ in \eqref{estuniformlinear}--\eqref{estimatelinear} may depend on the period.
\end{remark}

\begin{remark}
	Note that the problem \eqref{2-0-1}-\eqref{slipBC1} can be regarded as the case with $\lambda=0$ for the linearized spectral problem
	\begin{equation}\label{spectral}
		 \left\{ \ba
	&\lambda \Bv+ \bBU \cdot \nabla \Bv + \Bv \cdot \nabla \bBU - \Delta \Bv + \nabla P = \BF \ \ \ \mbox{in}\ \Omega, \\
	& {\rm div}~\Bv = 0\ \ \ \mbox{in}\ \Omega,\\
	&\Bv\cdot \Bn=0,\quad \ 2 \Bn \cdot D(\Bv) \cdot \boldsymbol{\tau}  + \alpha \Bu\cdot \boldsymbol{\tau}=0 \quad \text{on}\,\, \partial \Omega,
	\ \ \ \ \int_{\Sigma} \Bv \cdot \Bn \, dS = 0.
	\ea
	\right.
	\end{equation}
The  problem \eqref{spectral} plays an important role in studying the dynamical stability of Poiseuille flows in a pipe with Navier boundary conditions. So far, there has not yet been complete understanding for the problem \eqref{spectral}.  Recently, the stability and instability of motionless state in two dimensional strip with Navier boundary conditions was investigated in \cite{Ding1}, while the enhanced dissipation for the Poiseuille flows in two dimensional strip with Navier slip boundary condition in the case $\alpha=0$ was studied in \cite{Ding2}.
\end{remark}

With the aid of the uniform estimates for the linear system, we have the following results on uniform nonlinear structural stability of Poiseuille flows.
\begin{theorem}\label{mainthm}
Assume that $\BF=\BF(r, z)\in L^2(\Omega)$ is axisymmetric and that $\BF$ satisfies \eqref{compatibility1}.
\begin{enumerate}
\item[(a)]
There exists a constant $\varepsilon>0$ such that if
\begin{equation}\label{Fepsilon}
\|\BF\|_{L^2(\Omega)}\le \varepsilon,
\end{equation}
then the steady Navier-Stokes system \eqref{SNS} supplemented with the Navier slip boundary condition \eqref{slipBC}  and flux constraint \eqref{fluxBC} admits a unique axisymmetric strong solution $\Bu$ satisfying
\begin{equation}\label{est1}
\left\|\Bu - \bBU \right\|_{H^\frac32(\Omega)}\leq C\|\BF\|_{L^2(\Omega)}
\end{equation}
and
\begin{equation}\label{est2}
\|\Bu- \bBU\|_{H^2(\Omega)}\le C (1 + \Phi^\frac14 ) \|\BF\|_{L^2(\Omega)},
\end{equation}
where $C$ is a uniform constant independent of the flux $\Phi$, $\alpha$, and $\BF$.
\item[(b)] There exists a constant $\Phi_0$, such that for all $\Phi\ge \Phi_0$, if
\begin{equation}
\|\BF\|_{L^2(\Omega)}\le \Phi^\frac{1}{32},
\end{equation}
the steady Navier-Stokes equations \eqref{SNS} supplemented with the Navier boundary condition \eqref{slipBC} and flux constraint \eqref{fluxBC}  admit a unique axisymmetric  solution $\Bu$ satisfying \eqref{est1}-\eqref{est2}.
In particular, the solution satisfies
\be \nonumber
\| \mathcal{Q} (\Bu - \bBU ) \|_{L^2(\Omega)} \leq \Phi^{-\frac{7}{12}},
\ee
where $\mathcal{Q}$ is the projection operator defined by
\be \label{projection}
\mathcal{Q} \Bv = \sum_{n \neq 0} v_n^r e^{inz} \Be_r + v_n^z e^{inz} \Be_z + v_n^\theta e^{inz} \Be_\theta.
\ee
\end{enumerate}
\end{theorem}

\begin{remark}
Theorem \ref{mainthm} gives not only the existence but also the uniqueness  of axisymmetric solutions for the Navier-Stokes system supplemented with Navier boundary conditions near the Poiseuille flow \eqref{1.8} in the periodic pipes, which are multiply connected domains. Furthermore, the slip coefficient in Theorem \ref{mainthm} can be arbitrary and  the flux of the flows can be arbitrarily large, i.e., the background flow can be large.
In particular,  as implied in the proof of Theorem \ref{mainthm}, when the flux $\Phi$ is sufficiently large, $\BF=0$, and $\|\Bu - \bBU\|_{H^{\frac32}(\Omega)} \leq \Phi^{\frac{1}{64}},$ then $\Bu = \bBU$. This means that $\bBU$ is the unique solution in a large bounded set with radius $\Phi^{\frac{1}{32}}$ as the flux $\Phi$ is large.
\end{remark}

 \begin{remark}
  Another key point is that the constant $C$ in Theorem \ref{mainthm} is independent of both the flux and the slip coefficient.  The uniform estimates help to figure out how the solutions behave as $\alpha$ tends to $\infty$ or $0$. Similar uniform estimates with respect to $\alpha$ were obtained for the flows in bounded domains in \cite{AACG}.
  Furthermore, the boundary layers associated with different slip coefficients were studied in \cite{WWX} when the Reynolds number is high.
\end{remark}

The organization of the paper is as follows. In Section \ref{SecStream},  the stream function formulation for both the nonlinear and linearized problem are introduced.  Some uniform a priori estimates for the stream function associated with the  linear problem \eqref{2-0-1}-\eqref{slipBC1} in the axisymmetric case are given in Section \ref{Linear}. And consequently, some uniform estimates of the radial velocity and axial velocity  are derived. These estimates are established via different analysis for the problem with different frequencies.    The existence and uniform a priori estimates of the swirl velocity are proved in Section \ref{sec-swirl}.     With the aid of the analysis on the associated linearized problem and a fixed point argument, the uniform nonlinear structural stability of Poiseuille flows in axisymmetric case is proved in Section \ref{secnonlinear}.   Two appendices are included. The first one collects some lemmas and their proofs which are used here and there in the paper.
The second one gives the existence and regularity of solutions for the linearized problem,
which is quite similar to that in \cite{WX1}.



\section{Stream function formulation and linearized problem}\label{SecStream}
 { Suppose that $\Bu$ is an axisymmetric solution of \eqref{SNS} and $\bBU=\bar{U}(r) \Be_z$ is the Poiseuille flow with $\bar{U}(r)$ defined in \eqref{1.8}. Let
	\be \label{axisymmetricv}
	\Bv = \Bu - \bBU = v^r (r, z) \Be_r + v^\theta(r, z)  \Be_\theta + v^z (r, z) \Be_z.
	\ee
	Then  $\Bv$ satisfies the following nonlinear system
	\be \label{ap-40-1}
	\left\{\ba
	& \bar{U} (r) \frac{\partial v^r }{\partial z } + \frac{\partial P}{\partial r} -\left[ \frac{1}{r}\frac{\partial }{\partial r} \left( r \frac{\partial v^r}{\partial r} \right) + \frac{\partial^2 v^r}{\partial z^2}  - \frac{v^r}{r^2}  \right]  + v^r \partial_r v^r + v^z \partial_z v^r - \frac{(v^\theta)^2 }{r}=F^r, \\
	& v^r \frac{\partial \bar{U}}{\partial r} + \bar{U}(r) \frac{\partial v^z}{\partial z }+ \frac{\partial P}{\partial z} - \left[ \frac{1}{r}\frac{\partial }{\partial r} \left( r \frac{\partial v^z }{\partial r} \right) + \frac{\partial^2 v^z}{\partial z^2}    \right] +  v^r \partial_r v^z + v^z \partial_z v^z  = F^z, \\
	& \partial_r v^r + \partial_z v^z + \frac{v^r}{r} = 0, \ea \right.
	\ee
	and
	\be \label{ap-40-2}
	\bar{U}(r) \partial_z v^\theta - \left[ \frac{1}{r}\frac{\partial }{\partial r} \left( r \frac{\partial v^\theta  }{\partial r} \right) + \frac{\partial^2 v^\theta }{\partial z^2} - \frac{v^\theta}{r^2}    \right] + v^r \partial_r v^\theta + v^z \partial_z v^\theta +  \frac{v^r v^\theta}{r}  =F^\theta
	\ee
	in $D=\{(r,z)|(r,z)\in (0, 1)\times \mathbb{T}_{2\pi} \}$.
	Here $F^r$, $F^z$, and $F^\theta$ are the radial, axial, and azimuthal component of $\BF$, respectively.
	The Navier boundary conditions and the flux constraint \eqref{slipBC}-\eqref{fluxBC}  become
	\be \label{BC-1}
	v^r(1, z)  = 0,\ \ \ \ \partial_z v^r (1, z) - \partial_r v^z(1, z) = \alpha v^z(1, z), \ \ \ \quad  \int_0^1 r v^z(r, z)\, dr = 0,
	\ee
	and
	\begin{equation}\label{BC-swirl}
		\frac{\partial v^\theta }{\partial r} (1, z) = (1 - \alpha) v^\theta (1, z) .
	\end{equation}

Let $\Bo$ denote the vorticity of $\Bv$, i.e.,
\be \nonumber
 \Bo= {\rm curl}~\Bv = \omega^r \Be_r + \omega^\theta \Be_\theta + \omega^z \Be_z.
 \ee
Denote
\be \nonumber
\mathcal{L} = \frac{\partial }{\partial r} \left( \frac1r \frac{\partial }{\partial r} (r \cdot) \right) = \frac{\partial^2 }{\partial r^2} + \frac1r \frac{\partial }{\partial r} - \frac{1}{r^2}.
\ee
Taking the curl of the equations \eqref{ap-40-1} yields
\be \label{ap-40-5}
\bar{U}(r) \partial_z \omega^\theta - (\mathcal{L} + \partial_z^2)\omega^\theta + \partial_r (v^r \omega^\theta) + \partial_z (v^z \omega^\theta)  -
\partial_z  \left( \frac{ (v^\theta)^2 }{r}        \right)= \partial_z F^r- \partial_r  F^z.
\ee
	
	\subsection{Stream function formulation}\label{sec-stream}
The velocity $\Bv$, which is $2\pi$-periodic with respect to $z$, can be represented in terms of  its Fourier series as follows,
\be \label{period-velocity}
\Bv= \sum_{n \in \mathbb{Z}} v_n^r(r) e^{in z} \Be_r + v_n^\theta(r) e^{inz} \Be_\theta + v_n^z(r) e^{inz} \Be_z,
\ee
where
\be \nonumber
(v_n^r, v_n^\theta, v_n^z) (r) = \frac{1}{2\pi} \int_{-\pi}^{\pi} (v^r, v^\theta, v^z)(r, z) e^{-inz} \, dz.
\ee
Similarly, the external force $\BF$ and the vorticity $\Bo$ can be  written as
\be \label{period-force}
\BF = \sum_{n \in \mathbb{Z}} F_n^r(r) e^{inz} \Be_r + F_n^\theta(r) e^{inz} \Be_\theta + F_n^z (r) e^{inz} \Be_z,
\ee
and
\be \nonumber
\Bo = \sum_{n \in \mathbb{Z}} \omega_n^r (r) e^{inz} \Be_r + \omega_n^\theta (r) e^{inz} \Be_\theta + \omega_n^z (r)e^{inz} \Be_z,
\ee
respectively.

Hence the equation \eqref{ap-40-5} becomes
\be \nonumber \ba
& in \bar{U}(r)\omega_n^\theta  - (\mathcal{L} - n^2) \omega_n^\theta + \frac{d}{dr} \sum_{m \in \mathbb{Z}} v_m^r \omega_{n-m}^\theta + in \sum_{m\in \mathbb{Z}}
v_{n-m}^z \omega^\theta_m -\frac{in}{r} \sum_{m \in \mathbb{Z} } v_{n-m}^\theta v_m^\theta  \\
   = & in F_n^r - \frac{d}{dr}F_n^z .
\ea
\ee
In terms of the Fourier coefficients, the boundary conditions and the flux constraint \eqref{BC-1} can be written as
\be \label{BC-2}
v_n^r(1) = 0, \ \ \ \omega_n^\theta(1) = \alpha v_n^z (1) , \ \ \ \ \ \int_0^1 r v_n^z (r) \, dr =0\quad \text{for}\,\ n \in \mathbb{Z}.
\ee

Define
\be \label{streamdefinition}
\psi_n (r) = \left\{ \ba
& - \frac{1}{r}   \int_0^r s v_n^z (s) \, ds, \ \ \ \ &\mbox{if} \ \ n = 0, \\
&- i \frac{v_n^r}{n} ,\ \ \ \ \ \ \ &\mbox{if}\ \ n \neq 0.      \ea   \right.
\ee
Due to the divergence free property of $\Bv$, it holds that
\be \nonumber
\frac{d}{dr} v_n^r + in v_n^z + \frac{v_n^r}{r} = 0.
\ee
Hence,
\be \nonumber
v_n^r = in \psi_n, \ \ \ \ v_n^z = -\frac{1}{r} \frac{d(r \psi_n)}{dr},\ \  \ \ \mbox{and}\ \ \ \ \omega_n^\theta = (\mathcal{L} - n^2)\psi_n .
\ee
Therefore, for every $n \in \mathbb{Z}$, one gets a fourth order equation for $\psi_n$,
\be \nonumber \ba
& i n \bar{U}(r) (\mathcal{L} - n^2) \psi_n - (\mathcal{L} -n^2)^2 \psi_n  \\
&\ \  = inF_n^r - \frac{d}{dr}F_n^z - \frac{d}{dr} \sum_{m \in \mathbb{Z} } v_m^r \omega_{n-m}^\theta - in \sum_{m \in \mathbb{Z}} v_{n-m}^z \omega_m^\theta + \frac{in}{r} \sum_{m \in \mathbb{Z}} v_{n-m}^\theta v_m^\theta  . \\
 \ea
\ee

 Next, we derive the boundary conditions for $\psi_n$. As discussed in \cite{Liu-Wang}, in order to get classical axisymmetric solutions of the Navier-Stokes system, some compatibility conditions at the axis should be imposed. Assume that $\Bv$ and the vorticity $\Bo$ are continuous, $v^r_n(0, z)$ and $\omega^\theta_n(0, z)$ should vanish. This implies
\be \nonumber
in  \psi_n (0 ) = (\mathcal{L} - n^2 )\psi_n (0) = 0.
\ee
 Hence, for every $n \neq 0$,  the following compatibility conditions hold at the axis,
\be \nonumber
\psi_n (0) = \mathcal{L} \psi_n (0) = 0.
\ee
If $n = 0$, then one has
\be \nonumber
\psi_0(0) = \lim_{r \rightarrow 0+} -\frac1r \int_0^r s v_0^z (s) \, ds = 0 \ \ \ \mbox{and}\ \ \ \mathcal{L} \psi_0(0) =0 .
\ee
On the other hand, it follows from \eqref{BC-2} that
\be \nonumber
\psi_0(1) = (r \psi_0 (r))|_{r=1} = -  \int_0^1 r v_0^z(r) \, dr = 0.
\ee
Furthermore, if $n \neq 0$, it also holds that
\be \nonumber
\psi_n (1) = - i \frac{v_n^r(1)}{n}  = 0.
\ee

Moreover, according to the Navier boundary conditions in \eqref{BC-1} for $\Bv$, one has
\be \nonumber
 \omega^\theta_n (1)  = \alpha v^z_n (1).
\ee
This is equivalent to
\be \nonumber
 ( \mathcal{L} - n^2  ) \psi_n  = -  \alpha \frac{1}{r} \frac{d}{dr} (r \psi_n) \ \ \ \ \mbox{at}\ \ r=1 .
\ee
Thus one has
\be \label{2-0-4-5}
\mathcal{L} \psi_n (1) = - \alpha \frac{d }{d  r} \psi_n (1) .
\ee

In summary,  the governing equation and the boundary conditions for each $\psi_n$ are  as follows,
\be \label{stream-mode}
\left\{
\ba  & i n \bar{U}(r) (\mathcal{L} - n^2) \psi_n - (\mathcal{L} -n^2)^2 \psi_n  \\
&\ \ \  = inF_n^r - \frac{d}{dr}F_n^z - \frac{d}{dr} \sum_{m \in \mathbb{Z} } v_m^r \omega_{n-m}^\theta - in \sum_{m \in \mathbb{Z}} v_{n-m}^z \omega_m^\theta + \frac{in}{r} \sum_{m \in \mathbb{Z}} v_{n-m}^\theta v_m^\theta ,\\
& \psi_n (0)= \psi_n (1) = \mathcal{L} \psi_n (0) = 0, \ \ \  \mathcal{L}\psi_n(1)= - \alpha \frac{d}{dr}\psi_n(1).
\ea
\right.
 \ee

 Next, let us turn to the problem for $v^\theta_n$.  Assume that $\Bv$ is continuous, $v^\theta$ should vanish at the axis. This implies that $v^\theta (0, z) = 0$. Hence the problem for $v^\theta_n $ can be written as
\be \label{swirl-mode}
\left\{   \ba
& in  \bar U(r) v_n^\theta  - (\mathcal{L} - n^2) v_n^\theta + \sum_{m \in \mathbb{Z}} v_m^r \frac{d}{dr}v_{n-m}^\theta
 + \sum_{m \in \mathbb{Z}} im v_{n-m}^z v_m^\theta - \sum_{m \in \mathbb{Z}} \frac{v_{n-m}^\theta }{r} v_m^r =  F^\theta_n, \\
& v^\theta_n (0) = 0, \ \ \ \ \  \ \ \ \frac{d v^\theta_n}{d r}  (1) = (1 -  \alpha ) v^\theta_n (1).
\ea
\right.
\ee

\subsection{Linearized problem}\label{sec-stream-linear}
To get the existence of nonlinear problem \eqref{ap-40-1}-\eqref{BC-swirl}, we first investigate the  linearized system around Poiseuille flows,
\be \nonumber \left\{ \ba
&\bBU \cdot \nabla \Bv + \Bv \cdot \nabla \bBU - \Delta \Bv + \nabla P = \BF \ \ \ \mbox{in}\ \Omega, \\
& {\rm div}~\Bv = 0\ \ \ \mbox{in}\ \Omega,
\ea
\right.
\ee
supplemented with the boundary conditions and flux constraint
\begin{equation}\nonumber
\Bv\cdot \Bn=0,\quad \ 2 \Bn \cdot D(\Bv) \cdot \boldsymbol{\tau}  + \alpha \Bu\cdot \boldsymbol{\tau}=0 \quad \text{on}\,\, \partial \Omega,
\ \ \ \ \int_{\Sigma} \Bv \cdot \Bn \, dS = 0.
\end{equation}
The system for axisymmetric solutions can be written in the following form,
\be \label{2-0-1-1}
\left\{
\ba
& \bar U(r)  \frac{\partial v^r}{\partial z} + \frac{\partial P}{\partial r} -\left[ \frac{1}{r} \frac{\partial}{\partial r}\left(
r \frac{\partial v^r}{\partial r} \right) + \frac{\partial^2 v^r}{\partial z^2} - \frac{v^r}{r^2} \right] = F^r  \ \ \ \mbox{in}\ D ,\\
& v^r \frac{\partial \bar U }{\partial r} + \bar U(r)  \frac{\partial v^z}{\partial z} + \frac{\partial P}{\partial z}
- \left[ \frac{1}{r} \frac{\partial }{\partial r} \left( r \frac{\partial v^z}{\partial r}\right) + \frac{\partial^2 v^z}{\partial z^2} \right] = F^z \ \ \mbox{in}\ D ,                     \\
& \partial_r v^r + \partial_z v^z + \frac{v^r}{r} =0\ \ \ \mbox{in}\ D, \\
& v^r(1, z)  = 0,\ \ \ \ \partial_z v^r (1, z) - \partial_r v^z(1, z) = \alpha v^z(1, z), \ \ \ \quad  \int_0^1 r v^z(r, z)\, dr = 0
\ea \right.
\ee
and
\be \label{vswirl}  \left\{
\ba
&\bar U(r)  \partial_z  v^\theta - \left[ \frac{1}{r} \frac{\partial }{\partial r} \left( r \frac{\partial v^\theta}{\partial r}\right) + \frac{\partial^2 v^\theta}{\partial z^2} - \frac{v^\theta}{r^2} \right] =  F^\theta \ \ \ \mbox{in}\ \ D, \\
& v^\theta(0, z) = 0, \ \ \ \ \ \frac{\partial v^\theta}{\partial r}(1, z) = (1 - \alpha) v^\theta (1, z).
\ea  \right.
\ee

Similar to the case for the nonlinear problem \eqref{ap-40-1}, for each $n$-th mode, one can introduce the stream function $\psi_n$ for \eqref{2-0-1-1}. As same as what has been done in Subsection \ref{sec-stream},  $\psi_n $ satisfies the following linear problem,
\be \label{2-0-8}
\left\{
\ba  & i n \bar{U}(r) (\mathcal{L} - n^2) \psi_n - (\mathcal{L} -n^2)^2 \psi_n = f_n = inF_n^r - \frac{d}{dr}F_n^z \ \ \ \ \mbox{in}\ D,\\
& \psi_n (0)= \psi_n (1) = \mathcal{L} \psi_n (0) = 0, \ \ \  \mathcal{L}\psi_n(1)= - \alpha \frac{d}{dr}\psi_n(1).
\ea
\right.
\ee
The $n$-th mode $v_n^\theta$ of the swirl velocity satisfies
\be \label{swirl-system}
\left\{ \ba
& in \bar{U}(r) v_n^\theta - (\mathcal{L} - n^2) v_n^\theta = F_n^\theta\ \ \ \mbox{in}\ D, \\
& v_n^\theta(0)= 0, \ \ \ \ \frac{d}{dr}v_n^\theta(1) = (1 - \alpha) v_n^\theta(1).
 \ea \right.
\ee

\section{Uniform a priori estimates for the stream function }\label{Linear}
In this section, we focus on the linear perturbation system \eqref{2-0-8}. Some  a priori estimates for each  $\psi_n $ and $\Bv^*= \dis\sum_{n \in \mathbb{Z}} v_n^r e^{inz}\Be_r + v_n^z e^{inz} \Be_z
= \sum_{n \in \mathbb{Z}} i n \psi_n e^{inz} \Be_r - \frac{1}{r} \frac{d}{dr} (r \psi_n)e^{inz}  \Be_z $ are established. These estimates are uniform with respect to $\Phi$ and $\alpha$. Regarding the existence of solutions to \eqref{2-0-8}, we put the proof in the appendix, since it is quite similar to Section 4 of \cite{WX1}.

\subsection{ Basic A priori estimates for $\Bv^*$ }\label{sec-apri}
In this subsection, we prove some basic  uniform a priori estimates for $\psi_n$ and $\Bv^*$,  when  $\Phi$ is not large.
\begin{pro}\label{lemapri1-0}
Let $\psi_n$ be a smooth solution of the problem \eqref{2-0-8}, then the corresponding velocity $\Bv^*$ satisfies
\be \label{2-1-0}
\|\Bv^* \|_{H^2(\Omega)} \leq C (1 + \Phi^{\frac32}) \|\BF^*\|_{L^2(\Omega)},
\ee
where $\BF^* = F^r \Be_r + F^z \Be_z$, and  $C$ is a uniform constant independent of $\BF^*$, $\Phi$, and $\alpha$.
\end{pro}

\begin{pf}
 Multiplying the equation in \eqref{2-0-8} by $r \overline{\psi_n} $ (here and later on $\overline{\psi_n}$ denotes the complex conjugate of $\psi_n$) and integrating the resulting equation over $[0, 1]$ yield
\be \label{2-1-1}
\int_0^1 \left[ i n \bar U(r) (\mathcal{L} - n^2) \psi_n - (\mathcal{L} - n^2)^2  \psi_n      \right] \overline{\psi_n}  r  \,  dr
= \int_0^1  f_n \overline{\psi_n}  r \, dr.
\ee
For the first  term on the left hand  of \eqref{2-1-1}, it follows from integration by parts and the homogeneous boundary conditions for $\psi_n$ that
\be \la{2-1-3} \ba
& \int_0^1 i n  \bar U(r)  ( \mathcal{L}  - n^2) \psi_n   \overline{\psi_n }  r \, dr \\
= \,\,& i n  \int_0^1 \bar{U}(r)  \frac{d}{dr} \left(  \frac1r \frac{d}{dr} (r\psi_n )     \right)  r \overline{\psi_n } \, dr
-  i n^3 \int_0^1 \bar U(r) |\psi_n |^2  r \, dr \\
= \,\,& i n \frac{4\Phi}{\pi} \frac{\alpha}{\alpha + 4}  \int_0^1 \frac{d}{dr} (r \psi_n )  r \overline{\psi_n }\, dr
- i n \int_0^1 \frac{\bar{U}(r)}{r} \left| \frac{d}{dr} (r \psi_n )   \right|^2 \, dr - i n^3 \int_0^1 \bar{U}(r) |\psi_n |^2  r \, dr.
\ea
\ee
While for the second term on the left hand  of \eqref{2-1-1}, one has
\be \la{2-1-2} \ba
& \int_0^1 ( \mathcal{L}  - n^2)^2 \psi_n  \overline{\psi_n }  r \, dr \\
 =\,\, & \int_0^1 \frac{d}{dr} \left(  \frac1r \frac{d}{dr} ( r \mathcal{L} \psi_n )      \right)  \overline{\psi_n }  r \, dr
- 2 n^2 \int_0^1 \frac{d}{dr} \left( \frac1r \frac{d}{dr} (r\psi_n )    \right)  \overline{\psi_n }  r \, dr
+ n^4 \int_0^1 |\psi_n|^2  r\, dr \\
 =\,\,& \int_0^1 | \mathcal{L} \psi_n  |^2  r \, dr - \mathcal{L} \psi_n (1) \frac{d}{dr} ( r \overline{\psi_n} )(1)  + 2 n^2 \int_0^1 \left|  \frac{d}{dr} (r \psi_n )  \right|^2  \frac1r \, dr
+ n^4 \int_0^1 |\psi_n |^2  r \, dr\\
= \,\, &  \int_0^1 | \mathcal{L} \psi_n |^2  r \, dr + \alpha \left| \frac{d}{dr}(r \psi_n )(1) \right|^2   + 2 n^2 \int_0^1 \left|  \frac{d}{dr} (r \psi_n)  \right|^2  \frac1r \, dr
+ n^4 \int_0^1 |\psi_n|^2  r \, dr.
\ea \ee

From now on, we denote $\Im g$ and $\Re g$ by the imaginary and real part of  $g$ (a function or a number), respectively. Note that the homogeneous boundary conditions for $\psi_n $ imply
\be \nonumber
\Re \int_0^1 \frac{d}{dr} (r \psi_n ) r \overline{\psi_n } \, dr  = 0.
\ee
It follows from  \eqref{2-1-1}-\eqref{2-1-2} that
\be \la{2-1-5} \ba
& \int_0^1 | \mathcal{L} \psi_n |^2  r \, dr
+ 2 n^2 \int_0^1 \left|  \frac{d}{dr} (r \psi_n )   \right|^2  \frac1r \, dr
+ n^4 \int_0^1 |\psi_n |^2  r \, dr + \alpha \left| \frac{d}{dr}(r \psi_n )(1) \right|^2  \\
 =\,\, &- \Re \int_0^1 f_n  \overline{\psi_n }  r \, dr - \frac{4 \Phi}{\pi} \frac{\alpha}{\alpha + 4 } n  \Im \int_0^1  \left[ \frac{d}{dr} (r \psi_n )  r \overline{\psi_n }      \right] \, dr
\ea \ee
and
\be \la{2-1-7}
n \int_0^1 \frac{\bar{U}(r) }{ r } \left| \frac{d}{dr} ( r \psi_n )   \right|^2 \, dr
+ n^3 \int_0^1  \bar{U}(r) |\psi_n |^2  r\, dr
= - \Im \int_0^1 f_n \overline{\psi_n  }  r \, dr .
\ee

If $n=0$, it follows from \eqref{2-1-5} that
\be \nonumber
\int_0^1 |\mathcal{L} \psi_0|^2 r \, dr +  \alpha \left|\frac{d}{dr}(r \psi_0) (1)  \right|^2
\leq \left( \int_0^1 |F_0^z|^2 r \, dr \right)^{\frac12} \left( \int_0^1 \left|  \frac{d}{dr}(r \psi_0)  \right|^2 \frac1r \, dr     \right)^{\frac12} .
\ee
By Lemma \ref{lemma1} and Young's inequality, one has
\be \label{2-1-7-0}
\int_0^1 |\mathcal{L} \psi_0|^2 r \, dr +  \alpha \left|\frac{d}{dr}(r \psi_0) (1)  \right|^2
\leq C \int_0^1 |F_0^z|^2 r \, dr .
\ee

Since $\alpha \geq  0$, for every $0\leq r < 1$,
\be \label{2-1-7-1}
\frac{\frac{2\Phi}{\pi} (1 - r^2) }{\bar{U} (r) }
= \frac{2(1-r^2)}{\frac{4 + 2\alpha}{4 + \alpha} \left( 1 - \frac{2\alpha}{4+ 2\alpha } r^2 \right)   }
\leq \frac{2(1-r^2)}{\frac{4 + 2\alpha}{4 + \alpha} \left( 1 - r^2 \right)   } \leq 2.
\ee
Therefore, if $n\neq 0$, then it follows from \eqref{2-1-7} and \eqref{2-1-7-1} that one has
\be \label{2-1-8}
\Phi |n | \int_0^1  \left| \frac{d}{dr} (r \psi_n  )   \right|^2  \frac{1-r^2}{r} \, dr +
\Phi |n|^3 \int_0^1  \left| \psi_n  \right|^2  r(1-r^2) \, dr
\leq C \left|  \int_0^1 f_n \overline{\psi_n} r \, dr \right| .
\ee

Next, let us estimate the second term on the right hand of \eqref{2-1-5}. By Lemmas \ref{lemma1} and  \ref{weightinequality},
\be \label{2-1-9}
\ba
& \left|  \frac{4\Phi }{\pi} \frac{\alpha }{\alpha + 4}  n \int_0^1 \frac{d}{dr} ( r\psi_n) r \overline{\psi_n} \, dr \right| \\
\leq & C \Phi |n| \left(  \int_0^1 \left|\frac{d}{dr}(r \psi_n)  \right|^2 \frac1r \, dr         \right)^{\frac12} \left( \int_0^1 |\psi_n|^2 r \, dr  \right)^{\frac12} \\
\leq  & C \Phi |n| \left(\int_0^1 \left| \frac{d}{dr}(r \psi_n) \right|^2 \frac{1-r^2}{r} \, dr \right)^{\frac13} \left( \int_0^1 |\mathcal{L} \psi_n |^2 r \,dr  \right)^{\frac16} \\
 &\ \ \ \ \ \ \ \ \ \cdot \left( \int_0^1 \left|  \frac{d}{dr} ( r \psi_n) \right|^2 \frac1r \, dr   \right)^{\frac16} \left( \int_0^1 |\psi_n|^2 (1 - r^2) r\, dr  \right)^{\frac13}\\
\leq & \frac14 \int_0^1 |\mathcal{L} \psi_n|^2 r \,dr + \frac14 n^2 \int_0^1 \left| \frac{d}{dr} ( r \psi_n) \right|^2 \frac1r \, dr \\
&\ \ \ \ \ \ \ \
+ C \Phi^{\frac32} \int_0^1 \left|\frac{d}{dr}(r \psi_n)   \right|^2 \frac{1-r^2}{r} \, dr + C \Phi^{\frac32} n^2 \int_0^1 |\psi_n|^2 (1 - r^2) r \, dr .
\ea
\ee
Substituting \eqref{2-1-8} and \eqref{2-1-9} into \eqref{2-1-5} yields
\be \nonumber \ba
&\int_0^1 |\mathcal{L} \psi_n|^2 r \, dr + n^2 \int_0^1 \left|  \frac{d}{dr} (r \psi_n) \right|^2 \frac1r \, dr +  n^4 \int_0^1 |\psi_n|^2 r \, dr
 +  \alpha \left| \frac{d}{dr} ( r\psi_n)(1)  \right|^2 \\
\leq & C (1 + \Phi^{\frac12}  |n|^{-1} ) \left|  \int_0^1 f_n \overline{\psi_n} r\, dr  \right|\\
\leq & C (1  + \Phi^{\frac12}|n|^{-1} ) |n|^{-1} \left( \int_0^1(  |F_n^r|^2 + |F_n^z|^2 ) r\, dr \right)^{\frac12} \left( \int_0^1 n^2 \left|  \frac{d}{dr} (r \psi_n) \right|^2 \frac1r  +  n^4  |\psi_n|^2 r \, dr \right)^{\frac12}.
\ea \ee
This implies
\be \label{2-1-11-new}
 \ba
&\int_0^1 |\mathcal{L} \psi_n|^2 r \, dr + n^2 \int_0^1 \left|  \frac{d}{dr} (r \psi_n) \right|^2 \frac1r \, dr +  n^4 \int_0^1 |\psi_n|^2 r \, dr
+  \alpha \left|  \frac{d}{dr}(r \psi_n)(1)  \right|^2 \\
\leq & C(1  + \Phi |n|^{-2} )|n|^{-2}  \int_0^1 (|F_n^r|^2 + |F_n^z|^2 ) r \, dr.
\ea
\ee
Hence one has
\be \label{2-1-15} \ba
\|v^r_n \|_{L^2(\Omega)}^2 + \|v^z_n \|_{L^2(\Omega)}^2
& \leq C  \int_0^1 \left[    n^2 | \psi_n |^2 r + \left|\frac{d }{ d r}(r \psi_n )     \right|^2 \frac1r  \right] \, dr  \\
& \leq C (1 + \Phi) \|\BF^*_n \|_{L^2(\Omega)}^2.
\ea
\ee
Moreover, since ${\rm curl}~\Bv_n^* = \omega_n^\theta e^{inz} \Be_\theta $,
it holds that
\be \label{2-1-16-0} \ba
\int_{\Omega} |\nabla \Bv^*_n|^2 \, dV & = \int_{\Omega} -\Delta \Bv_n^* \cdot \Bv_n^* \, dV + \int_{\partial \Omega}
\frac{\partial \Bv_n^*}{ \partial \Bn } \cdot \Bv_n^* \, dS  \\
& = \int_{\Omega} {\rm curl}~(\omega_n^\theta \Be_\theta) \cdot \Bv_n^* \, dV + \int_{\partial \Omega} \frac{\partial v_n^z }{\partial r} \cdot v_n^z \, dS \\
& = \int_{\Omega} |\omega_n^\theta|^2 \, dV + \alpha \int_{\partial \Omega} |v_n^z|^2 \, dS - \alpha \int_{\partial \Omega} |v_n^z|^2 \, dS \\
& = \int_{\Omega} |\omega_n^\theta|^2 \, dV.
\ea
\ee
Hence one has
\be \label{2-1-16} \ba
\| \Bv^*_n  \|_{H^1(\Omega)}^2 & = \|\Bv^*_n \|_{L^2(\Omega)}^2 +  \|\omega^\theta_n \|_{L^2(\Omega)}^2  =  \|\Bv^*_n  \|_{L^2(\Omega)}^2 +  2\pi  \int_0^1 | (\mathcal{L} - n^2 ) \psi_n|^2 r \, dr \\
& \leq C (1 + \Phi ) \|\BF^*_n \|_{L^2(\Omega)}^2.
\ea
\ee
Summing \eqref{2-1-16} with respect to $n$ yields
\be \label{2-1-16-1}
\|\Bv^* \|_{H^1(\Omega)} \leq C (1  + \Phi^{\frac12} )\|\BF^*\|_{L^2(\Omega)}.
\ee

Similarly,  multiplying \eqref{2-1-11-new} by $n^2$ gives
\be \label{2-1-17}
\ba
& n^2 \int_0^1 |\mathcal{L} \psi_n|^2 r \, dr + n^4 \int_0^1 \left|  \frac{d}{dr} (r \psi_n) \right|^2 \frac1r \, dr +  n^6 \int_0^1 |\psi_n|^2 r \, dr \\
\leq & C (1 + \Phi |n|^{-2} )  \int_0^1(  |F_n^r|^2 + |F_n^z|^2 ) r \, dr                .
\ea
\ee
Similar to the proof of \eqref{2-1-16-1}, it follows from \eqref{2-1-17} that one has
\be \label{2-1-19}
\|\partial_z \Bv^* \|_{H^1(\Omega)} \leq C (1 + \Phi^{\frac12} ) \|\BF^*\|_{L^2(\Omega)}.
\ee

As shown in Proposition \ref{back}  , $\Bv^*$ satisfies the equation
\be \label{2-1-20}
\left\{
\ba
& \bar{U} \partial_z \Bv^* + v^r \partial_r \bBU - \Delta \Bv^* + \nabla P = \BF^* \ \ \ \mbox{in}\ \Omega, \\
& {\rm div}~\Bv^* = 0 \ \ \ \mbox{in}\ \Omega.
\ea
\right.
\ee
According to the regularity theory for Stokes equations (\cite[Lemma VI.1.2]{Galdi}) and the trace theorem for axisymmetric functions, one has
\be \label{2-1-21} \ba
\|\Bv^*\|_{H^2 (\Omega)} & \leq C \left( \|\BF^*\|_{L^2(\Omega)} + \Phi \|\partial_z \Bv^*\|_{L^2(\Omega)} +  \Phi \|v^r\|_{L^2(\Omega)} +  \|\Bv^*\|_{H^1(\Omega)} + \|\Bv^* \|_{H^{\frac32}(\partial \Omega) } \right) \\
& \leq C (1 + \Phi^{\frac32} ) \|\BF^*\|_{L^2(\Omega)} + C \|\partial_z \Bv^*\|_{H^1(\Omega)} \\
& \leq C (1 + \Phi^{\frac32} ) \|\BF^*\|_{L^2(\Omega)} .
\ea
\ee
This finishes the proof of Proposition \ref{lemapri1-0}.
\end{pf}


\subsection{Uniform estimate for the zero mode}
In this subsection, we give the uniform estimate for the zero mode of the solutions of \eqref{2-0-8} with respect to $\Phi$ and  $\alpha$. First, it follows from \eqref{2-1-7-0} that the following estimate holds.
\begin{pro}\label{Bpropcase1}
Let $\psi_0 $ be a smooth solution of the problem
\eqref{2-0-8} associated with $n=0$, then it holds that
\be \label{B-1}
\int_0^1 |\mathcal{L} \psi_0 |^2 r \, dr  \leq C  \int_0^1 |\BF^*_0|^2 r \, dr.
\ee
\end{pro}

In fact, the estimate \eqref{B-1} gives $H^2$ estimate for $\Bv_0^*$.

\begin{pro}\label{Bpropcase1-1} The corresponding velocity $\Bv^*_0 = - \frac{1}{r} \frac{d}{dr}(r \psi_0) \Be_z $  satisfies
\be \label{B6}
\|\Bv^*_{0} \|_{H^2(\Omega)} \leq C \|\BF^*_{0}\|_{L^2(\Omega)},
\ee
where $C$ is a uniform constant independent of $\Phi$, $\alpha$, and $\BF$.
\end{pro}

\begin{proof}
It follows from the similar proof of Proposition \ref{back} that  $\Bv^*_{0}$ is a strong solution to the following Stokes equations,
\be \label{stokes-low} \left\{ \ba
& - \Delta \Bv^*_{0} + \nabla P = \BF^*_{0}  \ \ \ \mbox{in}\ \Omega, \\
& {\rm div}~\Bv^*_{0} = 0\ \ \ \ \mbox{in}\ \Omega.
\ea \right. \ee
According to the regularity theory for the Stokes equations (\cite[Lemma VI.1.2]{Galdi}) and the trace theorem for axisymmetric functions, one has
\be \label{B7}  \ba
\|\Bv^*_0 \|_{H^2(\Omega)} & \leq  C \|\BF^*_{0}\|_{L^2(\Omega)} + C \|\Bv^*_{0}\|_{H^1(\Omega)} + C \| \Bv^*_{0}\|_{H^{\frac32}(\partial \Omega)}\\
& \leq  C \|\BF^*_{0}\|_{L^2(\Omega)} + C \|\Bv^*_{0}\|_{H^1(\Omega)} + C \|\partial_z \Bv^*_{0}\|_{H^1(\Omega)}\\
&  \leq C \|\BF^*_{0}\|_{L^2(\Omega)} + C \|\Bv^*_{0}\|_{H^1(\Omega)}. \ea
\ee
Herein, by virtue of the estimate \eqref{B-1} and Lemma \ref{lemma1}, following the same estimate as \eqref{2-1-16-0},
  one has
\be \label{B10-0}
\|\Bv^*_0\|_{L^2(\Omega)} = \|v^z_{0} \|_{L^2(\Omega)} \leq C \left(  \int_0^1 \left|  \frac{d}{dr}(r \psi_0  ) \right|^2 \frac1r \, dr    \right)^{\frac12}
\leq C \| \BF^*_{0} \|_{L^2(\Omega)}
\ee
and
\be \label{B10} \ba
 \|\Bv^*_{0} \|_{H^1 (\Omega)}  \leq &   \|\Bo^\theta_{0}\|_{L^2(\Omega)} + \|\Bv^*_0\|_{L^2(\Omega)} \leq  C \left(  \int_0^1  |  \mathcal{L} \psi_0 |^2 r \, dr  \right)^{\frac12} + C \|v^z_0\|_{L^2(\Omega)}  \\
\leq &   C \| \BF^*_0 \|_{L^2(\Omega)}.
\ea \ee
 Taking \eqref{B10-0}-\eqref{B10} into \eqref{B7} gives \eqref{B6} so that the proof of Proposition \ref{Bpropcase1-1} is completed.
\end{proof}

\subsection{Uniform estimate for the case with large flux and high frequency}
In this subsection, we give the uniform estimate for the solutions of \eqref{2-0-8} with respect to $\Phi$ and $\alpha$ when the flux is large and the frequency is high.
\begin{pro}\label{Bpropcase2}
Assume that $0<\epsilon_1 <1$ and $|n | \geq \epsilon_1 \sqrt{\Phi} \geq 1 $. Let $\psi_n$ be a  smooth solution to the problem \eqref{2-0-8}, then one has
\be \label{highf1} \ba
 \int_0^1 |\mathcal{L} \psi_n |^2 r + n^2  \left| \frac{d}{dr} ( r \psi_n)    \right|^2 \frac{1}{r}
+ n^4  |\psi_n |^2 r \, dr
\leq &  C(\epsilon_1) |n|^{-2}  \int_0^1 |\BF^*_n|^2 r \, dr\\
 \leq & C (\epsilon_1) (\Phi |n|)^{-\frac43}  \int_0^1 |\BF^*_n|^2 r \, dr ,
\ea
\ee
and
\be \label{highf2}
 |n| \int_0^1 \frac{\bar{U}(r)}{r}  \left| \frac{d}{dr} ( r \psi_n ) \right|^2 \, dr +  |n|^3 \int_0^1 \bar{U}(r) |\psi_n |^2 r \, dr
\leq C(\epsilon_1)  |n|^{-2}  \int_0^1 | \BF^*_n |^2 r \, dr .
\ee
\end{pro}

\begin{proof} According to \eqref{2-1-11-new}, one has
\be \label{B11}
\ba
  \int_0^1 |\mathcal{L} \psi_n |^2 r  + n^2  \left| \frac{d}{dr} ( r \psi_n)    \right|^2 \frac{1}{r}
+ n^4 |\psi_n |^2 r \, dr
\leq & C (1 + \Phi |n|^{-2} ) |n|^{-2} \int_0^1 |\BF^*_n|^2 r \, dr  \\
 \leq & C (\epsilon_1) |n|^{-2} \int_0^1 |\BF^*_n|^2 r \, dr.
\ea
\ee
Taking \eqref{B11} into \eqref{2-1-7} gives
\be \label{B12} \ba
&  |n| \int_0^1 \frac{ \bar{U}(r) }{r} \left| \frac{d}{dr} ( r \psi_n ) \right|^2 \, dr +   |n|^3 \int_0^1 \bar{U}(r) |\psi_n |^2 r \, dr \\
\leq & C \left( \int_0^1 |\BF^*_n |^2 r \, dr  \right)^{\frac12} \left( \int_0^1 \left| \frac{d}{dr}( r \psi_n )  \right|^2 \frac1r + n^2 |\psi_n |^2 r \, dr             \right)^{\frac12} \\
\leq & C(\epsilon_1)  |n|^{-2} \int_0^1 |\BF^*_n |^2 r \, dr .
\ea
\ee
This completes the proof of Proposition \ref{Bpropcase2}.
\end{proof}

Let $\psi_{high}$ denote the high-frequency part of stream function, i.e.,
\be \nonumber
\psi_{high}(r, z) = \sum_{|n| \geq \epsilon_1 \sqrt{\Phi}} \psi_n(r) e^{inz}.
\ee
Define
\be \nonumber
v^r_{high} = \partial_z \psi_{high},\ \ \ \ v^z_{high} = - \frac{\partial_r ( r \psi_{high} )}{r}, \ \ \ \ \Bv^*_{high} = v^r_{high}\Be_r + v^z_{high}\Be_z.
\ee
One can define $F^r_{high}, F^z_{high}, \BF^*_{high}$, $\Bo_{high}^\theta $ in a similar way.

\begin{pro}\label{Bpropcase2-1} The solution $\Bv^*$  satisfies
\be \label{B16}
\|\Bv^*_{high} \|_{H^2(\Omega)} \leq C(\epsilon_1) \Phi^{\frac14} \|\BF^*_{high} \|_{L^2(\Omega)}
\ee
and
\be \label{B16-1}
\|\Bv^*_{high}\|_{H^{\frac53} (\Omega)} \leq C (\epsilon_1) \|\BF^*_{high}\|_{L^2(\Omega)},
\ee
where $C$ is a uniform constant independent of $\Phi$, $\alpha$, and $\BF$.
\end{pro}

\begin{proof}In fact, $\Bv^*_{high}$ is a strong solution to the following Stokes equations,
\be \label{stokes-high} \left\{  \ba
& - \Delta \Bv^*_{high} + \nabla P = \BF^*_{high} - \bar{U} \partial_z \Bv^*_{high} - v^r_{high} \partial_r \bBU \ \ \ \mbox{in}\ \Omega, \\
& {\rm div}~\Bv^*_{high} = 0\ \ \ \ \mbox{in}\ \Omega.
\ea \right.
\ee
By virtue of the regularity theory for Stokes equations (\cite[Lemma VI.1.2]{Galdi}) again, one has
\be \label{B17} \ba
\|\Bv^*_{high}\|_{H^2(\Omega)}
\leq &  C \|\BF^*_{high}\|_{L^2(\Omega)}
 + C  \| \bar{U}(r)  \partial_z \Bv^*_{high} \|_{L^2(\Omega)} + C \Phi \|v^r_{high}\|_{L^2(\Omega)}\\
&\ \ \ \  + C \|\Bv^*_{high}\|_{H^1(\Omega)} + C \|  \Bv^*_{high} \|_{H^{\frac32}(\partial \Omega)}\\
\leq &  \|\BF^*_{high}\|_{L^2(\Omega)}
+ C  \| \bar{U}(r)  \partial_z \Bv^*_{high} \|_{L^2(\Omega)} + C \Phi \|v^r_{high}\|_{L^2(\Omega)}\\
&\ \ \ \  + C \|\Bv^*_{high}\|_{H^1(\Omega)} + C \|  \partial_z \Bv^*_{high} \|_{H^1 ( \Omega)} .
\ea
\ee
It follows from \eqref{highf1} and \eqref{highf2} that
\be \label{B18} \ba
&  \| \bar{U}(r)  \partial_z \Bv^*_{high}\|_{L^2(\Omega)}  \\
\leq & C \left\{ \sum_{|n| \geq \epsilon_1 \sqrt{\Phi} } \Phi  \left[  \int_0^1 n^4 \bar{U}(r) |\psi_n |^2 r  + n^2 \frac{\bar{U}(r)}{r} \left|  \frac{d}{dr} ( r \psi_n )  \right|^2 \, dr   \right]    \right\}^{\frac12} \\
\leq & C \left( \sum_{|n| \geq \epsilon_1 \sqrt{\Phi} } \Phi |n|^{-1} \int_0^1 |\BF^*_n |^2 r \, dr  \right)^{\frac12}\\
\leq & C (\epsilon_1) \Phi^{\frac14} \|\BF^*_{high}\|_{L^2(\Omega)}
\ea
\ee
and
\be \label{B19} \ba
\Phi \|v^r_{high}\|_{L^2(\Omega)}&  \leq C \Phi \left( \sum_{|n| \geq \epsilon_1 \sqrt{\Phi} } n^2 \int_0^1 |\psi_n |^2 r \, dr   \right)^{\frac12} \leq C(\epsilon_1)  \|\BF^*_{high}\|_{L^2(\Omega)}.
\ea
\ee
Furthermore, following the same proof as for \eqref{2-1-16-0}, one has
\be \label{B20}
\begin{aligned}
\| v^z_{high}  \|_{ L^2 (\Omega)} \leq & C \left(  \sum_{|n|\geq \epsilon_1 \sqrt{\Phi} }  \int_0^1 \left| \frac{d}{dr} ( r\psi_n )  \right|^2 \frac1r \, dr  \right)^{\frac12} \\
 \leq & C \left( \sum_{|n| \geq \epsilon_1 \sqrt{\Phi} }  |n|^{-4} \int_0^1 |\BF^*_n |^2 r \, dr  \right)^{\frac12} \leq  C \Phi^{-1}  \| \BF^*_{high}\|_{L^2(\Omega)}
 \end{aligned}
\ee
and
\be \label{B20-1} \ba
\|  \Bv^*_{high} \|_{H^1 (\Omega) } & \leq  \| \Bo^\theta_{high}\|_{L^2(\Omega)}  + \|\Bv^*_{high} \|_{L^2(\Omega)} \\
& \leq C \left( \sum_{|n| \geq \epsilon_1 \sqrt{\Phi} }  \int_0^1 | (\mathcal{L} - n^2) \psi_n |^2 r \, dr         \right)^{\frac12}  + C \Phi^{-1} \|\BF^*_{high} \|_{L^2(\Omega)}\\
& \leq C(\epsilon_1) \Phi^{-\frac12} \|\BF^*_{high} \|_{L^2(\Omega)} .
\ea
\ee
Hence it holds
\be \label{B20-2} \ba
\|  \partial_z \Bv^*_{high} \|_{H^1(\Omega) } & \leq   \| \partial_z \Bo^\theta_{high}\|_{L^2(\Omega)} +  \|\partial_z \Bv^*_{high}\|_{L^2(\Omega)} \\
& \leq C \left( \sum_{|n| \geq \epsilon_1 \sqrt{\Phi} }  \int_0^1 n^2  | (\mathcal{L} - n^2) \psi_n |^2 r \, dr         \right)^{\frac12} + C \Phi^{-\frac12} \|\BF^*_{high} \|_{L^2(\Omega)} \\
& \leq C(\epsilon_1) \|\BF^*_{high} \|_{L^2(\Omega)} .
\ea \ee
Substituting  \eqref{B18}-\eqref{B20-2} into \eqref{B17} gives \eqref{B16}. The estimate \eqref{B16-1} is the consequence of
 the interpolation between \eqref{B16} and \eqref{B20-1}.
 This finishes the proof of Proposition \ref{Bpropcase2-1}.
\end{proof}

\subsection{Uniform estimate for the case with large flux  and  intermediate frequency}\label{sec-intermediate}
In this subsection, we give the uniform estimate for the solutions of \eqref{2-0-8} with respect to $\Phi$ and $\alpha$,  when the flux is large and the frequency is intermediate. The analysis in this case is much more involved and  inspired by \cite{M}. To get the uniform estimates, we decompose the stream function into several parts. The first part comes from the solution with the slip boundary condition, the second part is the boundary layer profile, the third part and the fourth part are the irrotational solutions and the remainder term. When $|\alpha|$ is small compared to $\Phi |n|$, the boundary layer  is an exponential function, while when $|\alpha|$ is large, the boundary layer is close to an Airy function. Due to the differences of boundary layer functions, we divide the analysis into three subcases: $|\alpha|$ is small, large, and intermediate.

\subsubsection{The case with large flux, intermediate frequency, and small slip coefficient}
\begin{pro}\label{Bpropcase3} Assume that $\Phi \gg 1$.
There exist two small constants $\epsilon_1 \in (0, 1)$ and $\delta \in (0, 1)$,  such that as long as $1
\leq |n| \leq \epsilon_1 \sqrt{\Phi} $, and $4 + \alpha \leq \delta (\Phi |n|)^{\frac13}$,
the solution $\psi_n (r)$ to the problem \eqref{2-0-8} can be decomposed into four parts,
\be \label{case3-1}
\psi_n (r) = \psi_{s, n} (r) +  b\left[ \chi \psi_{BL, n}(r) + \psi_{e, n} (r) \right]  + a I_1(|n|r) .
\ee
The properties of these four parts are summarized as follows.

 $(1)$\ $\psi_{s, n} $ is a solution to the following problem
\be \label{slip}
\left\{ \ba   & i n \bar{U}(r) ( \mathcal{L} - n^2) \psi_{s, n}  - (\mathcal{L} - n^2)^2 \psi_{s, n} = f_n = - \frac{d}{dr} F^z_n + i n F^r_n, \ \ \ \mbox{in}\ (0, 1), \\
& \psi_{s, n} (0) = \psi_{s, n} (1) = \mathcal{L} \psi_{s, n} (0) = \mathcal{L} \psi_{s, n}  (1) = 0,
\ea      \right.
\ee
which satisfies the estimates
\be \label{case3-3} \ba
& \int_0^1 |\psi_{s, n} |^2 r   +
 \left|\frac{d}{dr}(r \psi_{s, n})  \right|^2  \frac{1}{r}
+ n^2 \left| \psi_{s ,n} \right|^2 r \, dr
\leq  C (\Phi |n|)^{- 2} (4 +  \alpha)^2  \int_0^1 |\BF^*_n |^2 r \, dr ,
\ea
\ee
\be \label{case3-4} \ba
& \int_0^1 | \mathcal{L} \psi_{s, n} |^2 r  + n^2  \left|  \frac{d}{dr} ( r \psi_{s, n} )\right|^2 \frac1r  + n^4
 | \psi_{s, n} |^2  r \, dr
 \leq  C ( \Phi |n|)^{-1} (4 +  \alpha) \int_0^1 |\BF^*_n |^2 r \, dr ,
 \ea
\ee
\be \label{case3-5} \ba
& \int_0^1 \left| \frac{d}{dr}( r \mathcal{L} \psi_{s, n}  )\right|^2 \frac1r  +
n^2   |\mathcal{L} \psi_{s, n} |^2 r
+\frac{n^4}{r}  \left|  \frac{d}{dr}( r \psi_{s, n} )       \right|^2   + n^6  |\psi_{s, n} |^2 r  \, dr
\leq  C \int_0^1 |\BF^*_n |^2 r\, dr ,
\ea
\ee
and
\begin{equation}\label{case3-5-1}
\left| \frac{d}{dr} ( r \psi_{s, n} ) (1) \right| \leq C ( \Phi |n|)^{-\frac34} (4 + \alpha)^{\frac34} \left( \int_0^1 | \BF^*_n |^2 r \, dr\right)^{\frac12} .
\end{equation}

$(2)$\ $I_1(\rho)$ is the modified Bessel function of the first kind, i.e.,
\be \label{eqBessel1}
\left\{
\ba & \rho^2 \frac{d^2}{d\rho^2} I_1 (\rho ) + \rho \frac{d}{d\rho} I_1 (\rho)  - (\rho^2 + 1) I_1 (\rho ) = 0, \\
& I_1 (0) = 0,\quad I_1(\rho ) >0 \,\,\text{if}\,\, \rho >0.
\ea
\right.
\ee
Furthermore, $a$ is a constant satisfying
\be \label{case3-6}
|a| \leq C (\Phi |n |)^{-\frac74} (4+  \alpha)^{\frac{11}{4}}   I_1( |n| )^{-1} \left( \int_0^1 | \BF^*_n |^2 r \, dr      \right)^{\frac12}.
\ee

$(3)$ \ $ \psi_{BL, n }$ is the boundary layer profile defined by
\be \label{case3-7}
\psi_{BL, n}  (r) = e^{- \sqrt{\beta} \cos \frac{\theta}{2} (1 - r) } e^{-i \sqrt{\beta} \sin \frac{\theta}{2} (1 - r) },
\ee
where $\beta$ and $\theta$ are defined as follows
\be \label{defbeta}
\beta^2 = \left( \frac{\Phi n}{\pi} \right)^2 \left( \frac{4}{4 + \alpha} \right)^2 + n^4,\ \ \ \cos \theta = \frac{n^2}{\beta}, \ \ \text{and}\ \
\sin \theta = \frac{4\Phi n}{\pi (4 + \alpha) \beta}.
\ee
Moreover,  $\chi$ is a  smooth increasing cut-off function satisfying
\be \label{defchi}
\chi (r) = \left\{ \ba  &  1, \ \ \ \ r\geq \frac12,  \\ & 0, \ \ \  \ r \leq \frac14,   \ea  \right.
\ee
and the constant $b$ satisfies
\be \label{case3-8}
|b|\leq C (\Phi |n|)^{-\frac74} (4 +  \alpha)^{\frac{11}{4}}   \left( \int_0^1 |\BF^*_n |^2 r \, dr  \right)^{\frac12}.
\ee

$(4)$\ $\psi_{e, n} $ is a remainder term satisfying
\be \label{case3-9}
\left\{  \ba  &  i n  \bar{U}(r) ( \mathcal{L} - n^2) (\chi \psi_{BL, n}  + \psi_{e,n} ) - (\mathcal{L} - n^2)^2 (\chi \psi_{BL, n}+ \psi_{e, n} ) = 0,   \\
& \psi_{e, n} (0) = \psi_{e, n} (1) = \mathcal{L}\psi_{e, n} (0) = \mathcal{L} \psi_{e, n} (1) = 0.
\ea  \right.
\ee
 And $\psi_{e, n} $ satisfies the following estimates,
\be \label{case3-10}
\int_0^1 \left| \frac{d}{dr}(r \psi_{e, n}) \right|^2 \frac1r   + n^2  |\psi_{e, n} |^2 r \, dr  \leq C (4 +  \alpha)^2 \beta^{- \frac12} ,
\ee
\be \label{case3-11}
\int_0^1 |\mathcal{L} \psi_{e, n} |^2 r  + n^2  \left|\frac{d}{dr}(r \psi_{e, n}) \right|^2 \frac1r+ n^4  | \psi_{e, n} |^2 r\, dr
\leq C \Phi |n| (4 +  \alpha) \beta^{-\frac12},
\ee
and
\be \label{case3-12} \ba
& \int_0^1 \left| \frac{d}{dr}( r \mathcal{L} \psi_{e, n} ) \right|^2 \frac1r  + n^2   |\mathcal{L} \psi_{e, n} |^2 r +
n^4 \left| \frac{d}{dr}( r \psi_{e, n} )  \right|^2 \frac1r    + n^6  | \psi_{e, n} |^2 r\, dr \\
\leq  &  C (\Phi |n|)^2 \beta^{-\frac12}.
\ea
\ee

In conclusion, $\psi_n $ satisfies
\be \label{case3-13}
\int_0^1 \left| \frac{d}{dr}(r \psi_n)\right|^2 \frac1r  + n^2  |\psi_n|^2  r \, dr
\leq C (\Phi |n|)^{-\frac43} \int_0^1 |\BF_n^*|^2 r \, dr
\ee
and
\be \label{case3-15}
\ba
&\int_0^1 \left| \frac{d}{dr}( r \mathcal{L} \psi_n ) \right|^2 \frac1r  + n^2   |\mathcal{L} \psi_n |^2 r +
n^4  \left| \frac{d}{dr}( r  \psi_n)  \right|^2 \frac1r   + n^6 | \psi_n |^2 r\, dr
\leq  C \int_0^1 | \BF^*_n |^2 \, dr .
\ea
\ee

\end{pro}

Before we give the proof of Proposition \ref{Bpropcase3}, we  study the linear problem \eqref{slip} in detail. The only difference between the problems \eqref{slip} and  \eqref{2-0-8} is that the vorticity vanishes on the solid boundary for the solutions of the problem \eqref{slip}. Define the function space
\be \nonumber
L_r^2(0, 1) = \left\{   f(r): \ \|f\|_{L_r^2(0, 1)}^2 = \int_0^1 |f|^2 r \, dr < +\infty            \right\}.
\ee
\begin{lemma}\label{propslip}
Given $f_n \in L_r^2(0, 1)$, the problem \eqref{slip} admits a unique solution $ \psi_{s, n} $  satisfying  the estimates
\be \label{propslip1-1}
\int_0^1 \left| \frac{d}{dr} ( r \psi_{s, n}  )  \right|^2 \frac1r  + n^2  | \psi_{s, n} |^2 r \, dr \leq C (\Phi |n|)^{-2} (4 +  \alpha)^2  \int_0^1 |f_n|^2 r \, dr ,
\ee
\be \label{propslip1-2} \ba
& \int_0^1 | \mathcal{L} \psi_{s, n} |^2 r  + n^2  \left|  \frac{d}{dr} ( r \psi_{s, n}  )\right|^2 \frac1r  + n^4
 | \psi_{s, n} |^2  r \, dr
 \leq  C ( \Phi |n|)^{-2} (4 +  \alpha)^2 \int_0^1 |f_n |^2 r \, dr ,
 \ea
\ee
\be \label{propslip1-3} \ba
& \int_0^1 \left| \frac{d}{dr}( r \mathcal{L} \psi_{s, n} )\right|^2 \frac1r  +
n^2  |\mathcal{L} \psi_{s, n}  |^2 r
+ n^4  \left|  \frac{d}{dr}( r \psi_{s, n} )       \right|^2 \frac1r  + n^6  | \psi_{s, n}  |^2  r \, dr\\
\leq &   C \left( \Phi |n|  \right)^{-1} (4 +  \alpha) \int_0^1 |f_n |^2 r\, dr ,
\ea
\ee
\be \label{propslip1-4}
\left| \frac{d}{dr}(r \psi_{s, n} ) (1) \right| \leq C (\Phi |n|)^{-1} (4 +  \alpha) \left( \int_0^1 |f_n|^2 r \, dr \right)^{\frac12}.
\ee
Moreover, if $f_n = in F_n^r - \frac{d}{dr} F_n^z $ and $F_n^r, F_n^z \in L_r^2(0, 1 )$, the solution $\psi_{s, n} $ also satisfies the estimates \eqref{case3-3}--\eqref{case3-5-1}.

\end{lemma}

\begin{proof}  {\it Step 1. Proof of \eqref{case3-3}--\eqref{case3-5-1}.}
Multiplying the equation in \eqref{slip} by $r (\mathcal{L}- n^2) \overline{\psi_{s, n} } $ and integrating over $[0, 1]$ yield
\be \label{B-51}
\ba
& \int_0^1 \left|\frac{d}{dr}(r\mathcal{L} \psi_{s, n}) \right|^2  \frac{1}{r} \, dr
+3 n^2 \int_0^1 \left|\mathcal{L} \psi_{s, n} \right|^2 r \, dr
+ 3 n^4 \int_0^1 \left| \frac{d}{dr} ( r \psi_{s, n} )\right|^2 \frac1r \, dr\\
&+ n^6 \int_0^1 | \psi_{s, n}|^2 r \, dr
 + i n \int_0^1 \bar{U}(r) | (\mathcal{L} - n^2) \psi_{s, n}  |^2  r \, dr
= \int_0^1 f_n  ( \mathcal{L} - n^2) \overline{ \psi_{s, n}  } r\, dr .
\ea
\ee
Note that
\begin{equation} \nonumber
\begin{aligned}
\left|\int_0^1 f_n  \mathcal{L}\overline{ \psi_{s, n} } r \, dr\right|\leq & \left|\int_0^1 F^r_n  i n \mathcal{L}\overline{ \psi_{s, n} }r + F^z_n
\frac{d}{dr}(r\mathcal{L}\overline{ \psi_{s, n} })\, dr\right|\\
\leq & 2 \int_0^1 | \BF^*_n |^2 r \, dr + \frac14   \int_0^1  \left|\frac{d}{dr}(r\mathcal{L} \psi_{s, n }) \right|^2 \frac1r  \, dr
+\frac14 n^2 \int_0^1 \left|\mathcal{L} \psi_{s, n}  \right|^2 r \, dr
\end{aligned}
\end{equation}
and
\begin{equation} \nonumber
\begin{aligned}
\left|\int_0^1 f_n n^2 \overline{\psi_{s, n} } r \, dr\right|\leq & \left|\int_0^1 F^r_n  in^3 \overline{\psi_{s, n} }r + F^z_n  n^2\frac{d}{dr} (r\overline{ \psi_{s, n} })dr\right|\\
\leq & 2 \int_0^1 | \BF^*_n |^2 rdr + \frac14 n^4 \int_0^1 \left|\frac{d}{dr}(r \psi_{s, n} ) \right|^2 \frac1r  \, dr
+\frac14 n^6 \int_0^1 \left| \psi_{s, n}  \right|^2 r \, dr.
\end{aligned}
\end{equation}
Hence one has the estimate \eqref{case3-5}.

On the other hand, it holds that
\be \label{B-55}
n \int_0^1 \bar{U}(r) |(\mathcal{L} - n^2 ) \psi_{s, n}  |^2 r \, dr = \Im \int_0^1 f_n ( \mathcal{L} - n^2) \overline{ \psi_{s, n} } r \, dr .
\ee
Since $\bar{U}(r) \geq \frac{\Phi}{\pi} \frac{4}{4 +  \alpha} $, the above expression \eqref{B-55} together with \eqref{case3-5} implies
\be \nonumber
\int_0^1 | (\mathcal{L} - n^2 ) \psi_{s, n}  |^2 r \, dr \leq C (\Phi |n|)^{-1}(4 +  \alpha) \int_0^1 | \BF^*_n |^2 r \, dr .
\ee
This is exactly the estimate \eqref{case3-4}.

Moreover,  multiplying the equation in \eqref{slip} by $r\overline{ \psi_{s, n} } $ and integrating over $[0, 1]$ yield
\be \label{5-51}
n \int_0^1 \frac{\bar{U}(r)}{r} \left|\frac{d}{dr} ( r \psi_{s, n} )   \right|^2 \, dr + n^3
\int_0^1 \bar{U}(r) |\psi_{s, n}|^2 r \, dr = - \Im \int_0^1 f_n \overline{ \psi_{s,n} } r \, dr. \ee
Hence one has
\be \label{B66-1}  \ba
& \Phi |n| \frac{4}{4 + \alpha} \left[  \int_0^1 \left| \frac{d}{dr}(r \psi_{s, n} )  \right|^2 \frac1r \, dr + n^2 \int_0^1 | \psi_{s, n} |^2 r \, dr \right] \\
\leq &  C \left( \int_0^1 | \BF^*_n |^2 r \, dr  \right)^{\frac12} \left[    \int_0^1 \left| \frac{d}{dr}(r \psi_{s, n})  \right|^2 \frac1r \, dr + n^2 \int_0^1 | \psi_{s, n} |^2 r \, dr    \right]^{\frac12} .
\ea
 \ee
 This, together with Lemma \ref{lemma1}, yields \eqref{case3-3}.

Furthermore, it follows from  Lemma \ref{lemmaA2} that one has
\begin{equation} \nonumber \ba
\left|\frac{d}{dr}(r \psi_{s, n} ) (1) \right| & \leq C \left( \int_0^1 |\mathcal{L} \psi_{s, n}  |^2 r \, dr    \right)^{\frac{1}{4}} \left( \int_0^1 \left|  \frac{d}{dr}(r \psi_{s, n}  ) \right|^2 \frac1r \, dr  \right)^{\frac14} \\
&\leq  C (\Phi |n|)^{-\frac34} (4 + \alpha)^{\frac34}  \left( \int_0^1 | \BF^*_n |^2 r \, dr \right)^{\frac12}.
\ea
\end{equation}

{\it Step 2. Proof of  \eqref{propslip1-1}--\eqref{propslip1-4}.  }
 Next, we estimate $\psi_{s, n}$ in terms of $f_n$. According to \eqref{B-55}, it holds that
\be \label{B66-3}
\int_0^1 | ( \mathcal{L} - n^2 ) \psi_{s, n} |^2 r \, dr \leq C (\Phi |n | )^{-2}(4  +  \alpha)^2 \int_0^1 | f_n |^2 r\, dr.
\ee
This is exactly the estimate \eqref{propslip1-2}, which together with Lemma \ref{lemma1} gives \eqref{propslip1-1}.
And consequently, by virtue of Lemma \ref{lemmaA2}, one has
\be \nonumber
\left|\frac{d}{dr}(r\psi_{s, n} )(1) \right| \leq C (\Phi |n|)^{-1 } (4  +  \alpha) \left( \int_0^1 | f_n |^2 r \, dr\right)^{\frac12}.
\ee

Substituting the estimate \eqref{B66-3} into \eqref{B-51} yields \eqref{propslip1-3}.

The existence of the solution $\psi_{s, n} $ follows from  similar arguments as in Section \ref{sec-ex}. Hence the proof of Lemma \ref{propslip} is completed.
\end{proof}

Now we give the proof of Proposition \ref{Bpropcase3}.

\begin{proof}[Proof of Proposition \ref{Bpropcase3}]
 Let $ \psi_{s, n} $ denote the solution to \eqref{slip}. Note that $ \psi_{s, n} $ satisfies the same equation as $ \psi_n $, but with a different boundary condition. The Navier slip boundary condition is recovered by the boundary layer analysis.
Define
\be \label{defA}
\mcA = i n \bar{U}(r) - \mathcal{L} + n^2 , \ \ \ \ \mcH= \mathcal{L} - n^2,
\ee
and
\be \label{deftildeA}
\widetilde{\mcA} = i \frac{ \Phi n }{\pi} \frac{4}{4 +  \alpha}  - \frac{d^2}{dr^2} + n^2,\ \ \ \ \widetilde{\mcH}= \frac{d^2}{dr^2} - n^2.
\ee
$\widetilde{\mathcal{A}}$ and $\widetilde{\mathcal{H}}$ can be regarded as the leading parts of $\mcA$ and $\mcH$ near the solid boundary $r=1$, respectively.
The straightforward computations show that $ \psi_{BL, n} $ satisfies
\be \nonumber
\widetilde{\mcA}\widetilde{\mcH} \psi_{BL, n}  = 0 .
\ee
Let $\psi_{BL,n}$ be defined in \eqref{case3-7} with $\beta$ and $\theta$ in \eqref{defbeta}.

Next, we construct the remainder term $\psi_{e, n} $, which is the solution to the following problem
\be \label{5-81}
\left\{  \ba
&  i n \bar{U}(r) ( \mathcal{L} - n^2) \psi_{e, n} - (\mathcal{L} - n^2)^2 \psi_{e, n} = \widetilde{\mcA}\widetilde{\mcH} ((1 - \chi) \psi_{BL, n}  ) + (\widetilde{\mcA}\widetilde{\mcH} - \mcA\mcH) ( \chi \psi_{BL, n}  ) ,\\
&  \psi_{e, n} (0) = \psi_{e, n} (1) = \mathcal{L} \psi_{e, n} (0) = \mathcal{L} \psi_{e, n} (1) = 0,
\ea
\right.
\ee
where $\chi$ is a smooth cut-off function satisfying \eqref{defchi}.
Regarding the right hand of \eqref{5-81}, one has
\be \label{5-83} \ba
&( \widetilde{\mcA}\widetilde{\mcH} - \mcA\mcH ) (\chi \psi_{BL, n}  )   =
(\widetilde{\mcA} - \mcA) \mcH ( \chi \psi_{BL, n}  )  + \widetilde{\mcA} ( \widetilde{\mcH} - \mcH) ( \chi \psi_{BL, n}  ) \\
 = &\left( i \frac{ \Phi n }{\pi} \frac{ - 2\alpha}{4 +  \alpha} (1 - r^2)  + \frac1r \frac{d}{dr} - \frac{1}{r^2}  \right)
\left( \frac{d^2}{dr^2} + \frac1r \frac{d}{dr} - \frac{1}{r^2} - n^2 \right) ( \chi \psi_{BL, n} ) \\
&\ \  \ + \left(  i \frac{ \Phi n }{\pi } \frac{4}{4 + \alpha}  - \frac{d^2}{dr^2} + n^2     \right) \left( - \frac1r \frac{d}{dr} +
\frac{1}{r^2}  \right) ( \chi \psi_{BL, n} ).
\ea
\ee
Denote $F_{e,n}^z$, $F_{e,n}^r$, and $f_{e,n}$ as follows
\be \label{83-2}
\begin{aligned}
F_{e, n}^z = &  i \frac{\Phi n }{\pi } \frac{2\alpha }{4 + \alpha} (1 -  r^2) \frac{d}{dr} (\chi \psi_{BL, n} )
-  i \frac{\Phi n}{\pi} \frac{ 4\alpha}{4 + \alpha} r \chi \psi_{BL, n} \\
 & \ + i \frac{\Phi n }{\pi}\frac{4}{4 + \alpha} \frac1r \chi \psi_{BL, n}
   -  \frac2r \frac{d^2}{dr^2} (\chi \psi_{BL, n} ) - \frac{1}{r^2} \frac{d}{dr} (\chi \psi_{BL, n} ),
\end{aligned}
\ee
\be \label{83-4}
F_{e, n}^r =  \frac{\Phi n^2}{\pi} \frac{2\alpha}{4 + \alpha} (1 - r^2) (\chi \psi_{BL, n}) +  2i n \frac{1}{r} \frac{d}{dr}(\chi \psi_{BL, n} ),
\ee
and
\be \nonumber
\ba
f_{e,n}
= &  i \frac{\Phi n}{\pi} \frac{4\alpha}{4 + \alpha} \chi \psi_{BL, n}    - i \frac{\Phi n }{\pi} \frac{4}{4 + \alpha} \frac{1}{r^2} \chi \psi_{BL, n}  + \frac{1}{r^2} \frac{d^2}{dr^2} (\chi \psi_{BL, n} ) \\
&\ - \frac{1}{r^2} \left( \frac{d^2}{dr^2} + \frac{1}{r} \frac{d}{dr}- \frac{1}{r^2} - n^2  \right) (\chi \psi_{BL, n})
 + \left( i \frac{\Phi n}{\pi} \frac{4}{4 + \alpha} - \frac{d^2}{dr^2} + n^2  \right) \frac{1}{r^2} (\chi \psi_{BL, n} ) \\
&\  + \left[ - i \frac{\Phi n}{\pi} \frac{ 2\alpha}{4 + \alpha} (1  - r^2) + \frac{1}{r} \frac{d}{dr}  \right] \left( \frac{1}{r} \frac{d}{dr}- \frac{1}{r^2} \right)  (\chi \psi_{BL, n}) .
\ea \ee
Then one has
\be  \nonumber
\ba
&( \widetilde{\mcA}\widetilde{\mcH} - \mcA\mcH ) (\chi \psi_{BL, n}  )   =
- \frac{d}{dr}F_{e,n}^z + in F_{e, n}^r + f_{e, n}.
\ea
\ee

 Note that $\theta \in (-\frac{\pi}{2},  \frac{\pi}{2} )$ and
\be \nonumber
\beta^2 = \left( \frac{\Phi n }{\pi} \right)^2 \left( \frac{4}{4 + \alpha} \right)^2 + n^4 \leq C \frac{\Phi^2 n^2}{(4 + \alpha)^2}.
\ee
 Thus
$
\cos \frac{\theta}{2} \geq \frac{\sqrt{2} }{2}
$
and
\be \label{83-7} \ba
 \int_0^1 | F_{e, n}^z |^2 r \, dr
 \leq & C \Phi^2 n^2 \int_0^1 (1 - r)^2  \left| \frac{d}{dr}(\chi \psi_{BL, n}) \right |^2 r + |\chi \psi_{BL, n} |^2 \frac1r \, dr \\
&\,\, + C\int_0^1 \left| \frac{d^2}{dr^2}(\chi \psi_{BL, n} )  \right|^2 \frac1r \, dr  + C \int_0^1 \left| \frac{d}{dr} ( \chi \psi_{BL, n} ) \right|^2 \frac{1}{r^3} \, dr \\
\leq & C \Phi^2 n^2 \left[\beta \left(\sqrt{\beta} \cos \frac{\theta}{2} \right)^{-3} +  \sqrt{\beta} e^{-\frac32 \sqrt{\beta} \cos \frac{\theta}{2}} +  \left(\sqrt{\beta} \cos \frac{\theta}{2}\right)^{-1}\right]  \\
&\ \ \ \ \   + C \beta^{\frac32} + C \sqrt{\beta} e^{-\sqrt{\beta} \cos \frac{\theta}{2} }  \\
\leq & C \Phi^2 n^2 \beta^{-\frac12}.
\ea \ee
Similarly, one has
\be \label{83-8} \ba
\int_0^1 | F_{e, n}^r|^2 r \, dr
 \leq & C \Phi^2 n^4 \left(\sqrt{\beta} \cos \frac{\theta}{2} \right)^{-3}  + C n^2 \beta^{\frac12} + C n^2 \left(\sqrt{\beta} \cos \frac{\theta}{2} \right)^{-1} e^{-\sqrt{\beta} \cos \frac{\theta}{2} }\\
 \leq & C \Phi^2 n^2 \beta^{-\frac12}
\ea
\ee
and
\be \label{83-9}
\int_0^1 | f_{e, n}|^2 r\, dr  + \int_0^1 | \widetilde{\mathcal{A}}\widetilde{\mathcal{H}} (1 - \chi ) \psi_{BL, n} |^2 r \, dr \leq C \Phi^2 n^2 \beta^{-\frac12}.
\ee

It follows from Lemma \ref{propslip} that one has
\be \label{5-84} \ba
& \int_0^1 | \psi_{e, n} |^2 r \, dr + \int_0^1 \left| \frac{d}{dr}( r \psi_{e, n}  )   \right|^2 \frac1r \, dr + n^2 \int_0^1 | \psi_{e,n} |^2 r \, dr \\
\leq &  C (\Phi |n|)^{-2} (4 + \alpha)^2 \Phi^2 n^2 \beta^{-\frac12} \\
\leq & C (4 + \alpha)^2 \beta^{-\frac12}.
\ea
\ee
Furthermore,  combining \eqref{propslip1-2} and \eqref{case3-4} gives
\be \label{5-85} \ba
& \int_0^1 | \mathcal{L} \psi_{e, n}  |^2 r\, dr + n^2 \int_0^1 \left| \frac{d}{dr} ( r \psi_{e, n} ) \right|^2 \frac1r \, dr + n^4 \int_0^1 |\psi_{e,n}|^2 r\, dr \\
\leq & C (\Phi|n|)^{-1} ( 4 + \alpha) \Phi^2 n^2 \beta^{-\frac12}
\leq  C \Phi |n| (4 + \alpha) \beta^{-\frac12}.
\ea
\ee
Similarly, it follows from the estimates \eqref{propslip1-3} and \eqref{case3-5} that one has
\be \label{5-86}\ba
& \int_0^1 \left| \frac{d}{dr} ( r \mathcal{L} \psi_{e, n} ) \right|^2 \frac1r  + n^2 |\mathcal{L} \psi_{e, n} |^2 r  + n^4  \left| \frac{d}{dr} ( r \psi_{e, n} ) \right|^2 \frac1r
 + n^6  | \psi_{e, n}|^2 r\, dr \\
\leq &  C( \Phi |n|)^2 \beta^{-\frac12} .
\ea
\ee
Moreover,
\be \label{5-87}
\left| \frac{d}{dr}( \psi_{e, n} ) (1)  \right|  \leq C (\Phi |n|)^{-\frac34} (4 + \alpha)^{\frac34} \Phi|n| \beta^{-\frac14}
\leq C (\Phi |n|)^{\frac14} ( 4 + \alpha)^{\frac34} \beta^{-\frac14}.
\ee

Finally, for the modified Bessel function of the first kind,  $I_1(z)$, which satisfies \eqref{eqBessel1}, the straightforward computations yield
\begin{equation}\nonumber
	\mcH I_1(|n|r)=0.
\end{equation}
To guarantee that $\psi_n$ satisfies the boundary conditions in \eqref{2-0-8}, one must have
\be \nonumber
\psi_n(1) = 0\ \ \ \ \ \mbox{and}\ \ \ \ \mathcal{L} \psi_n (1) = - \alpha \psi_n^{\prime}(1),
\ee
Hence the constants $a$ and $b$ must satisfy
\be \label{5-87} \left\{ \ba
& a I_1 (|n|) + b \psi_{BL, n}  (1) = 0, \\
& a\left[ |n |^2 I_1(|n|) +  \alpha |n| I_1^{\prime}(|n| ) \right] \\
&+  b \left[ \frac{d^2}{dr^2} \psi_{BL, n} (1) +  (1  + \alpha) \frac{d}{dr} \psi_{BL, n} (1)  - \psi_{BL, n} (1) +  \alpha \frac{d}{dr} \psi_{e, n} (1)  \right] =  - \alpha \frac{d}{dr} \psi_{s, n} (1) .    \ea
\right.
\ee
Solving the linear system \eqref{5-87} yields
\be \label{5-88}
b=  \frac{ -  \alpha \frac{d}{dr} \psi_{s, n}  (1)}{J}\ \ \ \ \text{and}\quad  a = -\frac{b \psi_{BL, n} (1) }{I_1(|n|)},
\ee
where
\be \label{5-89}
\begin{aligned}
J = & - \psi_{BL, n} (1) \left[ |n|^2 +  \alpha |n| \frac{I_1^{\prime}(|n|)}{I_1 (|n|)} \right] \\
&+ \left[ \frac{d^2\psi_{BL, n}}{dr^2}  (1) + (1 +  \alpha) \frac{d \psi_{BL, n} }{dr}(1) - \psi_{BL, n} (1) +  \alpha \frac{d\psi_{e, n} }{dr}  (1)   \right].
\end{aligned}
\ee
Herein, since $\beta > \frac{\Phi |n| }{\pi} \frac{4}{4 + \alpha}  $ and $4+  \alpha \leq \delta (\Phi |n| )^{\frac13}$, one has
\be \nonumber \ba
& \left|  \frac{d^2\psi_{BL, n}  }{dr^2} (1) + (1 +  \alpha) \frac{d\psi_{BL, n}}{dr}  (1) - \psi_{BL, n}(1) +  \alpha \frac{d  \psi_{e, n}}{dr} (1) \right| \\
\geq & \beta -  (1 + \alpha) \beta^{\frac12} - 1 - C \alpha (4 +  \alpha)^{\frac34}(\Phi |n|)^{\frac14} \beta^{-\frac14}  , \\
\geq &  \beta -  (1 + \alpha) \beta^{\frac12} - 1 - C (4 + \alpha)^{\frac74} (\Phi |n|)^{\frac14} \beta^{-\frac14}\\
\geq & \beta -  (1 + \alpha) \beta^{\frac12} - 1 - C \delta^3 \beta.
\ea \ee
It follows from Lemma \ref{lemBessel} that
\be \nonumber
\left| - \psi_{BL, n}(1) \left[ |n|^2 +  \alpha |n| \frac{I_1^{\prime}(|n|)}{I_1 (|n|)} \right] \right|
\leq  n^2 +  |\alpha| (1 + |n|) .
\ee
Hence, if $\epsilon_1$ and $\delta$ are small enough, then $\beta$ is large, and
\be \nonumber
|J | \geq \frac18 \beta.
\ee
Therefore, according to Lemma \ref{propslip}, one has
\be \nonumber \ba
|b| & \leq C |\alpha| \left| \frac{d}{dr} \psi_{s, n} (1)\right| \beta^{-1}
\leq C \alpha \beta^{-1} (\Phi |n| )^{-\frac34} (4 + \alpha)^{\frac34} \left( \int_0^1 |\BF^*_n|^2 r \, dr \right)^{\frac12} \\ & \leq C (4 + \alpha)^{\frac{11}{4}} ( \Phi |n|)^{- \frac74}  \left( \int_0^1  |\BF^*_n |^2 r \, dr \right)^{\frac12}
 \ea
\ee
and
\be
|a| \leq C (4 +  \alpha)^{\frac{11}{4}} (\Phi |n| )^{-\frac74}  I_1(|n|)^{-1} \left( \int_0^1  |\BF^*_n |^2 r \, dr \right)^{\frac12}.
\ee
Moreover, since $n^4 \leq \epsilon_1^{\frac83} (\Phi |n|)^{\frac43} $ and then $\beta \leq C \frac{ \Phi |n| }{4 + \alpha}$,  it holds that
\be \label{B-95-2} \ba
& b^2 \int_0^1 \left|\frac{d}{dr} \left[ r \chi \psi_{BL, n}  \right] \right|^2 \frac1r \, dr + b^2 n^2 \int_0^1 | \chi \psi_{BL, n}|^2 r \, dr   \\
\leq & C b^2 ( \beta^{\frac12} +  n^2 \beta^{-\frac12})  \leq C (4 + \alpha)^5 (\Phi |n|)^{-3} \int_0^1 |\BF^*_n|^2 r \, dr \\
\leq & C (\Phi |n|)^{-\frac43} \int_0^1 |\BF_n^*|^2 r \, dr
\ea
\ee
and
\be \label{B-95-4}\ba
& b^2 \int_0^1 \left|  \frac{d}{dr} ( r \mathcal{L} (\chi \psi_{BL, n} )) \right|^2 \frac1r \, dr + b^2 n^2 \int_0^1 \left|\mathcal{L} (\chi \psi_{BL, n} ) \right|^2 r \, dr \\
&+   b^2  n^4 \int_0^1 \left| \frac{d}{dr} ( r \chi \psi_{BL, n} )\right|^2 \frac1r \, dr    + b^2 n^6 \int_0^1 \left| \chi \psi_{BL, n} \right|^2 r \, dr\\
\leq & C b^2 (\beta^{\frac52} + n^6 \beta^{-\frac12}) \leq C(4 + \alpha)^3 (\Phi |n|)^{-1}  \int_0^1 |\BF^*_n|^2 r \, dr \\
\leq & C \int_0^1 |\BF_n^*|^2 r \, dr.
\ea
\ee

Furthermore, one has
\be \label{B-95-5} \ba
& b^2 \int_0^1 \left|\frac{d}{dr}(r \psi_{e, n}) \right|^2 \frac1r \, dr + b^2n^2 \int_0^1 |\psi_{e, n}|^2 r \, dr \\
\leq & C (4+ \alpha)^{\frac{11}{2}} (\Phi |n|)^{-\frac72} (4 + \alpha)^2 \beta^{-\frac12} \int_0^1 |\BF_n^*|^2 r \, dr \\
\leq & C (4 + \alpha)^8 (\Phi |n| )^{-4} \int_0^1 |\BF_n^*|^2 r \, dr \leq C (\Phi |n|)^{-\frac43} \int_0^1 |\BF^*_n|^2 r \, dr
\ea
\ee
and
\be \label{B-95-6}
\ba
& b^2 \int_0^1 \left| \frac{d}{dr}(r \mathcal{L} \psi_{e, n})  \right|^2 \frac1r \, dr + b^2 n^2 \int_0^1 |\mathcal{L} \psi_{e, n}|^2 r \, dr \\
&\ \ \ + b^2 n^4 \int_0^1 \left| \frac{d}{dr}(r \psi_{e,n})   \right|^2 \frac1r \, dr + b^2 n^6 \int_0^1 |\psi_{e,n}|^2 r \, dr \\
\leq & C (4+  \alpha)^{\frac{11}{2}} (\Phi |n|)^{-\frac72} (\Phi |n| )^2 \beta^{-\frac12} \int_0^1 |\BF^*_n|^2 r \, dr \\
\leq & C (4 + \alpha)^6 (\Phi |n|)^{-2} \int_0^1 |\BF^*_n|^2 r \, dr \leq C \int_0^1 |\BF^*_n|^2 r \, dr.
\ea
\ee

Meanwhile, according to Lemma \ref{AlemBessel2}, it holds that
\be \label{B-96-2}
\ba
&a^2\left( n^2 \int_0^1 | I_1(|n|r)|^2 r  \, dr +
\int_0^1 \left| \frac{d}{dr} (r I_1 (|n| r) ) \right|^2 \frac1r \, dr \right)\\
\leq & C (\min\{1, |n|^{-1}\} n^2 + \max\{1, |n|\} ) (\Phi |n|)^{-\frac72} (4 +  \alpha)^{\frac{11}{2}}    \int_0^1| \BF^*_n |^2 r \, dr \\
\leq & C  (\Phi |n|)^{-\frac43} \int_0^1|\BF^*_n |^2 r \, dr
\ea
\ee
and
\be \label{B-96-4} \ba
& a^2 \int_0^1 \left| \frac{d}{dr} ( r \mathcal{L} I_1(|n|r))  \right|^2 \frac1r \, dr + a^2 n^2 \int_0^1 |\mathcal{L} I_1(|n|r)|^2 r \, dr \\
&\quad + a^2 n^4 \int_0^1 \left| \frac{d}{dr} (r I_1 (|n| r)) \right|^2 \frac1r \, dr   + a^2  n^6 \int_0^1 |I_1 (|n| r)|^2 r \, dr\\
 \leq &  C \left( \min\{1, |n|^{-1}\} |n|^6 + \max\{1, |n|\}|n|^4 \right) (\Phi |n| )^{-\frac72} (4 +  \alpha)^{\frac{11}{2} }   \int_0^1 |\BF^*_n |^2 r \, dr \\
 \leq &  C   \int_0^1 | \BF^*_n |^2 r \, dr.
\ea
\ee
 Combining the estimates in Lemma \ref{propslip} and \eqref{B-95-2}--\eqref{B-96-4} gives \eqref{case3-13}-\eqref{case3-15}.     This finishes the proof of Proposition \ref{Bpropcase3}.
\end{proof}

\subsubsection{The case with large flux, intermediate frequency, and large slip coefficient}
\begin{pro} \label{case4}
Assume that $\Phi \gg 1$. There exist two small independent positive constants $\epsilon_1$ and  $\delta$, such that as long as
$1 \leq |n| \leq \epsilon_1 \sqrt{\Phi}$ and $4 + \alpha\geq \delta^{-1} (\Phi |n|)^{\frac13}>5$, the solution $\psi_n(r)$ to the problem \eqref{2-0-8} can be decomposed into four parts,
\be \nonumber
\psi_n (r) = \psi_{s, n}(r) + b \left[ \chi \psi_{BL, n}(r) + \psi_{e, n}(r)    \right] + a I_1 (|n| r).
\ee
Here $(1)$\ $\psi_{s, n}$ is a solution to the problem \eqref{slip} satisfying
\be \label{propslip1-large-6}
\int_0^1 \left|\frac{d}{dr}(r  \psi_{s, n} ) \right|^2  \frac{1}{r}\, dr
+ n^2 \int_0^1 \left| \psi_{s, n}  \right|^2 r \, dr
\leq C (\Phi |n|)^{- \frac43 } \int_0^1 | \BF^*_n |^2 r \, dr ,
\ee
\be \label{propslip1-large-7}
\begin{aligned}
&\int_0^1 | \mathcal{L} \psi_{s, n} |^2 r + n^2  \left|  \frac{d}{dr} ( r \psi_{s, n}  )\right|^2 \frac1r + n^4
 | \psi_{s, n} |^2  r \, dr
 \leq  C ( \Phi |n|)^{-\frac23} \int_0^1 |\BF^*_n|^2 r \, dr ,
\end{aligned}
\ee
\be \label{propslip1-large-8} \ba
& \int_0^1 \left| \frac{d}{dr}( r \mathcal{L} \psi_{s, n}  )\right|^2 \frac1r  +
n^2  |\mathcal{L} \psi_{s, n}  |^2 r  + n^4  \left|  \frac{d}{dr}( r \psi_{s, n} )       \right|^2 \frac1r  + n^6  |\psi_{s, n}  |^2 r \, dr\\
\leq &   C \int_0^1 | \BF^*_n |^2 r\, dr ,
\ea
\ee
and
\be \label{propslip1-large-9}
\left| \frac{d}{dr} \psi_{s, n}  (1) \right| \leq C (\Phi |n|)^{-\frac12} \left( \int_0^1 |\BF^*_n |^2 r \, dr \right)^{\frac12}.
\ee

$(2)$\ $I_1(\rho)$ is the modified Bessel function of the first kind as in Proposition \ref{Bpropcase3}, and $a$ is a constant satisfying
\be \nonumber
|a| \leq C (\Phi |n|)^{-\frac56} I_1(|n|)^{-1} \left(   \int_0^1 |\BF_n^*|^2 r \, dr  \right)^{\frac12}.
\ee

$(3)$\ $\psi_{BL, n}$ is the boundary layer profile,
\be \label{case4-7}
\psi_{BL, n}(r)  = G_{n, \Phi} ( |\beta| (1 - r)) \quad \text{with}\quad |\beta|=\left(\frac{4\Phi |n| }{\pi}\right)^{\frac13},
\ee
where $G_{n,  \Phi}(\rho) $ is a smooth function which decays exponentially at infinity and is uniformly bounded in the set
\be \nonumber
\mcE = \left\{ ( n, \Phi, \rho):\ \Phi \geq 1, \ 1 \leq |n| \leq \sqrt{\Phi},\  0\leq \rho < + \infty   \right\}.
\ee
Moreover, $\chi$ is the smooth function satisfying \eqref{defchi} and $b$ is a constant satisfying
\be \nonumber
|b| \leq C (\Phi |n|)^{-\frac56}  \left(   \int_0^1 |\BF_n^*|^2 r \, dr  \right)^{\frac12}.
\ee

$(4)$\ $\psi_{e, n}$ is a remainder term, which satisfies \eqref{case3-9} and the following estimates
\be \label{case4-10}
\int_0^1 \left| \frac{d}{dr}(r \psi_{e, n} ) \right|^2 \frac1r \, dr + n^2 \int_0^1 |\psi_{e, n} |^2 r \, dr \leq C (\Phi |n|)^{\frac13} ,
\ee
\be \label{case4-11}
\int_0^1 |\mathcal{L} \psi_{e, n} |^2 r \, dr + n^2 \int_0^1 \left| \frac{d}{dr}(r \psi_{e, n} ) \right|^2 \frac1r \, dr + n^4 \int_0^1 |\psi_{e, n} |^2 r \, dr \leq C \Phi |n|,
\ee
and
\be \label{case4-12} \ba
& \int_0^1 \left| \frac{d}{dr} ( r \mathcal{L} \psi_{e, n} ) \right|^2 \frac1r  + n^2  |\mathcal{L} \psi_{e, n} |^2 r  + n^4  \left| \frac{d}{dr} (r \psi_{e, n})  \right|^2 \frac1r
 + n^6 |\psi_{e, n}|^2 r \, dr
\leq C (\Phi |n| )^{\frac53}.
\ea
\ee

In conclusion, $\psi_n $ satisfies
\be \label{case4-13}
\int_0^1 \left|\frac{d}{dr}(r \psi_n) \right|^2 \frac1r \, dr + n^2 \int_0^1 | \psi_n|^2 r \, dr \leq C (\Phi |n| )^{-\frac43} \int_0^1 |\BF^*_n|^2 r \, dr
\ee
and
\be \label{case4-15}
\ba
&\int_0^1 \left| \frac{d}{dr}( r \mathcal{L} \psi_n ) \right|^2 \frac1r  + n^2   |\mathcal{L} \psi_n |^2 r +
n^4  \left| \frac{d}{dr}( r  \psi_n)  \right|^2 \frac1r  + n^6  | \psi_n |^2 r\, dr
\leq & C \int_0^1 | \BF^*_n |^2 r\, dr .
\ea
\ee
\end{pro}

Before the proof for Proposition \ref{case4}, we consider the linear problem \eqref{slip} again and give some new estimates when the slip coefficient is large.
\begin{lemma}\label{propslip-large}
Assume that $\alpha > 1 $. Given $f_n \in L_r^2(0, 1)$, the system \eqref{slip} admits a unique solution $\psi_{s, n} $  satisfying  the estimates

\be \label{propslip1-large-1}
 \int_0^1 \left| \frac{d}{dr} ( r \psi_{s, n} )  \right|^2 \frac1r \, dr + n^2 \int_0^1 | \psi_{s, n} |^2 r  \, dr \leq C (\Phi |n|)^{-\frac53}   \int_0^1 | f_n |^2 r \, dr ,
\ee
\be \label{propslip1-large-2} \ba
& \int_0^1 | \mathcal{L} \psi_{s, n} |^2 r  + n^2  \left|  \frac{d}{dr} ( r \psi_{s, n}  )\right|^2 \frac1r + n^4
 | \psi_{s, n} |^2  r \, dr
 \leq  C ( \Phi |n|)^{-\frac43} \int_0^1 |f_n |^2 r \, dr ,
 \ea
\ee
\be \label{propslip1-large-3} \ba
& \int_0^1 \left| \frac{d}{dr}( r \mathcal{L} \psi_{s, n} )\right|^2 \frac1r  +
n^2  |\mathcal{L} \psi_{s, n}  |^2 r  + n^4  \left|  \frac{d}{dr}( r \psi_{s, n} )       \right|^2 \frac1r
+ n^6 \int_0^1 |\psi_{s, n}  |^2  r \, dr\\
\leq &  C \left( \Phi |n|  \right)^{-\frac23} \int_0^1 |f_n |^2 r\, dr ,
\ea
\ee
and
\be \label{propslip1-large-4}
\left| \frac{d}{dr} \psi_{s, n} (1) \right| \leq C (\Phi |n|)^{-\frac34} \left( \int_0^1 |f_n |^2 r \, dr \right)^{\frac12}.
\ee
Moreover, if $f_n  = i n F^r_n  - \frac{d}{dr} F^z_n $, and $ F^r_n , F^z_n \in L_r^2(0,1)$, the solution $\psi_{s,n}$ satisfies the estimates
\eqref{propslip1-large-6}-\eqref{propslip1-large-9}.
\end{lemma}

\begin{proof} {\it Step 1. Proof of \eqref{propslip1-large-6}--\eqref{propslip1-large-9}.} Note that the proof of the estimate \eqref{case3-5} does not depend on $\alpha$, which thus also gives
 \eqref{propslip1-large-8}. Multiplying the equation in \eqref{slip} by $\mathcal{L} \overline{\psi_{s, n} } r$ and integrating over $[0, 1]$ give that
\be \label{large-11}
\begin{aligned}
& i \int_0^1 n^3 \frac{\bar{U}(r)}{r} \left|\frac{d}{dr} ( r \psi_{s, n} )   \right|^2 +  n
 \bar{U}(r) |\mathcal{L} \psi_{s, n} |^2 r  -  \frac{4\alpha \Phi}{(4 + \alpha)\pi} n^3
\left[ \frac{d}{dr}( r \overline{\psi_{s, n} }) r \psi_{s, n}   \right] \, dr\\
&+ \int_0^1 \left|\frac{d}{dr}(r\mathcal{L} \psi_{s, n} ) \right|^2 \frac1r
+2 n^2  \left|\mathcal{L} \psi_{s, n}  \right|^2 r
+ n^4  \left| \frac{d}{dr} ( r \psi_{s, n}  )\right|^2 \frac1r \, dr\\
= &   \int_0^1 f_n  \mathcal{L}\overline{\psi_{s,n} } r \, dr.
\end{aligned}
\ee
The imaginary part of \eqref{large-11} can be written as follows,
\be \label{large-12} \ba
 \int_0^1 n^3 \bar{U}(r)  \left| \frac{d}{dr} (r \psi_{s, n} )  \right|^2 \frac1r +n \bar{U}(r) |\mathcal{L} \psi_{s, n} |^2 r \, dr
=&   \Im \int_0^1 f_n \mathcal{L}\overline{\psi_{s,n} } r \, dr\\
  \leq& C \int_0^1 |\BF_n^*|^2 r \, dr,
\ea \ee
where the estimate \eqref{propslip1-large-8} has been used to get the last inequality.
Since $\bar{U}(r) \geq \frac{\Phi }{\pi } (1  - r^2) $, the above estimate \eqref{large-12}  implies
\be \label{large-15}
n^2 \int_0^1 (1 - r^2) \left| \frac{d}{dr} (r \psi_{s, n} )  \right|^2 \frac1r \, dr + \int_0^1 (1 - r^2)  |\mathcal{L} \psi_{s, n} |^2 r \, dr
\leq C (\Phi |n| )^{-1} \int_0^1 |\BF_n^*|^2 r \, dr .
\ee
Hence it follows from Lemmas \ref{weightinequality} and  \ref{lemma1}, \eqref{propslip1-large-8}, and  \eqref{large-15} that
\be \label{large-16} \ba
& n^2 \int_0^1  \left| \frac{d}{dr} (r \psi_{s, n} )  \right|^2 \frac1r \, dr + \int_0^1   |\mathcal{L} \psi_{s, n} |^2 r \, dr \\
\leq & C \left( n^2 \int_0^1 (1 - r^2) \left| \frac{d}{dr}(r \psi_{s, n}) \right|^2 \frac1r \, dr  \right)^{\frac23} \left( n^2 \int_0^1 |\mathcal{L} \psi_{s, n}|^2 r \, dr     \right)^{\frac13} \\
&\ \ \ + C \left(   \int_0^1 (1 - r^2) |\mathcal{L} \psi_{s, n}|^2 r \, dr     \right)^{\frac23} \left( \int_0^1 \left| \frac{d}{dr}(r \mathcal{L} \psi_{s, n} )  \right|^2 \frac1r \, dr   \right)^{\frac13} \\
\leq & C (\Phi |n| )^{-\frac23} \int_0^1 |\BF_n^*|^2 r \, dr .
\ea
\ee

Multiplying the equation in \eqref{slip}  by $n^2 \overline{\psi_{s, n}}r$ and integrating over $[0, 1]$ yield
\be \label{large-17}
n^3 \int_0^1 \bar{U} (r) \left| \frac{d}{dr} (r \psi_{s, n} )  \right|^2 \frac1r \, dr + n^5 \int_0^1 \bar{U}(r) |\psi_{s, n}|^2 r \, dr
= - n^2 \Im \int_0^1 f_n \overline{\psi_{s, n} } r\, dr.
\ee
Thus it follows from \eqref{propslip1-large-8} that
\be \label{large-18}
n^4 \int_0^1 (1 - r^2) |\psi_{s, n}|^2 r \, dr \leq C(\Phi |n|)^{-1} \int_0^1 |\BF_n^*|^2 r \, dr.
\ee
Hence applying Lemmas \ref{weightinequality} and  \ref{lemma1} again gives
\be \label{large-19}
n^4 \int_0^1 |\psi_{s, n}|^2 r \, dr  \leq C  (\Phi |n| )^{-\frac23} \int_0^1 |\BF_n^*|^2 r \, dr .
\ee
Combining \eqref{large-16} and \eqref{large-19} gives the estimate \eqref{propslip1-large-7}.

On the other hand, the equality \eqref{5-51} implies that
\be \label{large-20}
\Phi |n| \int_0^1 (1  - r^2) \left| \frac{d}{dr}( r \psi_{s, n} ) \right|^2 \frac1r \, dr + \Phi |n|^3 \int_0^1 (1 - r^2) |\psi_{s, n}|^2 r \, dr
\leq C \left| \int_0^1   f_n \overline{\psi_{s, n} }    r \, dr   \right|.
\ee
By Lemmas \ref{weightinequality} and \ref{lemma1}, and the estimate \eqref{large-16}, one has
\be \label{large-21} \ba
 & \int_0^1 \left| \frac{d}{dr} ( r\psi_{s, n})   \right|^2 \frac1r \, dr + n^2 \int_0^1 |\psi_{s, n}|^2 r \, dr \\
\leq & C \left(  \int_0^1 (1 - r^2) \left| \frac{d}{dr} (r \psi_{s, n})   \right|^2 \frac1r \, dr \right)^{\frac23} \left(  \int_0^1 |\mathcal{L} \psi_{s, n} |^2 r \, dr  \right)^{\frac13} \\
& + C \left(  n^2 \int_0^1 (1  - r^2) |\psi_{s, n}|^2 r \, dr \right)^{\frac23} \left(  n^2 \int_0^1 \left| \frac{d}{dr}( r \psi_{s, n})  \right|^2 \frac1r    \, dr \right)^{\frac13}\\
\leq & C (\Phi |n|)^{-\frac23} \left( \left| \int_0^1 f_n \overline{\psi_{s, n}} r\, dr       \right| \right)^{\frac23} (\Phi |n| )^{-\frac29} \left(         \int_0^1 |\BF_n^* |^2 r \, dr \right)^{\frac13}.
\ea \ee
Hence it holds that
\be \nonumber
\ba
&\int_0^1 \left| \frac{d}{dr} (r  \psi_{s, n})   \right|^2 \frac1r \, dr + n^2 \int_0^1 |\psi_{s, n}|^2 r \, dr\\
\leq  & C (\Phi |n| )^{-\frac89} \left(         \int_0^1 |\BF_n^* |^2 r \, dr \right)^{\frac23} \left( \int_0^1 \left| \frac{d}{dr} ( r \psi_{s, n})   \right|^2 \frac1r \, dr + n^2 \int_0^1 |\psi_{s, n}|^2 r \, dr \right)^{\frac13}.
\ea
\ee
This gives the estimate \eqref{propslip1-large-6}.

{\it Step 2. Proof of \eqref{propslip1-large-1}--\eqref{propslip1-large-4}.} The equalities  \eqref{B-51} and \eqref{B-55} give that
\be \label{large-23} \ba
& \int_0^1 \left| \frac{d}{dr} ( r \mathcal{L} \psi_{s, n} )  \right|^2 \frac1r  +  n^2 |\mathcal{L} \psi_{s, n} |^2 r
+ n^4 \left| \frac{d}{dr} (r \psi_{s, n} )     \right|^2 \frac1r
+ n^6  |\psi_{s, n} |^2 r \, dr \\
\leq &\left| \int_0^1 f_n (\mathcal{L} - n^2) \overline{\psi_{s, n} } r\, dr    \right|
\ea
\ee
and
\be \label{large-24}
\int_0^1 (1 - r^2) |( \mathcal{L} - n^2)  \psi_{s, n}|^2 r \, dr
\leq C (\Phi |n|)^{-1} \left| \int_0^1 f_n (\mathcal{L} - n^2) \overline{\psi_{s, n} } r\, dr    \right| .
\ee
The estimates \eqref{large-23}-\eqref{large-24}, together with Lemmas \ref{weightinequality} and  \ref{lemma1},  yield
\be \nonumber
\int_0^1 |(\mathcal{L} - n^2) \psi_{s, n} |^2 r \, dr
\leq C (\Phi |n|)^{-\frac23} \left| \int_0^1 f_n (\mathcal{L} - n^2) \overline{\psi_{s, n} } r\, dr    \right| .
\ee
By Young's inequality, one has
\be \label{large-26}
\int_0^1 |(\mathcal{L} - n^2) \psi_{s, n} |^2 r \, dr \leq C (\Phi |n| )^{-\frac43} \int_0^1 |f_n|^2 r \, dr .
\ee
With the aid of integration by parts, this is exactly the estimate \eqref{propslip1-large-2}.

Substituting \eqref{large-26} into \eqref{large-23} gives
\be \label{large-28} \ba
& \int_0^1 \left| \frac{d}{dr} ( r \mathcal{L} \psi_{s, n} )  \right|^2 \frac1r  +  n^2  |\mathcal{L} \psi_{s, n} |^2 r
+ n^4  \left| \frac{d}{dr} (r \psi_{s, n} )     \right|^2 \frac1r  + n^6 |\psi_{s, n} |^2 r \, dr \\
\leq\, & C (\Phi |n| )^{-\frac23} \int_0^1 |f_n|^2 r \, dr .
\ea
\ee

Furthermore, according to \eqref{propslip1-large-2} and \eqref{large-21}, one has
\be \label{large-29} \ba
&\int_0^1 \left|  \frac{d}{dr} ( r \psi_{s, n}) \right|^2 \frac1r \, dr + n^2 \int_0^1 |\psi_{s, n}|^2 r \, dr \\
\leq & C (\Phi |n|)^{-\frac23} \left| \int_0^1 f_n \overline{\psi_{s, n}} r \, dr  \right|^{\frac23} (\Phi |n| )^{-\frac49} \left(   \int_0^1 |f_n|^2 r \, dr    \right)^{\frac13} \\
\leq & C (\Phi |n| )^{-\frac{10}{9} } \left( \int_0^1 |f_n|^2 r \, dr      \right)^{\frac23} \left( \int_0^1 |\psi_{s, n}|^2 r \, dr  \right)^{\frac13}.
\ea \ee
This implies
\be \nonumber
\int_0^1 \left|  \frac{d}{dr} ( r \psi_{s, n}) \right|^2 \frac1r \, dr + n^2 \int_0^1 |\psi_{s, n}|^2 r \, dr  \leq C (\Phi |n|)^{-\frac53}
\int_0^1 |f_n|^2 r \, dr,
\ee
which is exactly \eqref{propslip1-large-1}. The inequality \eqref{propslip1-large-4} is the result of \eqref{propslip1-large-1}-\eqref{propslip1-large-2} and Lemma \ref{lemmaA2}.  Hence, the proof of Lemma \ref{propslip-large} is completed.
\end{proof}

Now we give the proof of Proposition \ref{case4}. The main idea is the same as that of Proposition \ref{Bpropcase3}, while the boundary layer function is different.

\begin{proof}[Proof of Proposition \ref{case4}]
Let $\psi_{s, n} $ denote the solution to \eqref{slip}. Define
\be \nonumber
\mathcal{A} = in \bar{U}(r) - \mathcal{L} +n^2, \ \ \ \ \ \ \ \ \ \  \mathcal{H} = \mathcal{L}- n^2,
\ee
\be \nonumber
\widetilde{\mcA_\infty} = i \frac{ \Phi n }{\pi} 4 (1 - r) - \frac{d^2}{dr^2} + n^2, \ \ \ \ \ \ \widetilde{\mathcal{H}} = \frac{d^2}{dr^2} - n^2.
\ee
$\widetilde{\mathcal{A_\infty}}$ can be regarded as the leading parts of the operator $\mcA$  near the boundary, when $|\alpha|$ is large.

We look for a boundary layer $\psi_{BL, n} $, which is a solution to
\be \nonumber
\widetilde{\mcA_\infty}\widetilde{\mcH} \psi_{BL, n } = 0.
\ee
First, consider the problem
\be \nonumber
\widetilde{\mcA_\infty} \phi = 0.
\ee
As discussed in \cite{M}, the operator $\widetilde{\mcA_\infty}$ can be written as  the Airy operator with  complex coefficients. Let the Airy function $Ai(z)$ denote the solution to
\be \nonumber
\frac{d^2 Ai}{dz^2} -  z Ai =0 \ \ \ \mbox{in}\ \mathbb{C}.
\ee
Define
\be \label{5-73-large}
\widetilde{G}_{n , \Phi} ( \rho) = \left\{  \ba  & Ai \left( C_{+} (\rho + \frac{\pi |\beta | n }{4 i \Phi }) \right), \ \ \ \mbox{if}\ n>0 , \\
 & Ai \left(  C_{-}  (\rho + \frac{\pi |\beta | n }{4 i \Phi })   \right),\ \ \ \ \mbox{if}\ n < 0 , \ea
\right.
\ee
where $ |\beta | = \left( \frac{ 4  \Phi |n| }{\pi}  \right)^{\frac13}$ and $C_{\pm} = e^{\pm i \frac{\pi}{6}}$.
Define
\be \label{5-74-large}
\widetilde{\widetilde{G}}_{n, \Phi}(r) = \widetilde{G}_{n, \Phi} (|\beta| ( 1  - r) ).
\ee
It is straightforward to check that
\be \nonumber
\widetilde{\mcA_\infty} \widetilde{\widetilde{G}}_{n, \Phi} = 0.
\ee

Without loss of generality, we assume that $n > 0$ from now on.
Next, define
\be \label{5-75-large}
G_{n, \Phi} (\rho) = \int_{\rho}^{+\infty} e^{- \frac{|n|}{|\beta|} ( \rho - \tau) } \int_{\tau}^{+ \infty}
e^{- \frac{|n|}{|\beta|} (\sigma - \tau)} \widetilde{G}_{n, \Phi} (\sigma) \, d\sigma d\tau .
\ee
It satisfies
\be \label{5-76-large}
\frac{d^2 G_{n, \Phi} }{d \rho^2} - \frac{|n|^2 }{|\beta|^2 } G_{n, \Phi} = \tilde{G}_{n, \Phi} (\rho).
\ee
Let
\be \label{5-77-large}
C_{0, n, \Phi} = \left\{  \ba & \frac{1}{G_{n, \Phi} (0) }, \ \ \ \mbox{if}\ |G_{n, \Phi} (0) | \geq 1,   \\ &
1,\ \ \ \ \ \ \ \ \ \ \ \ \mbox{otherwise}   \ea  \right.
\ee
and
\be \label{5-78-large}
\psi_{BL, n}  (r) : = C_{0, n, \Phi} G_{n, \Phi} (|\beta| ( 1  - r) ).
\ee
The straightforward computations show that
\be \label{5-79-large}
\widetilde{\mcA_\infty}\widetilde{\mcH} \psi_{BL, n} = 0, \ \ \ \text{for } r\in (0,1) ,
\ee
and $|\psi_{BL, n}(1)| \leq 1. $

Now we are ready to construct the remainder term $\psi_{e, n}$ such that
\be \label{5-81-large}
\left\{  \ba
&  i n \bar{U}(r) ( \mathcal{L} - n^2) \psi_{e, n} - (\mathcal{L} - n^2)^2 \psi_{e, n}  = \widetilde{\mcA_\infty}\widetilde{\mcH} ((1 - \chi) \psi_{BL, n} ) + (\widetilde{\mcA_\infty}\widetilde{\mcH} - \mcA\mcH) ( \chi \psi_{BL, n} ) ,\\
&  \psi_{e, n} (0) = \psi_{e, n} (1) = \mathcal{L} \psi_{e, n} (0) = \mathcal{L} \psi_{e, n} (1) = 0,
\ea
\right.
\ee
where $\chi$ is the smooth cut-off function satisfying \eqref{defchi}.

Denote
\be \nonumber
\mcA_{\infty} = i \frac{2 \Phi n}{\pi} (1  -  r^2) - \mathcal{L} + n^2.
\ee
Then
\be \label{5-83-large} \ba
& (\mcA_{\infty}\mcH - \mcA \mcH) (\chi \psi_{BL, n} )   =i \frac{  \Phi n }{\pi} \frac{4}{4 + \alpha} (1  - 2r^2) \mcH (\chi \psi_{BL, n})\\
= & \frac{d}{dr} \left[ i \frac{\Phi n}{\pi} \frac{4}{4 + \alpha} (1 - 2r^2) \frac{d}{dr} (\chi \psi_{BL, n})    \right]
+ i \frac{\Phi n}{\pi }\frac{16r}{4 + \alpha} \frac{d}{dr} ( \chi \psi_{BL, n})  \\
&\ \ \ + i \frac{\Phi n }{\pi} \frac{4}{4 + \alpha} (1  - 2r^2)\left(\frac1r \frac{d}{dr} - \frac{1}{r^2}\right)(\chi \psi_{BL, n})
- i \frac{\Phi n^3}{\pi}\frac{4}{4 + \alpha} (1 - 2r^2) (\chi \psi_{BL, n}) \\
= & - \frac{d}{dr} F_{e,n}^z +in F_{e,n}^r +f_{e,n},
\ea \ee
where
\be \label{5-85-large}
F^r_{e, n} = - \frac{ \Phi n^2 }{\pi} \frac{4}{4 + \alpha} (1-  2r^2) (\chi \psi_{BL, n}), \ \ F^z_{e, n}= - \frac{i  \Phi n }{\pi} \frac{4}{4 + \alpha} (1 - 2r^2) \frac{d}{dr} (\chi \psi_{BL, n}),
\ee
and
\be
f_{e, n}= \frac{i \Phi n }{\pi}\frac{16r}{4 +  \alpha} \frac{d}{dr} (\chi \psi_{BL, n}) +i \frac{\Phi n }{\pi} \frac{4}{4 + \alpha} (1  - 2r^2)\left(\frac1r \frac{d}{dr} - \frac{1}{r^2}\right)(\chi \psi_{BL, n}).
\ee
Note that $0 < \epsilon_1 < 1$ and $|n| \leq \epsilon_1 \sqrt{\Phi}$, it is easy to see $|n|\leq c |\beta|$. Thus one has
\be \label{5-87-large} \ba
\int_0^1 | F^r_{e, n} |^2  r \, dr
& \leq C \frac{\Phi^2 n^4}{(4 + \alpha)^2 } \int_0^1 \left| (1 -  2r^2) (\chi \psi_{BL, n} )\right|^2 r\, dr \\
& \leq C \frac{\Phi^2 n^4}{(4 + \alpha)^2 } \int_{\frac14}^1 |\psi_{BL, n}|^2 \, dr \\
& \leq C \frac{\Phi^2 n^4}{(4 + \alpha)^2 }  |\beta|^{-1}  \leq C \frac{ \Phi^2 n^2}{(4+  \alpha)^2} |\beta| . \ea
\ee
Similarly, it holds that
\be \label{5-87-1-large}
\ba
\int_0^1 |F_{e, n}^z|^2 r \, dr  \leq C \frac{\Phi^2 n^2}{(4 + \alpha)^2 } \int_0^1 \left| (1 - 2r^2) \frac{d}{dr}(\chi \psi_{BL, n})  \right|^2 r \, dr
\leq C \frac{\Phi^2 n^2}{(4 + \alpha)^2} |\beta|
\ea
\ee
and
\be \label{5-88-large}
\int_0^1 | f_{e, n} |^2 r \, dr \leq C \frac{ \Phi^2 n^2 }{(4 + \alpha)^2} |\beta| .
\ee
Meanwhile, as same as the computations in  \cite[Section 4]{WX1}, one has
\be \label{5-88-1-large}
\int_0^1 |  (\widetilde{\mcA_\infty} \widetilde{\mcH} - \mcA_{\infty} \mcH ) ( \chi \psi_{BL, n})  |^2 r \, dr + \int_0^1 |\widetilde{\mcA_\infty} \widetilde{\mcH} ((1 - \chi ) \psi_{BL, n} )|^2 r \, dr \leq C |\beta|^5.
\ee

Hence it follows from Lemma \ref{propslip-large} that one has
\be \label{5-89-large}\ba
\int_0^1 \left| \frac{d}{dr} (r \psi_{e, n} )  \right|^2 \frac1r  + n^2 | \psi_{e, n} |^2 r \, dr
\leq & C (\Phi |n| )^{-\frac53} |\beta|^5 + C (\Phi |n| )^{-\frac43} \frac{ (\Phi |n|)^2 }{(4 + \alpha)^2} |\beta|\\
\leq & C (\Phi |n| )^{\frac13}
\ea \ee
and
\be \label{5-90-large} \ba
 &\int_0^1 | \mathcal{L} \psi_{e, n}  |^2 r + n^2  \left| \frac{d}{dr} ( r  \psi_{e, n} ) \right|^2 \frac1r  + n^4  |\psi_{e, n}|^2 r\, dr\\
\leq &  C  (\Phi |n| )^{-\frac43} |\beta|^5 + C (\Phi |n|)^{-\frac23} \frac{(\Phi |n|)^2}{(4 + \alpha)^2} |\beta|
 \leq  C \Phi |n| .
\ea \ee
Moreover, it holds that
\be \label{5-91-large}\ba
& \int_0^1 \left| \frac{d}{dr} ( r \mathcal{L} \psi_{e, n} ) \right|^2 \frac1r  + n^2  |\mathcal{L} \psi_{e, n} |^2 r  + n^4  \left| \frac{d}{dr} ( r \psi_{e, n} ) \right|^2 \frac1r  + n^6  | \psi_{e, n} |^2 r\, dr \\
\leq & C (\Phi |n|)^{-\frac23} |\beta|^5 + C \frac{(\Phi |n|)^2}{(4 + \alpha)^2} |\beta| \\
\leq & C (\Phi |n| )^{\frac53}
\ea
\ee
and
\be \label{5-92-large}
\begin{aligned}
\left| \frac{d}{dr} \psi_{e, n}  (1) \right|
 \leq & C (\Phi |n|)^{-\frac34} |\beta|^{\frac52} + C  \frac{(\Phi |n|)^{\frac12} }{4 +  \alpha }|\beta|^{\frac12} .
\end{aligned}
\ee

Finally, as in the proof of Proposition \ref{Bpropcase3} , choose the constants $a$ and $b$ to be of the form \eqref{5-88}, i. e.,
 \be \nonumber
 b= - \frac{\alpha \frac{d}{dr} \psi_{s, n}(1)}{J}, \ \ \ \ \ \ \ \ a= - \frac{b \psi_{BL, n} (1) }{I_1 (|n|) },
 \ee
 where $J$ is defined in \eqref{5-89}. Hence, \eqref{5-87} is satisfied.
Before the estimates of $a$ and $b$, let us insert one lemma, which gives the properties of Airy function and also the estimate of $C_{0, n, \Phi}$. The proof of the lemma  is exactly the same as that of  \cite[Lemma 3.7]{M}, so we omit the details here.
\begin{lemma}\label{bound} $(1)$\ It holds that
\be \nonumber
\tilde{C_0} : = \inf \left\{  | C_{0, n, \Phi} | :\ \Phi\geq 1, \ 1 \leq |n| \leq  \sqrt{\Phi}    \right\} >0 ,
\ee
and the function $G_{n, \Phi}$ defined in \eqref{5-75-large} satisfies for $\varsigma$ large enough
\be \label{5-80-large}
\sup_{\Phi \geq 1 } \sup_{ 1 \leq |n| \leq  \sqrt{\Phi} } \sup_{\rho \geq \varsigma}  e^{\rho}
\left| \frac{d^k G_{n, \Phi}}{d \rho^k} (\rho) \right| < \infty, \ \ \ k =0, 1, 2, 3.
\ee

$(2)$\  There is a constant $\epsilon \in (0, 1)$ such that
\be \nonumber
\Sigma_{\epsilon} : =  \{  \mu \in \mathbb{C}\, | \, arg \mu = -\frac{\pi}{6} ,\  0\leq |\mu| \leq \epsilon \},
\ee
then
\be \nonumber
K_{\epsilon} : = \inf_{\mu \in \Sigma_{\epsilon}}  \Re \left[ C_{-}  \int_0^{+ \infty} e^{- \mu s } Ai (s + \mu^2) \, ds \right] \geq \frac16.
\ee

\end{lemma}

Note that
\be \nonumber \ba
& \frac{d^2}{dr^2} \psi_{BL, n}(1) + (1  +  \alpha) \frac{d}{dr} \psi_{BL, n}(1) \\
= & C_{0, n, \Phi} \left[ |n|^2 G_{n, \Phi} (0) + |\beta|^2 Ai \left(C_+ \frac{\pi |\beta| n}{4 i \Phi} \right) \right]\\
&+ (1 + \alpha ) C_{0, n, \Phi} \left[ |\beta| C_{-} \int_0^\infty e^{-\mu z}Ai ( z + \mu^2) \, dz + |n| G_{n, \Phi} (0) \right],
\ea \ee
where $\mu = \frac{|n|}{|\beta| } C_{-}$ and $ |\mu| = \left( \frac{\pi}{4} \right)^{\frac13} |n|^{\frac23} \Phi^{-\frac13} \leq \left( \frac{\pi}{4} \right)^{\frac13} \epsilon_1^{\frac23} $.
If $\epsilon_1$ is small enough, according to Lemma \ref{bound},
\be \nonumber \ba
& \left| \frac{d^2}{dr^2} \psi_{BL, n}(1) + (1  + \alpha) \frac{d}{dr} \psi_{BL, n}(1) \right|\\
\geq & |C_{0, n, \Phi} | \left[ \frac12 |\beta|^2 Ai(0) + \frac16 (1 + \alpha) |\beta|    \right] - n^2 - (1 + \alpha) |n|\\
 \geq &  \frac12 \tilde{C_0} |\beta|^2 Ai (0) + \frac16 \tilde{C_0} (1 +  \alpha) |\beta| - n^2 - (1 + \alpha) |n| \\
 \geq &\frac14 \tilde{C_0} |\beta|^2 Ai (0) + \frac{1}{12} \tilde{C_0} (1 +  \alpha) |\beta|.
\ea
\ee
On the other hand, it follows from Lemma \ref{lemBessel} and \eqref{5-92-large} that
\be \nonumber
\left| \psi_{BL, n} (1) \left[ |n|^2 + \alpha |n| \frac{I_1^{\prime}(|n|)}{I_1 (|n|)} \right] \right|
\leq |n|^2 + |\alpha |  (1 + |n|)
\ee
and
\be \nonumber
\left| \psi_{BL, n} (1) -  \alpha \frac{d}{dr} \psi_{e, n} (1)  \right| \leq 1 + C |\alpha| (\Phi|n|)^{-\frac34} |\beta|^{\frac52}
+ C \frac{|\alpha|}{4 + \alpha}(\Phi |n|)^{\frac12} |\beta|^{\frac12}.
\ee
If $4+ \alpha \geq \delta^{-1} (\Phi |n|)^{\frac13}$ and $\delta $ is small enough, then
\be \nonumber
J \geq \frac{1}{24}\tilde{C_0} (1 +  \alpha) |\beta|.
\ee
Therefore, by Lemma \ref{propslip-large}, one has
\be \label{5-100-large} \ba
|b| & \leq C|\alpha| \left| \frac{d}{dr} \psi_{s, n} (1) \right| (1 +  \alpha)^{-1} |\beta|^{-1}
\leq C (\Phi |n|)^{-\frac12} |\beta|^{-1} \left( \int_0^1 |\BF^*_n |^2 r \, dr \right)^{\frac12} \\
& \leq C (\Phi |n|)^{-\frac56}\left( \int_0^1 |\BF^*_n |^2 r \, dr \right)^{\frac12}
\ea
\ee
and
\be \label{5-101-large}
|a| \leq C  (\Phi |n|)^{-\frac56}  I_1(|n|)^{-1} \left( \int_0^1 |\BF^*_n |^2 r \, dr \right)^{\frac12} .
\ee
Consequently, it holds that
\be \label{5-102-large-1} \ba
 & b^2 \int_0^1 \left| \frac{d}{dr}(r \chi \psi_{BL, n}) \right|^2 \frac1r \, dr + b^2 n^2 \int_0^1 |\chi \psi_{BL, n}|^2 r \, dr \\
 \leq  & C (\Phi |n|)^{-\frac53}   |\beta| \int_0^1 |\BF_n^*|^2 r \, dr
  \leq C (\Phi |n| )^{-\frac43} \int_0^1 |\BF^*_n|^2 r \, dr
 \ea
\ee
and
\be \label{5-102-large} \ba
& b^2 \int_0^1 \left|  \frac{d}{dr} ( r \mathcal{L} (\chi \psi_{BL, n} )) \right|^2 \frac1r \, dr + b^2 n^2 \int_0^1 \left|\mathcal{L} (\chi \psi_{BL, n} ) \right|^2 r \, dr \\ & \ \ \  +   b^2  n^4 \int_0^1 \left| \frac{d}{dr} ( r \chi \psi_{BL, n} )\right|^2 \frac1r \, dr   + b^2 n^6 \int_0^1 \left| \chi \psi_{BL, n} \right|^2 r \, dr \\
 \leq & C (\Phi |n| )^{-\frac53 } (|\beta|^5 + n^6 |\beta|^{-1})
\int_0^1 |\BF^*_n |^2 r \, dr \\
 \leq & C \int_0^1 |\BF^*_n|^2 r \, dr.
\ea \ee
Furthermore, one has
\be \label{5-103-large-1}
\ba
& a^2 \int_0^1 \left| \frac{d}{dr} (r I_1 (|n| r) )  \right|^2 \frac1r \, dr  + a^2 n^2 \int_0^1 |I_1(|n| r )|^2 r \, dr \\
\leq & C (\Phi |n|)^{-\frac53 }  |n| \int_0^1 |\BF^*_n|^2 r \, dr \leq C (\Phi |n| )^{-\frac43} \int_0^1 |\BF^*_n|^2 r \, dr
\ea
\ee
and
\be \label{5-103-large} \ba
& a^2 \int_0^1 \left|  \frac{d}{dr}(r \mathcal{L} I_1(|n|r)) \right|^2 \frac1r \, dr + a^2 n^2 \int_0^1 |\mathcal{L} I_1 (|n| r)|^2 r \, dr \\
&\ \ + a^2 n^4 \int_0^1 \left| \frac{d}{dr} (r I_1(|n| r )) \right|^2 \frac1r \, dr + a^2 n^6 \int_0^1 |I_1(|n|r)|^2 r \, dr \\
\leq & C   (\Phi |n|)^{-\frac53 }  |n|^5 \int_0^1 |\BF^*_n|^2 r \, dr \\
\leq & C \int_0^1 |\BF^*_n|^2  r \, dr.
\ea \ee

Combining the estimates in Lemma \ref{propslip-large} and \eqref{5-102-large-1}-\eqref{5-103-large} completes the proof of Proposition \ref{case4}.
\end{proof}


\subsubsection{The case with large flux, intermediate frequency, and intermediate slip coefficient}
\begin{pro}\label{case5}
Assume that $\Phi \gg 1$ and $\epsilon_1$, $\delta$ are two independent constants in $(0, 1)$. As long as $n$ satisfies
$1 \leq |n| \leq \epsilon_1 \sqrt{\Phi}$ and $  \delta (\Phi |n|)^{\frac13} \leq 4 +  \alpha \leq \delta^{-1} (\Phi |n|)^{\frac13}$, the solution $\psi_{n}$ to the problem \eqref{2-0-8} satisfies
\be \label{6-0}
 \int_0^1 |\psi_n|^2 r \, dr \leq C (\Phi |n|)^{-\frac53} \int_0^1 |\BF^*_n|^2 r \, dr,
\ee
\be \label{6-1}
\int_0^1 \left| \frac{d}{dr}(r \psi_n) \right|^2 \frac1r \, dr + n^2 \int_0^1 |\psi_n|^2 r \, dr \leq C (\Phi |n| )^{-\frac43} \int_0^1 |\BF^*_n|^2 r \, dr ,
\ee
\be \label{6-2}
|n| \int_0^1  \bar{U}(r) \left| \frac{d}{dr} ( r \psi_n)\right|^2 \frac1r \, dr + |n|^3 \int_0^1 \bar{U}(r) |\psi_n|^2 r \, dr \leq C (\Phi |n|)^{-\frac23} \int_0^1 |\BF^*_n|^2 r \, dr,
\ee
and
\be \label{6-3} \ba
& \int_0^1 |\mathcal{L} \psi_n|^2 r  + n^2  \left|\frac{d}{dr} ( r\psi_n)  \right|^2 \frac1r  + n^4  |\psi_{n} |^2 r \, dr
+  \alpha \left| \frac{d}{dr} ( r \psi_n) (1) \right|^2 \\
\leq &
C (\Phi |n|)^{-\frac12} \int_0^1 |\BF^*_n|^2 r \, dr,
\ea
\ee
\end{pro}
\begin{proof}
According to \eqref{2-1-7}, it holds that
\be \label{6-5}
n \int_0^1 \frac{\bar{U}(r)}{r} \left|\frac{d}{dr}(r \psi_n)  \right|^2 \, dr + n^3 \int_0^1 \bar{U} (r) |\psi_n|^2 r \, dr
= - \Im \int_0^1 f_n \overline{\psi_n} r \, dr.
\ee
Note that
\be \nonumber
\bar{U}(r) = \frac{2\Phi}{\pi} (1 - r^2) \frac{ \alpha}{4 + \alpha} + \frac{\Phi}{\pi} \frac{4}{4 + \alpha} \geq \frac{4\Phi}{\pi} \delta (\Phi |n| )^{-\frac13}.
\ee
Hence one has
\be \label{6-6} \ba
& (\Phi |n| )^{\frac23} \left( \int_0^1 \left| \frac{d}{dr}(r \psi_n) \right|^2 \frac1r \, dr + n^2 \int_0^1 |\psi_n|^2r \, dr \right) \\
\leq & C (\delta) \left[  \int_0^1 |F_n^r| |n \psi_n| r \, dr + \int_0^1 |F^z_n| \left| \frac{d}{dr}(r \psi_n)    \right| \, dr              \right].
\ea \ee
This, together with Schwarz inequality, gives \eqref{6-1}.
Taking \eqref{6-1} into \eqref{6-5} proves \eqref{6-2}.

Moreover, due to the fact that $\bar{U}(r) \geq \frac{\Phi}{\pi} (1- r^2)$ and Lemma \ref{weightinequality}, one has
\be \label{6-8}
\int_0^1 |\psi_n|^2 r \, dr \leq
C\Phi^{-1}  \int_0^1 \frac{\bar{U}(r)}{r} \left| \frac{d}{dr}(r \psi_n)  \right|^2 \, dr  \leq C(\delta) (\Phi |n| )^{-\frac53} \int_0^1 |\BF^*_n|^2 r \, dr.
\ee
This is exactly the estimate \eqref{6-0}.

On the other hand, it follows from \eqref{2-1-5}, \eqref{6-1}, and \eqref{6-8} that
\be \label{6-9} \ba
& \int_0^1 |\mathcal{L} \psi_n|^2 r \, dr +  2n^2 \int_0^1 \left| \frac{d}{dr}(r \psi_n) \right|^2 \frac1r \, dr + n^4 \int_0^1 |\psi_n|^2 r \, dr \\
\leq & \int_0^1 |F_n^r| |n \psi_n| r \, dr + \int_0^1 |F_n^z| \left| \frac{d}{dr}(r \psi_n)  \right| \, dr + \frac{4\Phi |n| }{\pi}
\int_0^1 \left| \frac{d}{dr} (r \psi_n)  \right| | r \psi_n | \, dr \\
\leq & C (\Phi |n| )^{-\frac23} \int_0^1 |\BF^*_n|^2 r \, dr + C \Phi |n| \left( \int_0^1 \left| \frac{d}{dr}(r \psi_n) \right|^2 \frac1r \, dr \right)^{\frac12} \left(
\int_0^1 |\psi_n|^2 r \, dr \right)^{\frac12}\\
\leq & C (\Phi |n|)^{-\frac12} \int_0^1 |\BF^*_n|^2 r \, dr .
\\
\ea \ee
This completes the proof of Proposition \ref{case5}.
\end{proof}


Let
\be \nonumber
Z_1 = \{ n \in \mathbb{Z}:\ 1 \leq |n| \leq \epsilon_1 \sqrt{\Phi}, \ |n| \geq \delta^{-3} (4 + \alpha)^3 \Phi^{-1}        \} ,
\ee
\be \nonumber
Z_2 = \{ n \in \mathbb{Z}:\ 1\leq |n| \leq \epsilon_1 \sqrt{\Phi},\  |n| \leq \delta^3 (4 + \alpha)^3 \Phi^{-1}         \},
\ee
\be \nonumber
Z_3 = \{ n \in \mathbb{Z}:\ 1 \leq |n| \leq \epsilon_1 \sqrt{\Phi},\ \delta^3 (4 + \alpha)^3 \Phi^{-1} < |n| < \delta^{-3} (4  +  \alpha)^3 \Phi^{-1} \}.
\ee
Denote
\be \nonumber
\psi_{med, j} = \sum_{n \in Z_j} \psi_n e^{inz},\ \ \ j=1, 2, 3
\ee
and define
\be \nonumber
v^r_{med, j}= \partial_z \psi_{med, j},\ \ \ v^z_{med, j} = -\frac{\partial_r ( r \psi_{med, j}) }{r}, \ \ \ \text{and}\ \ \Bv_{med, j}^* = v^r_{med, j} \Be_r + v^z_{med, j} \Be_z.
\ee
Furthermore,  $F^r_{med, j}$, $F^z_{med, j}$, $\BF^*_{med, j}$, and $\Bo_{med, j}^\theta $ can be defined similarly.

\begin{pro}\label{medregularity1}
The solution $\Bv^*$ satisfies
\be \label{7-1}
\|\Bv^*_{med, 1} \|_{H^2(\Omega)} \leq C \|\BF^*_{med, 1}\|_{L^2(\Omega)} \ \ \ \text{and}\ \  \  \|\Bv^*_{med, 2} \|_{H^2(\Omega)} \leq C \|\BF^*_{med, 2}\|_{L^2(\Omega)},
\ee
where the constant $C$ is a uniform constant independent of $\Phi$, $\alpha$, and $\BF$.
\end{pro}

\begin{proof}
It follows from Proposition \ref{Bpropcase3} that
\be \label{7-2}
\|\Bv_{med, 1}^* \|_{L^2(\Omega)}^2 \leq C \sum_{n \in Z_1 } \int_0^1 \left( n^2 |\psi_n|^2 r + \left| \frac{d}{dr}(r \psi_n) \right|^2 \frac1r   \right) \, dr \leq C \|\BF_{med, 1}^*\|_{L^2(\Omega)}^2.
\ee
Similarly, one has
\be \label{7-3-0} \ba
\|\Bv_{med, 1}^*\|_{H^1(\Omega)}^2 & =  \|\omega^\theta_{med, 1} \|_{L^2(\Omega)}^2 +  \|\Bv^*_{med, 1 }\|_{L^2(\Omega)}^2 \\
& \leq C \sum_{n \in Z_1} \int_0^1 | (\mathcal{L} - n^2 ) \psi_n  |^2 r \, dr + C  \|\Bv^*_{med, 1 }\|_{L^2(\Omega)}^2\\
& \leq C \|\BF^*_{med, 1} \|_{L^2(\Omega)}^2
\ea \ee
and
\be \label{7-5-1} \ba
\|\partial_z \Bv^*_{med, 1} \|_{H^1(\Omega)}^2 & =  \|\partial_z \omega^\theta_{med, 1} \|_{L^2(\Omega)}^2  +  \| \partial_z \Bv^*_{med, 1}\|_{L^2(\Omega)}^2 \\
& \leq C  \sum_{n \in Z_1} \int_0^1 n^2 | (\mathcal{L} - n^2 ) \psi_n  |^2 r \, dr + C  \|\BF^*_{med, 1 }\|_{L^2(\Omega)}^2\\
& \leq C \|\BF^*_{med, 1} \|_{L^2(\Omega)}^2 .
\ea
\ee
It follows from the estimates \eqref{7-3-0}-\eqref{7-5-1} and the trace theorem for axisymmetric functions that
\be \label{7-6}
\|\Bv^*_{med, 1} \|_{H^{\frac32} (\partial \Omega)} \leq C \|\Bv^*_{med, 1} \|_{H^1(\Omega)} + C \|\partial_z \Bv^*_{med, 1}\|_{H^1(\Omega)}
\leq C \|\BF^*_{med, 1} \|_{L^2(\Omega)}.
\ee

On the other hand, the straightforward computations give
\be \label{7-9-0}
\Delta \Bv^*_{med, 1 } = -{\rm curl}~\Bo^\theta_{med, 1} = \partial_z \omega^\theta_{med, 1} \Be_r - \left( \partial_r \omega^\theta_{med, 1} + \frac{\omega^\theta_{med, 1}}{r} \right)\Be_z.
\ee
Thus according to Proposition \ref{Bpropcase3}, one has
\be \label{7-11} \ba
\left\| \partial_r \omega^\theta_{med, 1} + \frac{\omega^\theta_{med, 1}}{r}   \right\|_{L^2(\Omega)}^2
& = \left\| \frac1r \frac{\partial}{\partial r} ( r \omega^\theta_{med, 1} )          \right\|_{L^2(\Omega)}^2 \\
& \leq C \sum_{n \in Z_1}  \int_0^1 \left| \frac{d}{d r} [ r ( \mathcal{L} - n^2) \psi_n ]   \right|^2
\frac1r \, dr  \\
& \leq C\|\BF^*_{med, 1 } \|_{L^2(\Omega)}^2 .
\ea
\ee
Applying the regularity theory for the elliptic equation \eqref{7-9-0}  yields that  $\Bv^*_{med, 1}$ satisfies
\be \label{7-15}
\| \Bv^*_{med, 1} \|_{H^2(\Omega)}
\leq C \|{\rm curl}~\Bo_{med, 1}^\theta \|_{L^2(\Omega)} + C \|\Bv^*_{med, 1} \|_{H^{\frac32}(\partial \Omega)} \leq C \|\BF^*_{med, 1} \|_{L^2(\Omega)}.
\ee

Following the same lines as above and applying Proposition \ref{case4}, one can prove that
\be \label{7-16}
\| \Bv^*_{med, 2} \|_{H^2(\Omega)} \leq C \|\BF^*_{med, 2} \|_{L^2(\Omega)}.
\ee
This finishes the proof of Proposition \ref{medregularity1}.
\end{proof}

\begin{pro}\label{medregularity2}
The solution $\Bv^*$ satisfies
\be \label{8-1}
\|\Bv^*_{med, 3} \|_{H^{\frac32} (\Omega)} \leq C \|\BF^*_{med, 3} \|_{L^2(\Omega)} \ \ \mbox{and}\  \
\|\Bv^*_{med, 3} \|_{H^{2} (\Omega)} \leq C \Phi^{\frac14} \|\BF^*_{med, 3} \|_{L^2(\Omega)}.
\ee
where the constant $C$ is independent of $\Phi$, $\alpha$, and $\BF^*$.
\end{pro}
\begin{proof}
In fact, following similar proof of Proposition \ref{back}, it can be proved that $\Bv^*_{med, 3} $ satisfies the following equation
\be \label{8-2}
\left\{
\ba
& - \Delta \Bv^*_{med, 3} + \nabla P = \BF^*_{med, 3}- \bar{U} \partial_z \Bv^*_{med, 3}  - v^r_{med, 3} \partial_r \bBU\ \ \ \ \mbox{in}\ \Omega,  \\
& {\rm div}~ \Bv^*_{med, 3} = 0\ \ \ \ \mbox{in}\ \Omega.
\ea
\right.
\ee
It follows from the regularity theory for Stokes equations that
\be \label{8-3}
\ba
\|\Bv^*_{med, 3} \|_{H^2(\Omega)}
 \leq & C \|\BF^*_{med, 3} \|_{L^2(\Omega)} + C \|\bar{U} \partial_z \Bv^*_{med, 3} \|_{L^2(\Omega)} + C \|v^r_{med, 3}  \partial_r \bar{U} \|_{L^2(\Omega)}\\
 &\ \ \ +C \|\Bv^*_{med, 3} \|_{H^1(\Omega)} + C \|\Bv^*_{med, 3}\|_{H^{\frac32}(\partial \Omega)} \\
 \leq &  C \|\BF^*_{med, 3} \|_{L^2(\Omega)} + C \|\bar{U} \partial_z \Bv^*_{med, 3} \|_{L^2(\Omega)} + C \|v^r_{med, 3}  \partial_r \bar{U} \|_{L^2(\Omega)}\\
 &\ \ \ + C \|\Bv^*_{med, 3} \|_{H^1(\Omega)} + C \|\partial_z \Bv^*_{med, 3}\|_{H^1 ( \Omega)}.
\ea
\ee
Herein, according to Proposition \ref{case5}, one has
\be \label{8-5} \ba
& \|\bar{U} \partial_z \Bv^*_{med, 3} \|_{L^2(\Omega)}^2 \\
\leq & C \Phi  \sum_{n \in Z_3} \left[ n^2 \int_0^1 \bar{U}(r) \left| \frac{d}{dr}(r \psi_n) \right|^2 \frac1r \, dr  +
n^4 \int_0^1 \bar{U}(r) |\psi_n|^2 r \, dr \right]   \\
\leq & C \Phi \sum_{n \in Z_3} |n| (\Phi |n| )^{-\frac23} \int_0^1 |\BF^*_n|^2 r \, dr \\
\leq & C \Phi^{\frac12} \|\BF^*_{med, 3} \|_{L^2(\Omega)}^2
\ea \ee
and
\be \label{8-6} \ba
& \| v^r_{med, 3} \partial_r \bar{U} \|_{L^2(\Omega)}^2
\leq  C \Phi^2 \sum_{n \in Z_3} n^2 \int_0^1 |\psi_n|^2 r \, dr \\
\leq & C \Phi^2 \sum_{n \in Z_3} (\Phi |n|)^{-\frac53} n^2 \int_0^1 |\BF^*_n|^2 r \, dr \\
\leq & C \sum_{n \in Z_3} (\Phi |n| )^{\frac13} \int_0^1 |\BF^*_n|^2 r \, dr \\
\leq & C \Phi^{\frac12} \|\BF^*_{med, 3} \|_{L^2(\Omega)}^2.
\ea \ee
Similarly, it holds that
\be  \label{8-7}
\ba
& \|\Bv^*_{med, 3} \|_{H^1(\Omega)}^2
=  \|\Bo^\theta_{med, 3}\|_{L^2(\Omega)}^2 + \|\Bv^*_{med, 3}\|_{L^2(\Omega)}^2 \\
\leq & C \sum_{n \in Z_3} \left[ \int_0^1 |(\mathcal{L} - n^2) \psi_n|^2 r \, dr +  \int_0^1 \left|\frac{d}{dr} (r \psi_n)  \right|^2 \frac1r + n^2 |\psi_n|^2 r \, dr  \right]\\
\leq & C \Phi^{-\frac12} \|\BF^*_{med, 3} \|_{L^2(\Omega)}^2
\ea
\ee
and
\be \label{8-8}
\ba
 \|\partial_z \Bv^*_{med, 3} \|_{H^1(\Omega)}^2
= & \|\partial_z \Bo^\theta_{med, 3} \|_{L^2(\Omega)}^2  +  \|\partial_z \Bv^*_{med, 3} \|_{L^2 (\Omega)}^2  \\
\leq & C \sum_{n \in Z_3}  \int_0^1 n^2 |(\mathcal{L} - n^2) \psi_n|^2 r \, dr   + C \|\Bv^*_{med, 3} \|_{H^1(\Omega)}^2 \\
\leq & C \sum_{n \in Z_3} n^2 (\Phi |n|)^{-\frac12} \|\BF^*_n\|_{L^2(\Omega)}^2 + C \|\Bv^*_{med, 3} \|_{H^1(\Omega)}^2\\
\leq & C \Phi^{\frac14} \|\BF^*_{med, 3}\|_{L^2(\Omega)}^2.
\ea
\ee
Taking \eqref{8-5}--\eqref{8-8} into \eqref{8-3} yields
\be \label{8-9}
\|\Bv^*_{med, 3} \|_{H^2(\Omega)}\leq C \Phi^{\frac14} \|\BF^*_{med, 3}\|_{L^2(\Omega)}.
\ee
The interpolation for the estimates \eqref{8-7} and \eqref{8-9} gives \eqref{8-1} so that the proof of Proposition \ref{medregularity2} is completed.
\end{proof}

\begin{remark}
The estimate \eqref{8-7} gives some good control of $\|\Bv^*_{med, 3}\|_{H^1(\Omega)}$ when $\Phi$ is large.  During the proof of \eqref{8-7}, we used the fact that $|n| \geq 1$. The fact is only true for periodic pipes. That is why we can only consider the periodic case in this paper.
\end{remark}


\section{Analysis on the linearized problem for swirl velocity}\label{sec-swirl}

In this section, we give prove the existence of solutions to the linearized problem \eqref{vswirl} and give the uniform estimates for $\Bv^\theta$ defined by $\Bv^\theta = v^\theta \Be_\theta$.

\begin{pro}\label{swirl}
Assume that $\BF^\theta = F^\theta (r, z) \Be_\theta \in L^2 (\Omega)$.  If $F^\theta$ satisfies the following compatibility condition,
\be \label{compatibility}
2\pi \int_0^1 F_0^\theta(r) r^2 \, dr = \int_{-\pi}^{\pi} \int_0^1 F^\theta(r, z) r^2 \, dr dz =  0,
\ee
then the linear problem \eqref{vswirl} admits a unique solution $v^\theta$ satisfying
\be \label{swirl-1}
\| \Bv^{\theta} \|_{H^2(\Omega)} \leq C \|\BF^\theta \|_{L^2 (\Omega)},
\ee
where the  constant $C$ is independent of $\BF^\theta$, $\Phi$ and $\alpha$.
\end{pro}

\begin{proof}
For each fixed $n \in \mathbb{Z}$,  the $n$-th mode $v_n^\theta$ satisfies
\be \label{swirl-2}
\left\{
\begin{aligned}
&i n \bar{U}(r) v_n^\theta - ( \mathcal{L} - n^2 )  v^\theta_n = F^\theta_n,\\
& v^\theta_n (0) = 0,\ \ \ \ \ \ \frac{d}{dr} v^\theta_n (1) = ( 1 - \alpha )v^\theta_n (1).
 \end{aligned}
 \right.
\ee

{\it Step 1. Existence.} Instead of proving  the existence of solutions to the problem \eqref{swirl-2} directly,  we first consider an associated linear problem
\be \label{swirl-4} \left\{
\ba
& i n  \bar{U}(r) V^\theta_n - \Delta_4 V^\theta_n + n^2 V^\theta_n = \widetilde{F^\theta_n}\ \ \ \ \mbox{in} \ B_1^4(0), \\
& \frac{\partial V^\theta_n}{\partial \Bn } = - \alpha V^\theta_n \ \ \ \ \ \ \mbox{on}\ \partial B_1^4(0).
\ea
\right.
\ee
Here  $B_1^4(0)$ is the unit ball centered at the origin in $\mathbb{R}^4$,
\be \nonumber
\Delta_4 = \sum_{i=1}^4 \partial_{x_i}^2 ,\ \ \ \ \ \widetilde{F^\theta_n} (x_1, x_2, x_3, x_4) = \frac{F^\theta_n(r)}{r}, \ \ \  \mbox{and}\ \ \ r= \left( \sum_{i=1}^4 x_i^2  \right)^\frac12.
\ee
The estimates for $V_n^\theta$ and the solvability of \eqref{swirl-4} are investigated by the three different cases based on whether $n$ and $\alpha$ equal $0$.

{\it Case 1. $n =0$ and  $\alpha =0$.}  The problem \eqref{swirl-4} reduces to
\be \label{swirl-4-0}
\left\{ \ba
& - \Delta_4 V^\theta_0 = \widetilde{F^\theta_0}\ \ \ \ \mbox{in} \ B_1^4(0), \\
& \frac{\partial V^\theta_0 }{\partial \Bn } = 0 \ \ \ \ \ \ \mbox{on}\ \partial B_1^4(0).
\ea \right.
\ee
Note that the condition \eqref{compatibility}
implies
\be \label{swirl-4-1}
\int_{B_1^4(0)} \widetilde{F_0^\theta}\, dx = 0.
\ee
For every $\widetilde{F_0^\theta} \in L^2(B_1^4(0))$ satisfying \eqref{swirl-4-1}, the problem \eqref{swirl-4-0} admits a unique solution $\widetilde{V_0^\theta}\in H^1(B_1^4(0))$ with
\be \nonumber
\int_{B_1^4(0)} \widetilde{V_0^\theta} \, dx = 0.
\ee
Since $\widetilde{F_0^\theta}$ is radially symmetric, the uniqueness of solutions implies that $\widetilde{V_0^\theta}$ is also radially symmetric. Let
\be \nonumber
V_0^\theta (x) = V_0^\theta(r) = \widetilde{V_0^\theta}(r) - \widetilde{V_0^\theta}(1).
\ee
It is clear that $V_0^\theta$ is also a solution to \eqref{swirl-4-0} and satisfies $V_0^\theta(1) = 0$. Moreover, it holds that
\be \label{swirl-5}
\|V_0^\theta \|_{H^2(B_1^4(0))} \leq C \|\widetilde{F_0^\theta} \|_{L^2(B_1^4(0))}.
\ee

{\it Case 2. $n =0$ and  $\alpha \neq 0$.} The problem \eqref{swirl-4} becomes
\be \label{swirl-7}
\left\{ \ba
& - \Delta_4 V_0^\theta = \widetilde{F_0^\theta}\ \ \ \ \mbox{in}\ B_1^4(0), \\
& \frac{\partial V_0^\theta}{\partial \Bn } = - \alpha V_0^\theta \ \ \ \ \ \ \ \mbox{on}\ \partial B_1^4(0).
\ea \right.
\ee
For every $\varphi$, $\phi \in H^1(B_1^4(0))$, define
\be \nonumber
\mathcal{N} ( \varphi, \,  \phi) = \int_{B_1^4(0)} \nabla_4 \varphi \cdot \nabla_4 \overline{\phi}\ dx +
\alpha \int_{\partial B_1^4(0)} \varphi \overline{\phi} \, dS .
\ee
According to Lemma \ref{sobolev}, there exists a  constant $c$ such that for every $\varphi \in H^1(B_1^4(0))$,
\be \nonumber
|\mathcal{N} (\varphi, \varphi) | \geq c \|\varphi\|_{H_1(B_1^4(0))}^2.
\ee
This means that $\mathcal{N}(\cdot, \cdot)$ is strongly coercive on $H^1(B_1^4(0))$. Hence, by Lax-Milgram theorem, for every $\widetilde{F^\theta_0} \in L^2(B_1^4(0))$, there exists a unique solution $V^\theta_0 \in H^1(B_1^4(0))$ such that
\be \nonumber
\mathcal{N} (V^\theta_0, \, \phi )= \int_{B_1^4(0)} \widetilde{F^\theta}  \overline{\phi} \, dx\ \ \ \ \text{for any } \, \phi \in H^1(B_1^4(0) ).
\ee
Thus $V^\theta_0$ is a weak solution to \eqref{swirl-7} and satisfies
\be \nonumber
\| V^\theta_0 \|_{H^1(B_1^4(0))} \leq C(\alpha) \| \widetilde{F^\theta_0}\|_{L^2(B_1^4(0))},
\ee
where the constant $C(\alpha)$ may depend on $\alpha$.
Due to the uniqueness of solutions and the fact that $\widetilde{F_0^\theta}$ is radially symmetric, the solution $V_0^\theta$ is also radially symmetric , i.e., $V_0^\theta (x) = V_0^\theta(r)$. Moreover,
\be \label{swirl-9}
0 = \int_{B_1^4(0)} \widetilde{F_0^\theta} \, dx = - \int_{\partial B_1^4(0)} \frac{\partial V_0^\theta}{\partial \Bn} \, dS =  \alpha \int_{\partial B_1^4(0)} V_0^\theta \, dS= \alpha V_0^\theta (1) |\partial B_1^4 (0)|.
\ee
This implies that $V_0^\theta (1) = 0$. Thus the problem \eqref{swirl-7} is  equivalent to
\be \nonumber
\left\{ \ba
& - \Delta_4 V_0^\theta = \widetilde{F_0^\theta}\ \ \ \ \mbox{in}\ B_1^4(0), \\
&  V_0^\theta=0 \ \ \ \ \ \ \ \mbox{on}\ \partial B_1^4(0).
\ea \right.
\ee
Applying the regularity theory for elliptic equations (\cite{ADN}), one has
\be \nonumber
\| V^\theta_0 \|_{H^2(B_1^4(0))} \leq C \|\widetilde{F_0^\theta} \|_{L^2(B_1^4(0))}.
\ee
Hence, no matter $\alpha = 0 $ or not, $V_0^\theta$ is in fact a strong solution to the problem \eqref{swirl-4}.

Let $v_0^\theta = rV_0^\theta$.  It can be verified that
\be \nonumber 
-\mathcal{L} v_0^\theta = -r \Delta_4 V_0^\theta= F_0^\theta, \ \ \ \ \ v_0^\theta(0)= 0, \ \ \ \ \ \frac{d}{dr}v_0^\theta(1) = \frac{\partial V_0^\theta}{\partial \Bn} (1) + V_0^\theta(1) = (1 - \alpha)v_0^\theta(1). 
\ee
 $v_0^\theta$ is a strong solution to the problem \eqref{swirl-2} with $n=0$.

We give some estimates (independent of $\Phi$ and $\alpha$) for $v^\theta_0$.  It has been proved that
$v_0^\theta(1) = 0$. Hence multiplying the equation in  \eqref{swirl-2} by $r \overline{v_0^\theta}$, and integrating over $[0, 1]$ yield
\be \label{swirl-17}
\int_0^1 \left| \frac{d}{dr} ( r v_0^\theta)   \right|^2 \frac1r \, dr = \Re \int_0^1 F_0^\theta \overline{v_0^\theta} r \, dr.
\ee
According to Lemma \ref{lemma1}, it holds that
\be \label{swirl-18}
\int_0^1 \left| \frac{d}{dr} ( r v_0^\theta)   \right|^2 \frac1r \, dr \leq C \int_0^1 |F_0^\theta|^2 r \, dr.
\ee
With the aid of integration by parts and the fact $v_0^\theta(0) = 0$, one has
\be \label{swirl-19}
\int_0^1 \left| \frac{d}{dr} v_0^\theta \right|^2 r \, dr+ \int_0^1 \frac{|v_0^\theta|^2}{r} \, dr
\leq \int_0^1 \left| \frac{d}{dr} ( r v_0^\theta)   \right|^2 \frac1r \, dr  \leq C \int_0^1 |F_0^\theta|^2 r \, dr.
\ee
Multiplying the equation in \eqref{swirl-2} by $r \mathcal{L} \overline{v_0^\theta}$ and integrating over $[0, 1]$ give
\be \label{swirl-20}
- \int_0^1 |\mathcal{L} v_0^\theta|^2 r\, dr = \Re \int_0^1 F_0^\theta \mathcal{L} \overline{v_0^\theta} r \, dr .
\ee
Thus it follows from  Cauchy-Schwarz inequality that one has
\be \label{swirl-21}
\int_0^1 |\mathcal{L} v_0^\theta|^2 r \, dr \leq C \int_0^1 |F_0^\theta|^2 r \, dr.
\ee

{\it Case 3. $n \neq 0$.}\ \ For every $\varphi$, $\phi \in H^1(B_1^4(0))$, define
\be \nonumber
\mathcal{N} ( \varphi, \,  \phi) = \int_{B_1^4(0)} in \bar{U}(r) \varphi \overline{\phi}   + \nabla_4 \varphi \cdot \nabla_4 \overline{\phi} + n^2 \varphi \overline{\phi} \, dx +
\alpha \int_{\partial B_1^4(0)} \varphi \overline{\phi} \, dS  .
\ee
Obviously, $\mathcal{N}(\cdot, \cdot)$ is strongly coercive on $H_1(B_1^4(0))$. Hence, by Lax-Milgram theorem, for every $\widetilde{F^\theta_n} \in L^2(B_1^4(0))$, there exists a unique solution $V^\theta_n  \in H^1(B_1^4(0))$ such that
\be \label{swirl-11}
\mathcal{N} (V^\theta_n , \, \phi )= \int_{B_1^4(0)} \widetilde{F^\theta_n}  \overline{\phi} \, dx\ \ \ \ \text{for any } \, \phi \in H^1(B_1^4(0) ).
\ee
Hence $V^\theta_n$ is a weak solution to \eqref{swirl-4}, which satisfies
\be \nonumber
\|V^\theta_n \|_{H^1(B_1^4(0))} \leq C(n, \Phi, \alpha) \|\widetilde{F_n^\theta}\|_{L^2(B_1^4(0) )}.
\ee
Applying the regularity theory for elliptic equations (\cite{ADN}) yields
\be \nonumber
\ba
\|V_n^\theta\|_{H^2(B_1^4(0))}&  \leq C \Phi |n| \|V_n^\theta \|_{L^2(B_1^4(0))} + C \|\widetilde{F_n^\theta}\|_{L^2(B_1^4(0))}
+ C(n, \alpha)  \|V_n^\theta \|_{H^1(B_1^4(0))}\\
& \leq C (n, \Phi, \alpha) \|\widetilde{F_n^\theta}\|_{L^2(B_1^4(0))}.
\ea
\ee
Hence
$V^\theta_n$ is in fact a strong solution to the problem \eqref{swirl-4}. Furthermore, due to the uniqueness of solutions and the fact that $\widetilde{F^\theta_n}$ is radially symmetric, $V^\theta_n $ is also radially symmetric, i.e., $V^\theta_n = V^\theta_n (r)$.
Let $v^\theta_n= rV^\theta_n$. Clearly, $v^\theta_n$ is a solution to the problem \eqref{swirl-2}.

{\it Step 2. Uniform estimate}.  We give some  estimates (independent of $n$, $\Phi$, and $\alpha$) for $v^\theta_n$, $n \neq 0$.
 Multiplying the equation in \eqref{swirl-2} by $r v_n^\theta$ and integrating over $[0, 1]$ yield
\be \label{swirl-22}
\int_0^1 \left| \frac{d}{dr}(r v_n^\theta)  \right|^2 \frac1r \, dr + (\alpha-2 ) |v_n^\theta(1)|^2 + n^2 \int_0^1 |v_n^\theta|^2 r \, dr
= \Re \int_0^1 F_n^\theta \overline{v_n^\theta} r \,dr
\ee
and
\be \label{swirl-23}
n \int_0^1 \bar{U}(r) |v_n^\theta|^2 r \, dr = \Im \int_0^1 F_n^\theta \overline{v_n^\theta} r \, dr .
\ee
Note that
\be \label{swirl-25} \ba
|v_n^\theta(1) | \leq \int_0^1 \left|\frac{d}{dr}(r v_n^\theta)  \right| \, dr
& \leq \left( \int_0^1 \left|\frac{d}{dr}(r v_n^\theta) \right|^2 \frac1r \, dr \right)^{\frac12} \left( \int_0^1 r \, dr \right)^{\frac12} \\
& = \frac{\sqrt{2}}{2}\left( \int_0^1 \left|\frac{d}{dr}(r v_n^\theta) \right|^2 \frac1r \, dr \right)^{\frac12}.
\ea
\ee
Taking \eqref{swirl-25} into \eqref{swirl-22} gives
\be \label{swirl-26}
 \alpha |v_n^\theta(1)|^2 + n^2 \int_0^1 |v_n^\theta|^2 r \, dr
\leq \left|  \int_0^1 F_n^\theta \overline{v_n^\theta} r \,dr \right|.
\ee
On the other hand,
\be \label{swirl-27} \ba
|v_n^\theta(1)|^2 &  \leq 2 \int_0^1 \left|\frac{d}{dr}(r v_n^\theta)  \right| |  rv_n^\theta  | \, dr
\leq \frac14 \int_0^1 \left| \frac{d}{dr}(r v_n^\theta) \right|^2 \frac1r  \, dr +4 \int_0^1 |v_n^\theta|^2 r \, dr.
\ea \ee
Substituting \eqref{swirl-26}-\eqref{swirl-27} into \eqref{swirl-22} yields
\be \label{swirl-28}
\int_0^1 \left| \frac{d}{dr}(r v_n^\theta)  \right|^2 \frac1r \, dr +  \alpha  |v_n^\theta(1)|^2 + n^2 \int_0^1 |v_n^\theta|^2 r \, dr
\leq C \left| \int_0^1 F_n^\theta \overline{v_n^\theta} r \,dr \right| .
\ee
This is equivalent to
\be \label{swirl-29}
\int_0^1 \left|  \frac{d v_n^\theta}{dr}     \right|^2 r \, dr + \int_0^1 \frac{|v_n^\theta|^2 }{r} \, dr + (1 + \alpha) |v_n^\theta(1)|^2 + n^2 \int_0^1 |v_n^\theta|^2 r \, dr
\leq C \int_0^1 |F_n^\theta| |v_n^\theta| r \, dr .
\ee

Multiplying \eqref{swirl-29} by $n^2$  and using Young's inequality give
\be \label{swirl-30}
\begin{aligned}
& \int_0^1  n^2 \left|  \frac{d v_n^\theta}{dr}     \right|^2 r  +  n^2  \frac{|v_n^\theta|^2 }{r} + n^4 |v_n^\theta|^2 r \, dr + (1 +  \alpha)n^2 |v_n^\theta(1)|^2
 \leq  C \int_0^1 |F_n^\theta|^2 r \, dr.
\end{aligned}
\ee
The estimates \eqref{swirl-19} and \eqref{swirl-29}-\eqref{swirl-30} imply that
\be \label{swirl-31}
\| \Bv^\theta \|_{H^1(\Omega)}+ \|\partial_z \Bv^\theta\|_{H^1(\Omega)}
\leq C \|\BF^\theta \|_{L^2(\Omega)}.
\ee

Multiplying the equation in \eqref{swirl-2} by $(\mathcal{L} - n^2) \overline{v_n^\theta} r $ and integrating over $[0, 1]$ give
\be \label{swirl-32}
- \int_0^1 |(\mathcal{L} - n^2) v_n^\theta  |^2 r \, dr = \Re \int_0^1 F_n^\theta (\mathcal{L } - n^2) \overline{v_n^\theta} r \, dr
- \frac{4\Phi}{\pi } \frac{\alpha }{\alpha + 4} n \Im \int_0^1 r v_n^\theta \frac{d}{dr}( r \overline{v_n^\theta}) \, dr .
\ee
Herein,
\be \label{swirl-33}
\left| \frac{4\Phi}{\pi } \frac{\alpha }{\alpha + 4} n \Im \int_0^1 r v_n^\theta \frac{d}{dr}( r \overline{v_n^\theta}) \, dr \right|
\leq C \Phi |n| \left(  \int_0^1 |v_n^\theta|^2 r \, dr \right)^{\frac12} \left( \int_0^1 \left| \frac{d}{dr}(r v_n^\theta) \right|^2 \frac1r \, dr  \right)^{\frac12} .
\ee
According to \eqref{swirl-23}, it holds that
\be \label{swirl-34}
\Phi |n| \int_0^1 (1 - r^2) |v_n^\theta|^2 r \, dr \leq C \int_0^1 |F_n^\theta| |v_n^\theta| r \, dr.
\ee
It follows from Lemmas \ref{lemma1} and \ref{weightinequality},   \eqref{swirl-28}, and \eqref{swirl-34} that
\be \label{swirl-35} \ba
\int_0^1 |v_n^\theta|^2 r \, dr & \leq C \left(  \int_0^1 (1  - r^2) |v_n^\theta|^2 r \, dr \right)^{\frac23}\left( \int_0^1 \left| \frac{d}{dr}(r v_n^\theta) \right|^2 \frac1r \, dr  \right)^{\frac13} \\
& \leq C (\Phi |n|)^{-\frac23} \int_0^1 |F_n^\theta| |v_n^\theta| r \, dr .
\ea
\ee
Therefore, one has
\be \label{swirl-36}
\int_0^1 |v_n^\theta|^2 r \, dr \leq C (\Phi |n|)^{-\frac43} \int_0^1 |F_n^\theta|^2 r \, dr.
\ee
This, together with \eqref{swirl-28}, yields
\be \label{swirl-37}
\int_0^1 \left| \frac{d}{dr}(r v_n^\theta)   \right|^2 \frac1r \, dr + n^2 \int_0^1 |v_n^\theta|^2 r \, dr
\leq C (\Phi |n|)^{-\frac23} \int_0^1 |F_n^\theta|^2 r \, dr.
\ee

Substituting \eqref{swirl-36}-\eqref{swirl-37} into \eqref{swirl-32}-\eqref{swirl-33} gives
\be \label{swirl-39}
\int_0^1 |(\mathcal{L} - n^2 ) v_n^\theta|^2 r \, dr \leq C \int_0^1 |F_n^\theta|^2 r \, dr.
\ee

{\it Step 3. Regularity of $\Bv^\theta$.} It follows from \eqref{swirl-21} and \eqref{swirl-39} that one has
\be \label{swirl-40}
\| (\mathcal{L} + \partial_z^2)v^\theta \|_{L^2(\Omega)} \leq C \|\BF^\theta\|_{L^2(\Omega)}.
\ee
Note that $\Bv^\theta$ satisfies
\be \nonumber
\Delta \Bv^\theta = ( \mathcal{L} + \partial_z^2) v^\theta \Be_\theta \ \ \ \mbox{in} \ \Omega.
\ee
It follows from the trace theorem and the axisymmetric property of $v^\theta$  that
\be \nonumber
\|\Bv^\theta\|_{H^{\frac32}(\partial \Omega)} \leq C  \|\Bv^\theta\|_{H^1(\Omega)} + C \|\partial_z \Bv^\theta \|_{H^1(\Omega)}  \leq C  \|\BF^\theta \|_{L^2(\Omega)}.
\ee
Applying the regularity theory for elliptic equations (\cite{ADN}) yields
\be \label{swirl-18}
\|\Bv^\theta\|_{H^2 (\Omega)}
\leq C \| ( \mathcal{L} + \partial_z^2) v^\theta\|_{L^2(\Omega)} + C \|\Bv^\theta\|_{H^{\frac32} (\partial \Omega)}
 \leq C \|\BF^\theta\|_{L^2(\Omega)},
\ee
where the constant $C$ is independent of $F^\theta$, $\Phi$, and $\alpha$. This finishes the proof of Proposition \ref{swirl}.
\end{proof}

Combining the regularity estimates in Propositions \ref{lemapri1}, \ref{Bpropcase1-1}, \ref{Bpropcase2-1}, \ref{medregularity1}-\ref{medregularity2}
and \ref{swirl},  one completes the proof of Theorem \ref{thm1}.


\section{Nonlinear Structural Stability}\label{secnonlinear}
In this section, the existence and local uniqueness of solutions for the nonlinear problem \eqref{SNS} is proved.  This gives the uniform nonlinear structural stability of Poiseuille flows.

 Let $\Bv=\Bu-\bBU$. Then $\Bv$ satisfies the perturbation  system
 \be  \label{perturb1}
\left\{ \ba
&\bBU \cdot \nabla \Bv + \Bv \cdot \nabla \bBU + (\Bv \cdot \nabla )\Bv - \Delta \Bv + \nabla P = \BF \ \ \ \mbox{in}\ \Omega, \\
& {\rm div}~\Bv = 0\ \ \ \mbox{in}\ \Omega, \\
\ea
\right.
\ee
supplemented with the boundary conditions and the flux constraint in \eqref{slipBC1}.

\subsection{Existence and uniqueness of solutions when $\BF$ is small}
 Set
\be \nonumber
Y = H^\frac32_{axis, \sigma} (\Omega) = \left\{ \Bv = \Bv(r, z) \in H^{\frac32}(\Omega):\ {\rm div}~\Bv =0\ \ \mbox{in}\ \Omega \  \right\},
\ee
\be \nonumber
L_{axis}^2(\Omega)= \left\{  \BF= \BF(r, z) \in L^2(\Omega):\ \int_{-\pi}^{\pi} \int_0^1 F^\theta (r, z) r^2 \, dr dz = 0              \right\}.
\ee
For any given $\BF \in L^2_{axis} (\Omega)$, as proved in Theorem \ref{thm1}, there exists  a strong solution $\Bv\in H^2(\Omega)\cap Y$ to the linear problem \eqref{2-0-1}-\eqref{slipBC1}. We denote this solution by $\mathcal{T} \BF$.
According to Theorem \ref{thm1}, one has
\be \label{perturb3}
\|\mathcal{T} \BF\|_{H^{\frac32} (\Omega) } \leq C \| \BF\|_{L^2(\Omega)}.
\ee

Now we develop an iteration scheme. Given $\BF\in L_{axis}^2(\Omega)$, denote $\Bv^0 = \mathcal{T}\BF$.
For every $j \geq 0$, let
\be \nonumber
\Bv^{j+1} = \mathcal{T} \BF - \mathcal{T} ( (\Bv^j \cdot \nabla ) \Bv^j ).
\ee
Since
\be \nonumber
\|- (\Bv^j \cdot \nabla ) \Bv^j  \|_{L^2(\Omega)} \leq C \|\Bv^j\|_{L^{12}(\Omega)} \|\nabla \Bv^j\|_{L^{\frac{12}{5}}(\Omega)}
\leq C \|\Bv^j \|_{ H^{\frac32} (\Omega) }^2
\ee
and
\be \label{verify-compatibility-1} \ba
- \int_{-\pi}^{\pi} \int_0^1 \left[ (\Bv^j \cdot \nabla ) \Bv^j   \right] \cdot \Be_\theta r^2  \, dr dz
& = - \int_{-\pi}^{\pi} \int_0^1 \left[ (v^{j, r} \partial_r + v^{j, z} \partial_z )v^{j, \theta} + \frac{v^{j, r}}{r}  v^{j, \theta}  \right] r^2 \, dr dz \\
& = \int_{-\pi}^{\pi} \int_0^1 \left[ \partial_r (r v^{j, r} ) + \partial_z (r v^{j, z} )            \right] v^{ j, \theta}  r \, dr dz = 0,
\ea \ee
$ \Bv^{j+1}$ is well-defined and $\Bv^{j+1} \in Y$.  Moreover, one has
\be \nonumber \ba
\|\Bv^{j+1} \|_{H^{\frac32}(\Omega)} & \leq C \|\BF\|_{L^2(\Omega)} + C \|(\Bv^j \cdot \nabla ) \Bv^j \|_{L^2(\Omega)}  \leq C \|\BF\|_{L^2(\Omega)} + C\|\Bv^j \|_{H^{\frac32}(\Omega)}^2 .
\ea
\ee
Hence it follows from the proof for contraction mapping theorem  that if $\BF$ satisfies \eqref{Fepsilon} with some suitably small $\varepsilon$,
 then
the problem
\be
\Bv = \mathcal{T} \BF - \mathcal{T}( \Bv \cdot \nabla )\Bv
\ee
has a unique solution $\Bv$ satisfying $\|\Bv\|_{H^{\frac32} (\Omega) } \leq C \|\BF\|_{L^2(\Omega)} \leq C \varepsilon$. In fact,  $\Bv$ is a strong solution to the problem \eqref{perturb1} and \eqref{slipBC1}. By virtue of Theorem \ref{thm1}, it holds that
\be \nonumber
\|\Bv\|_{H^2(\Omega)} \leq C (1 + \Phi^{\frac14}) \left(   \|\BF\|_{L^2(\Omega)} + \|(\Bv \cdot \nabla ) \Bv\|_{L^2(\Omega)}          \right) \leq C (1 + \Phi^{\frac14}) \|\BF\|_{L^2(\Omega)}.
\ee
Thus the proof for Part $(a)$ of  Theorem \ref{mainthm} is completed.


\subsection{Existence and uniqueness of large solutions when flux is large}
In this subsection, we prove the existence and uniqueness of strong axisymmetric solution to \eqref{SNS}-\eqref{fluxBC} when both $\BF$ and $\Phi$ are large. Before the proof, let us recall some  estimates for solutions obtained for the linearized problem.
\begin{pro}\label{lemma1-add}
Assume that $\psi_n$ is a smooth solution to the problem \eqref{2-0-8} and $v_n^\theta$ is a smooth solution to the problem \eqref{swirl-2}. It holds that
\be \label{9-1}
\int_0^1 \left| \frac{d}{dr}(r \psi_0) \right|^2 \frac1r \, dr \leq C \int_0^1 |\BF_0^*|^2 r \, dr,\ \ \ \ \ \int_0^1 |v_0^\theta|^2 r \, dr \leq C \int_0^1 |F_0^\theta|^2 r \, dr.
\ee
\be \label{9-2}
\int_0^1 \left| \frac{d}{dr}(r \psi_n) \right|^2 \frac1r \, dr + n^2 \int_0^1 |\psi_n|^2 r \, dr \leq C (\Phi |n|)^{-\frac43} \int_0^1 |\BF_n^*|^2 r \, dr ,\ \ \ \ n \neq 0.
\ee
\be \label{9-3}
\int_0^1 |v_n^\theta|^2 r \, dr \leq C (\Phi |n|)^{-\frac43} \int_0^1 |F_n^\theta|^2 r \, dr,\ \ \ \ n\neq 0.
\ee
\end{pro}

\begin{proof}
The first inequality in \eqref{9-1} is implied by \eqref{2-1-7-0} and Lemma \ref{lemma1}, while the second inequality is implied by \eqref{swirl-19}.
The results in Propositions \ref{Bpropcase2}, \ref{Bpropcase3}, \ref{case4} and \ref{case5} give the inequality \eqref{9-2}.  The inequality \eqref{9-3} is exactly \eqref{swirl-36}.

\end{proof}

Combining the results in Proposition \ref{swirl}, Proposition \ref{lemma1-add} and Theorem \ref{thm1}, one has
\begin{pro}\label{lemma2-add}
Assume that $\BF= \BF(r, z) \in L^2 (\Omega)$ satisfies \eqref{compatibility1}. There exists a constant $C_1$, which is independent of $\BF$, $\alpha$, and $\Phi$, such that
the unique axisymmetric solution $\Bv$ to the linearized problem \eqref{2-0-1}-\eqref{slipBC1} satisfies
\be \label{9-5}
v_0^r = 0, \ \ \ \|\Bv_0\|_{H^2(\Omega)} \leq C_1 \|\BF_0\|_{L^2(\Omega)}, \ \ \ \|\Bv^\theta \|_{H^2(\Omega)} \leq C_1 \|\BF^\theta\|_{L^2(\Omega) },
\ee
\be \label{9-6}
\|\mathcal{Q} \Bv\|_{L^2(\Omega)} \leq C_1 \Phi^{-\frac23} \|\mathcal{Q} \BF\|_{L^2(\Omega)},\ \ \ \
\|\mathcal{Q} \Bv^* \|_{H^{\frac32}(\Omega)}\leq C_1 \|\mathcal{Q} \BF^*\|_{L^2(\Omega)},
\ee
where $\mathcal{Q}$ is defined by \eqref{projection}.
\end{pro}

For every $n \in \mathbb{Z}$, let $\psi_n$ denote the stream function associated with the $n$-th mode of velocity $\Bv_n$, i.e.,
\be \nonumber
v_n^r = in \psi_n\ \ \ \  \mbox{and}\ \ \ \ v_n^z = - \frac{1}{r} \frac{d}{dr}(r \psi_n).
\ee
Then $\psi_n$ and $v_n^\theta$ satisfy the nonlinear system \eqref{stream-mode} and \eqref{swirl-mode}, respectively.
\begin{proof}[\bf{Proof for Part $(b)$ of Theorem \ref{mainthm}}]
We also use the iteration method to prove the existence. The proof is divided into 3 steps.

{\it Step 1. Iteration scheme.} Given $\BF \in L_{axis}^2(\Omega)$, for every $j \geq 0$,
the linearized problems \eqref{2-0-8} and \eqref{swirl-system} admit a unique solution $\psi_n^0$ and $v_n^{0, \theta} $ for each $n$. The corresponding velocity field is
\be \nonumber \ba
\Bv^0 & = \sum_{n \in \mathbb{Z}} v^{0, r}_n e^{inz} \Be_r + v^{0, z}_n e^{inz} \Be_z + v^{0, \theta}_n e^{inz}\Be_\theta \\
\\ & = \sum_{n \in \mathbb{Z}} in \psi_n^0 e^{in z} \Be_r - \frac1r \frac{d}{dr}(r \psi_n^0) e^{inz} \Be_z + v_n^{0, \theta} e^{inz} \Be_\theta.
\ea \ee
For every $j \geq 0$, let
$\psi_{n}^{j+1}$ be the solution to the iteration problem
\be \label{7-3}
\left\{
 \ba
& in \bar{U}(r) (\mathcal{L} - n^2) \psi_n^{j+1} - (\mathcal{L} - n^2)^2 \psi_n^{j+1}  \\
&=  in F_n^r - \frac{d}{dr} F_n^z - \frac{d}{dr} \sum_{m \in \mathbb{Z}} v_{m}^{j, r} \omega_{n-m}^{j, \theta} - in \sum_{m\in \mathbb{Z}} v_{n-m}^{j, z} \omega_m^{j, \theta} + in \sum_{m\in \mathbb{Z}} \frac{ v_{n-m}^{j, \theta} v_m^{j, \theta}}{r} ,\\
&\psi_n^{j+1} (0) = \psi_n^{j+1} (1) = \mathcal{L} \psi_n^{j+1} (0) = 0,\ \ \ \ \mathcal{L} \psi_n^{j+1} (1) = - \alpha \frac{d}{dr}\psi_n^{j+1} (1) .
\ea
\right.
\ee
Similarly, for $j\geq 0$, let $v_n^{j+1, \theta}$ be the solution to the iteration problem
\be \label{7-4}
\left\{
\begin{aligned}
&in \bar{U} v_{n}^{j+1, \theta}- (\mathcal{L} - n^2)v_{n}^{j+1, \theta}\\
&\quad = F_n^\theta - \sum_{m \in \mathbb{Z}} v_{m}^{j, r} \frac{d}{dr} v_{n-m}^{j, \theta} - \sum_{m \in \mathbb{Z}} im v_{n-m}^{j, z} v_m^{j, \theta} -
\sum_{m \in \mathbb{Z}} \frac{v_{n-m}^{j, \theta} }{r } v_m^{j, r},\\
&v_n^{j+1, \theta} (0) = 0, \ \ \ \ \ \frac{d}{dr} v_n^{j+1, \theta} (1) = (1- \alpha ) v_n^{j+1, \theta}(1).
\end{aligned}
\right.
\ee
Here $v_n^{j, r} = in \psi_n^j $, $v_n^{j, z}= -\frac1r \frac{d}{dr} (r \psi_n^j)$, and $\omega_n^{j, \theta} = (\mathcal{L} -n^2) \psi_n^j$.

{\it Step 2. Mathematical induction and the existence of solutions}. Let
\be \nonumber
\Bv^j = \sum_{n \in \mathbb{Z}} v_{n}^{j, r} e^{inz} \Be_r + v_n^{j,z}e^{inz} \Be_z + v_n^{j, \theta} \Be_\theta\ \  \ \text{and} \ \ \ \omega^{j, \theta} =
\sum_{n \in \mathbb{Z}} \omega_n^{j, \theta} e^{inz} .
\ee

Define
\be \nonumber
\mathcal{J} = \left\{ \Bv= \sum_{n \in \mathbb{Z}} \Bv_n \left|  \ba
& v_0^r = 0,\ \ \ \ \ \| \mathcal{Q}\Bv \|_{L^2(\Omega)} \leq \Phi^{-\frac{7}{12}}, \\
& \left( \|\Bv_0\|_{H^2(\Omega)}^2 + \|\mathcal{Q} \Bv^* \|_{H^{\frac32}(\Omega)}^2 + \|\mathcal{Q}\Bv^\theta \|_{H^2(\Omega)}^2       \right)^{\frac12} \leq 6 C_1 \|\BF\|_{L^2(\Omega)}      \ea         \right.   \right\},
\ee
where $C_1$ is the constant indicated in Proposition \ref{lemma2-add}.

Assume that $\|\BF\|_{L^2(\Omega)} \leq \Phi^{\frac{1}{32}}$. According to Proposition \ref{lemma2-add}, one has $\Bv^0 \in \mathcal{J}$. Assume that $\Bv^j \in \mathcal{J}$, we claim that $\Bv^{j+1} \in \mathcal{J}$.
Denote
\be \nonumber
F_n^{j, z} = \sum_{m \in \mathbb{Z}} v_{m}^{j, r} \omega_{n-m}^{j, \theta}, \ \ \ \ \ \ \ F_n^{j, r} = - \sum_{m \in \mathbb{Z}} v_{n-m}^{j, z} \omega_m^{j, \theta}
+   \sum_{m \in \mathbb{Z}} \frac{v_{n-m}^{j, \theta} v_m^{j, \theta}}{r},
\ee
and
\be \nonumber
F_n^{j, \theta }= - \sum_{m \in\mathbb{Z}} v_{m}^{j, r} \frac{d}{dr} v_{n-m}^{j, \theta} - \sum_{m \in \mathbb{Z}} im v_{n-m}^{j, z} v_m^{j, \theta} -
\sum_{m \in \mathbb{Z}} \frac{v_{n-m}^{j, \theta} }{r } v_m^{j, r}.
\ee

By Hausdorff-Young inequality and Sobolev embedding inequality, one has
\be \label{7-5}
\ba
 \left( \sum_{n \in \mathbb{Z}} \|F_{n}^{j, z} \|_{L^2(\Omega)}^2  \right)^{\frac12}
= & \left(  \sum_{n \in \mathbb{Z}} \left\|  \sum_{m \in \mathbb{Z}} v_{m}^{j, r} \omega_{n-m}^{j, \theta}    \right\|_{L^2(\Omega)}^2   \right)^{\frac12}
=  \left(     \sum_{n \in \mathbb{Z}} \left\|  \sum_{m \neq 0 } v_{m}^{j, r} \omega_{n-m}^{j, \theta}    \right\|_{L^2(\Omega)}^2        \right)^{\frac12} \\
\leq & C \|\mathcal{Q} \Bv^j \|_{L^6(\Omega)} \|\omega^{j, \theta} \|_{L^3(\Omega)}
\leq  C \|\mathcal{Q} \Bv^j \|_{H^1(\Omega)} \|\Bv^j \|_{H^{\frac32}(\Omega)} \\
\leq & C \|\mathcal{Q} \Bv^j \|_{L^2(\Omega)}^{\frac13} \|\mathcal{Q} \Bv^j \|_{H^{\frac32}(\Omega)}^{\frac23} \|\Bv^j \|_{H^{\frac32}(\Omega)} \\
\leq &  C  \Phi^{-\frac{25}{144}}\|\BF\|_{L^2(\Omega)}
\ea
\ee
and
\be \label{7-6}
\ba
  \left( \sum_{n \neq 0 } \|F_{n}^{j, r} \|_{L^2(\Omega)}^2  \right)^{\frac12}
\leq & \left( \sum_{n \neq 0 } \left\|\sum_{m \in \mathbb{Z}} \frac{v_{n-m}^{j, \theta} v_m^{j, \theta} }{r}      \right\|^2_{L^2(\Omega)}  \right)^{\frac12} + \left(    \sum_{n\neq 0} \left\|  \sum_{m \in \mathbb{Z}} v_{n-m}^{j, z} \omega_m^{j, \theta}    \right\|_{L^2(\Omega)}^2   \right)^{\frac12} \\
\leq & \left( \sum_{n \neq 0 } \left\|\sum_{m \neq n } \frac{v_{n-m}^{j, \theta} v_m^{j, \theta} }{r}      \right\|^2_{L^2(\Omega)}  \right)^{\frac12} + \left( \sum_{n \neq 0} \left\| \frac{v_0^{j, \theta} v_n^{j, \theta} }{r}  \right\|_{L^2(\Omega)}^2        \right)^{\frac12} \\
&\ \ \ \  + \left(    \sum_{n\neq 0} \left\|  \sum_{m \neq n } v_{n-m}^{j, z} \omega_m^{j, \theta}    \right\|_{L^2(\Omega)}^2   \right)^{\frac12} + \left(          \sum_{n \neq 0 } \| v_0^{j, z} \omega_{n}^{j, \theta} \|_{L^2(\Omega)}^2 \right)^{\frac12} \\
\ea
\ee
Thus one has
\be
\ba
&\left( \sum_{n \neq 0 } \|F_{n}^{j, r} \|_{L^2(\Omega)}^2  \right)^{\frac12}  \\
\leq & C \|\mathcal{Q} \Bv^{j, \theta} \|_{L^6(\Omega)} \left\| \frac{v^{j, \theta}}{r} \right\|_{L^3(\Omega)} +
C \left\| \frac{\Bv_0^{j, \theta} }{r} \right\|_{L^3(\Omega)} \|\mathcal{Q} \Bv^{j, \theta} \|_{L^6(\Omega) }  \\
&\ \ \ \ + C \| \mathcal{Q} \Bv^{j, *}\|_{L^6(\Omega)} \|\omega^{j, \theta} \|_{L^3(\Omega)} + C \|\Bv_0^{j} \|_{L^\infty(\Omega)} \|\mathcal{Q} \Bv^{j, *}\|_{H^1(\Omega)}\\
\leq & C \|\mathcal{Q} \Bv^{j, \theta}\|_{L^2(\Omega)}^{\frac12} \|\mathcal{Q}\Bv^{j, \theta}\|_{H^2(\Omega)}^{\frac12} \|\nabla \Bv^{j, \theta}\|_{L^3(\Omega)}  \\
& + C \|\mathcal{Q} \Bv^{j, *}\|_{H^{\frac32}(\Omega)}^{\frac23} (\|\mathcal{Q} \Bv^{j, *}\|_{L^2(\Omega)}^{\frac13}\| \Bv^{j, *}\|_{H^{\frac32}(\Omega)}+ \|\Bv_0^j \|_{H^2(\Omega)} \|\mathcal{Q} \Bv^{j, *} \|_{L^2(\Omega)}^{\frac13} )\\
\leq & C \Phi^{- \frac{25}{144}} \|\BF\|_{L^2(\Omega)}.
\ea
\ee
Furthermore, it holds that
\be \label{7-7} \ba
&  \left( \sum_{n \in \mathbb{Z}} \|F_{n}^{j, \theta} \|_{L^2(\Omega)}^2  \right)^{\frac12}  \\
\leq & \left( \sum_{n \in \mathbb{Z}} \left\|  \sum_{m \in \mathbb{Z}} v_m^{j, r} \frac{d}{dr} v_{n-m}^{j, \theta} \right\|_{L^2(\Omega)}^2 \right)^{\frac12} + \left(  \sum_{n \in \mathbb{Z}} \left\|  \sum_{m \in \mathbb{Z}} v_{n-m}^{j, z} i m v_m^{j, \theta}     \right\|_{L^2(\Omega)}^2 \right)^{\frac12} \\
&\ \ \ \ \ + \left( \sum_{n \in \mathbb{Z}} \left\|  \sum_{m \in \mathbb{Z}} v_m^{j, r} \frac{v_{n-m}^{j, \theta}}{r}      \right\|_{L^2(\Omega)}^2  \right)^{\frac12}\\
 = & \left( \sum_{n \in \mathbb{Z}} \left\|  \sum_{m \neq 0 } v_m^{j, r} \frac{d}{dr} v_{n-m}^{j, \theta} \right\|_{L^2(\Omega)}^2 \right)^{\frac12} + \left(  \sum_{n \in \mathbb{Z}} \left\|  \sum_{m \neq 0 } v_{n-m}^{j, z} i m v_m^{j, \theta}     \right\|_{L^2(\Omega)}^2 \right)^{\frac12} \\
&\ \ \ \ \ + \left( \sum_{n \in \mathbb{Z}} \left\|  \sum_{m \neq 0 } v_m^{j, r} \frac{v_{n-m}^{j, \theta}}{r}      \right\|_{L^2(\Omega)}^2  \right)^{\frac12}\\
\leq & C \|\mathcal{Q} \Bv^{j, *} \|_{L^6(\Omega)} \left(  \left\| \frac{d v^{j, \theta} }{dr} \right\|_{L^3(\Omega)}+ \left\|  \frac{v^{j, \theta}}{r}\right\|_{L^3(\Omega)} \right)  +
C \| \Bv^{j, * } \|_{L^3(\Omega)} \|\partial_z  \mathcal{Q} \Bv^{j, \theta} \|_{L^6(\Omega)} \\
\leq & C \|\mathcal{Q} \Bv^{j, *} \|_{L^6(\Omega)} \| \nabla \Bv^{j, \theta} \|_{L^3(\Omega)} +  C \| \Bv^{j, * } \|_{L^3(\Omega)} \|\partial_z  \mathcal{Q} \Bv^{j, \theta} \|_{L^6(\Omega)} \\
\leq & C \Phi^{- \frac{25}{144}} \|\BF\|_{L^2(\Omega)}.
\ea
\ee

By virtue of Proposition \ref{lemma2-add},
\be \nonumber
\|\Bv_0^{j+1}\|_{H^2(\Omega)} \leq C_1 \|\BF\|_{L^2(\Omega)} + C_1 C  \Phi^{-\frac{25}{144}}\|\BF\|_{L^2(\Omega)},
\ee
\be \nonumber
\| \mathcal{Q} \Bv^{j+1, * } \|_{H^{\frac32}(\Omega)}+ \|\mathcal{Q} \Bv^{j+1, \theta} \|_{H^2(\Omega)}  \leq C_1 \|\BF\|_{L^2(\Omega)} + C_1 C \Phi^{-\frac{25}{144}} \|\BF\|_{L^2(\Omega)},
\ee
and
\be \nonumber
\|\mathcal{Q} \Bv^{j+1} \|_{L^2(\Omega)} \leq C_1 \Phi^{-\frac23} \|\BF\|_{L^2(\Omega)} + C_1 C \Phi^{-\frac{121}{144}} \|\BF\|_{L^2(\Omega)}.
\ee
Hence there exists a constant $\Phi_0$ such that for every $\Phi \geq \Phi_0$, one has $\Bv^{j+1} \in \mathcal{J}$.

By mathematical induction, $\Bv^j \in \mathcal{J}$ for every $j \in \mathbb{N}$. Due to the uniform bound for $\Bv^j$, there exists a function $\Bv \in \mathcal{J}$, such that $\Bv^j \rightharpoonup \Bv$ in $H^{\frac32}(\Omega)$. And it holds that
 \be \nonumber
 \|\mathcal{Q} \Bv\|_{L^2(\Omega)} \leq \Phi^{-\frac{7}{12}} \ \ \ \ \text{and}\ \ \ \ \left(\|\Bv_0\|_{H^2(\Omega)}^2 + \|\mathcal{Q} \Bv^*\|_{H^{\frac32}(\Omega)}^2 + \|\mathcal{Q} \Bv^{\theta} \|_{H^2(\Omega)} \right)^\frac12 \leq 6 C_1 \|\BF\|_{L^2(\Omega)}.
 \ee
 By the way, for every $j \in \mathbb{N}$, it holds that
\be \label{7-11}
\|\Bv^{j+1} \|_{H^2(\Omega)}
\leq C (1 + \Phi^{\frac14} ) \|\BF\|_{L^2(\Omega)} + C (1 + \Phi^{\frac14} ) \Phi^{-\frac{25}{144}} \|\BF\|_{L^2(\Omega)}
\leq C (1 + \Phi^{\frac14} ) \|\BF\|_{L^2(\Omega)}.
\ee
Taking the limit for $j\to \infty$ yields
\be \label{7-12}
\|\Bv\|_{H^2(\Omega)} \leq C (1 + \Phi^{\frac14} ) \|\BF\|_{L^2(\Omega)}.
\ee

 On the other hand, since $(\psi_n^{j+1}, v_n^{j+1, \theta})$ is the solution to the problem \eqref{7-3}-\eqref{7-4}, one can prove that
\be \label{7-9}
{\rm curl}~\left(  (\bBU \cdot \nabla ) \Bv^{j+1} + (\Bv^{j+1} \cdot \nabla ) \bBU - \Delta \Bv^{j+1} + (\Bv^j \cdot \nabla )\Bv^j - \BF                        \right)=0\ \ \ \mbox{in}\ \Omega.
\ee
Taking the limit of the equation \eqref{7-9} yields
\be \nonumber
{\rm curl}~\left(  (\bBU \cdot \nabla ) \Bv + (\Bv \cdot \nabla ) \bBU - \Delta \Bv + (\Bv \cdot \nabla )\Bv - \BF                        \right)=0\ \ \ \mbox{in}\ \Omega.
\ee
Therefore, there exists a function $P$ with $\nabla P \in L^2(\Omega)$, such that
\be \label{7-10}
(\bBU \cdot \nabla ) \Bv + (\Bv \cdot \nabla ) \bBU - \Delta \Bv + (\Bv \cdot \nabla )\Bv  + \nabla P = \BF \ \ \ \mbox{in}\ \Omega.
\ee

{\it Step 3. Uniqueness.}
Suppose that $\Bv, \tilde{\Bv} \in \mathcal{J}$ are two solutions of the nonlinear problem.
Let $\psi$ and $\tilde{\psi}$ be the stream functions associated with $\Bv$ and $\tilde{\Bv}$. Then $\psi-\tilde{\psi}$ and $v^\theta-\tilde{v}^\theta$ satisfy
\be \label{7-21} \ba
 & \bar{U}(r) \partial_z (\mathcal{L} + \partial_z^2 ) (\psi - \tilde{\psi} ) - (\mathcal{L} + \partial_z^2)^2 (\psi- \tilde{\psi} ) \\
 = & - \partial_r (v^r \omega^\theta - \tilde{v}^r \tilde{\omega}^\theta) - \partial_z (v^z \omega^\theta - \tilde{v}^z \tilde{\omega}^\theta)
 +  \partial_z \left(  \frac{v^\theta}{r} v^\theta - \frac{\tilde{v}^\theta}{r} \tilde{v}^\theta      \right)
 \ea
\ee
and
\be \label{7-22}
\ba
&\bar{U}(r) \partial_z (v^\theta - \tilde{v}^\theta) - (\mathcal{L} +  \partial_z^2) (v^\theta - \tilde{v}^\theta) \\
= & - (v^r \partial_r v^\theta - \tilde{v}^r \partial_r \tilde{v}^\theta) - (v^z \partial_z v^\theta - \tilde{v}^z \partial_z v^\theta)
- \left( \frac{v^r}{r} v^\theta - \frac{\tilde{v}^r }{r } \tilde{v}^\theta        \right),
\ea
\ee
respectively.
By Proposition \ref{lemma2-add}, one has
\be \label{7-23}
\ba
& \|\Bv_0 - \tilde{\Bv}_0\|_{L^2(\Omega) } \\ \leq &
C_1 \left\|\sum_{m \neq 0} v_m^r \omega_{-m}^\theta - \tilde{v}_m^r \tilde{\omega}_{-m}^\theta \right\|_{L^2(\Omega)} + C_1 \left\|\sum_{m \neq 0}  v_m^r \frac{d}{dr}v_{-m}^\theta - \tilde{v}^r  \frac{d}{dr}\tilde{v}_{-m}^\theta  \right\|_{L^2(\Omega)}\\
&\ \ + C_1 \left\| \sum_{m \neq 0 } im v_m^z v_{-m }^\theta - im \tilde{v}_m^z \tilde{v}_{-m}^\theta \right\|_{L^2(\Omega)}
+ C_1 \left\| \sum_{m \neq 0} \frac{v_m^r v_{-m}^\theta }{r}- \frac{\tilde{v}_m^r \tilde{v}_{-m}^\theta}{r}  \right\|_{L^2(\Omega)}.
\ea
\ee
Herein,
\be \label{7-24}
\ba & \left\|\sum_{m \neq 0} v_m^r \omega_{-m}^\theta - \tilde{v}_m^r \tilde{\omega}_{-m}^\theta \right\|_{L^2(\Omega)} \\
\leq & \left\| \sum_{m \neq 0 } (v_m^r - \tilde{v}_m^r ) \omega_{-m}^\theta  \right\|_{L^2(\Omega)} + \left\| \sum_{m \neq 0} \tilde{v}^r_m
(\omega_{-m}^\theta - \tilde{\omega}_{-m}^\theta )        \right\|_{L^2(\Omega)} \\
\leq & C \|\mathcal{Q}\Bv - \mathcal{Q} \tilde{\Bv} \|_{L^9(\Omega)}  \|\mathcal{Q} \omega^\theta \|_{L^{\frac{18}{7}}(\Omega)}
+ C \|\mathcal{Q} \tilde{\Bv}^* \|_{L^{18}(\Omega)} \|\mathcal{Q}\omega^\theta  - \mathcal{Q} \tilde{\omega}^\theta  \|_{ L^{\frac94 }(\Omega)} \\
\leq & C  \|\mathcal{Q} \Bv - \mathcal{Q} \tilde{\Bv}  \|_{H^{\frac76}(\Omega)} \|\mathcal{Q} \Bv^* \|_{H^{\frac43}(\Omega)}  \\
\leq & C \Phi^{-\frac{7}{12} \cdot \frac19 + \frac{1}{32} \cdot \frac89}\|\mathcal{Q} \Bv - \mathcal{Q} \tilde{\Bv}  \|_{H^{\frac76}(\Omega)}
= C \Phi^{-\frac{1}{27}} \|\mathcal{Q} \Bv - \mathcal{Q} \tilde{\Bv}  \|_{H^{\frac76}(\Omega)}.
\ea
\ee
Similarly, one has
\be \label{7-25}
\ba
& \left\|\sum_{m \neq 0}  v_m^r \frac{d}{dr}v_{-m}^\theta - \tilde{v}^r_m \frac{d}{dr}\tilde{v}_{-m}^\theta  \right\|_{L^2(\Omega)}
+ \left\| \sum_{m \neq 0 } im v_m^z v_{-m }^\theta - im \tilde{v}_m^z \tilde{v}_{-m}^\theta \right\|_{L^2(\Omega)} \\
&\ \ \ \ \ \ \ + \left\| \sum_{m \neq 0} \frac{v_m^r v_{-m}^\theta }{r}- \frac{\tilde{v}_m^r \tilde{v}_{-m}^\theta}{r}  \right\|_{L^2(\Omega)} \\
\leq & C \Phi^{-\frac{1}{27}}  \|\mathcal{Q}\Bv - \mathcal{Q} \tilde{\Bv} \|_{H^{\frac76}(\Omega)}.
\ea
\ee
Meanwhile,
\be \label{7-26}
\ba & \left(   \sum_{n \neq 0} \left\|  \sum_{m\neq 0} v_m^r \omega_{n -m}^\theta - \tilde{v}^r_m  \tilde{\omega}_{n-m}^\theta         \right\|_{L^2(\Omega)}^2       \right)^{\frac12} \\
\leq & \left(   \sum_{n \neq 0} \left\|  \sum_{m\neq 0} (v_m^r -\tilde{v}_m^r)  \omega_{n -m}^\theta      \right\|_{L^2(\Omega)}^2       \right)^{\frac12} + \left(   \sum_{n \neq 0} \left\|  \sum_{m\neq 0} \tilde{v}_m^r ( \omega_{n -m}^\theta -  \tilde{\omega}_{n-m}^\theta  )       \right\|_{L^2(\Omega)}^2       \right)^{\frac12}\\
\leq & C \|\Bv- \tilde{\Bv} \|_{L^9(\Omega)}  \|\omega^\theta \|_{ L^{\frac{18}{7}}(\Omega)}
+ C \|\omega^\theta - \tilde{\omega}^\theta \|_{L^{\frac94}(\Omega)} \| \tilde{ \Bv}\|_{ L^{18} (\Omega)}\\
\leq & C \|\Bv - \tilde{\Bv}\|_{H^{\frac76}(\Omega)} \|\tilde{\Bv} \|_{H^{\frac43}(\Omega)} \\
\leq & C \Phi^{\frac{1}{32} } \|\Bv - \tilde{\Bv} \|_{H^{\frac76}(\Omega)}.
\ea
\ee
Similarly, one has
\be \label{7-27}
\ba &
\left(   \sum_{n \neq 0} \left\|  \sum_{m\in \mathbb{Z} } v_m^z \omega_{n -m}^\theta - \tilde{v}^z_m  \tilde{\omega}_{n-m}^\theta         \right\|_{L^2(\Omega)}^2       \right)^{\frac12}  +  \left(   \sum_{n \neq 0} \left\|  \sum_{m\in \mathbb{Z} }  \frac{v^\theta_m v^\theta_{n-m}  }{r} - \frac{\tilde{v}^\theta_m  \tilde{v}^\theta_{n-m}  }{r }        \right\|_{L^2(\Omega)}^2       \right)^{\frac12} \\
& + \left(   \sum_{n \neq 0} \left\|  \sum_{m \in \mathbb{Z} } v_m^r \frac{d}{dr}v^\theta_{n-m}  - \tilde{v}^r_m  \frac{d}{dr}\tilde{v}^\theta_{n-m}         \right\|_{L^2(\Omega)}^2       \right)^{\frac12}
+ \left(   \sum_{n \neq 0} \left\|  \sum_{m\in \mathbb{Z} } \frac{v_m^r v_{n-m}^\theta}{r} - \frac{\tilde{v}_m^r \tilde{v}_{n-m}^\theta }{r}       \right\|_{L^2(\Omega)}^2       \right)^{\frac12} \\
& + \left(   \sum_{n \neq 0} \left\|  \sum_{m\in \mathbb{Z} } i(n-m) v_m^z  v^\theta_{n-m}  - i(n-m) \tilde{v}^z_m  \tilde{v}_{n-m}^\theta         \right \|_{L^2(\Omega)}^2       \right)^{\frac12} \\
\leq & C \Phi^{\frac{1}{32}} \|\Bv - \tilde{\Bv} \|_{H^{\frac76}(\Omega)}.
\ea
\ee

Applying Proposition \ref{lemma2-add} again gives
\be \label{7-28} \ba
& \|\Bv_0 - \tilde{\Bv}_0 \|_{H^2(\Omega)} + \|\mathcal{Q} \Bv - \mathcal{Q} \tilde{\Bv} \|_{H^{\frac76}(\Omega)}  \\
\leq &  \|\Bv_0 - \tilde{\Bv}_0 \|_{H^2(\Omega)} + C \|\mathcal{Q} \Bv - \mathcal{Q} \tilde{\Bv} \|_{L^2 (\Omega)}^{\frac29}
 \|\mathcal{Q} \Bv - \mathcal{Q} \tilde{\Bv} \|_{H^{\frac32} (\Omega)}^{\frac79} \\
\leq & C \Phi^{-\frac{1}{27} }  \|\mathcal{Q}\Bv - \mathcal{Q} \tilde{\Bv} \|_{H^{\frac76}(\Omega)} + C \Phi^{-\frac{101}{864} } \|\Bv - \tilde{\Bv} \|_{H^{\frac76}(\Omega)} .
\ea \ee
For sufficiently large $\Phi \geq \Phi_0$, it holds that
\be \label{7-29}
\|\Bv_0 - \tilde{\Bv}_0 \|_{H^2(\Omega)} + \|\mathcal{Q} \Bv - \mathcal{Q} \tilde{\Bv} \|_{H^{\frac76}(\Omega)}  = 0.
\ee
This implies that $\Bv = \tilde{\Bv}$. Hence the uniqueness is proved.
This finishes the proof for Part $(b)$ of Theorem \ref{mainthm}.
\end{proof}

\begin{remark}
Assume that $\BF=0$ and $\Bu$ is a solution to the steady Navier-Stokes equations \eqref{SNS} supplemented with the Navier boundary condition \eqref{slipBC} and flux constraint \eqref{fluxBC}. If $\Phi$ is sufficently large and  $\Bv= \Bu - \bBU$ satisfies $\|\Bv \|_{H^{\frac23} (\Omega)} \leq \Phi^{\frac{1}{64}}$, then by Proposition \ref{lemma2-add} ,  it holds that
\be \nonumber
\|\mathcal{Q} \Bv\|_{L^2(\Omega)} \leq \Phi^{-\frac{7}{12}} , \ \ \ \ \ \|\Bv_0\|_{H^2(\Omega)} + \|\mathcal{Q}\Bv\|_{H^{\frac32} (\Omega)} +
\|\mathcal{Q} \Bv^\theta\|_{H^2(\Omega)} \leq C \Phi^{\frac{1}{32}}.
\ee
According to the above uniqueness proof, one can prove that $\Bv = 0$, which gives the uniqueness of Poiseuille flow.
\end{remark}



\appendix
\section{Some elementary lemmas}
In this appendix, we collect some basic lemmas which play important roles in the paper and might be useful elsewhere. The proof of most of lemmas can be found in \cite{WX1}.
We first give some Poincar\'e type inequalities.
\begin{lemma}\label{lemma1}
For a  function $g\in C^2([0,1])$,
it holds that
\be \label{2-1-11}
\int_0^1 |g |^2  r \, dr \leq  \int_0^1 \left|   \frac{d}{dr} (r g )     \right|^2 \frac1r \, dr.
\ee
If, in addition, $g(0)=g(1)=0$, then one has
\be
  \int_0^1 \left|   \frac{d}{dr} (r g )     \right|^2 \frac1r \, dr
  \leq \left( \int_0^1 | \mathcal{L}g|^2  r \, dr  \right)^{\frac12} \left( \int_0^1 |g|^2 r\, dr  \right)^{\frac12}
\leq  \int_0^1 | \mathcal{L}g|^2  r \, dr.
\ee
\end{lemma}


In the following lemma,  the pointwise estimates for the functions on the boundary are provided.
\begin{lemma}\label{lemmaA2}
For a function $g\in C^3([0,1])$, one has
\begin{equation}\label{estLinfty}
\begin{aligned}
\left|\frac{d}{dr}(rg)(1)\right|\leq  & 2  \left( \int_0^{1}\left|\frac{d}{dr}(rg)\right|^2 \frac{1}{r}\, dr \right)^{\frac14} \left( \int_0^{1}\left|\mathcal{L}g\right|^2r \, dr \right)^{\frac14} + 4 \left( \int_0^1 \left| \frac{d}{dr}(r g) \right|^2 \frac1r \, dr \right)^{\frac12}
\end{aligned}
\end{equation}
and
\begin{equation}\label{3-3-1-16}
| \mathcal{L} g (1) |  \leq 2 \left( \int_{\frac12}^1 |\mathcal{L } g |^2  r \, dr \right)^{\frac12}
+ 2 \left(   \int_{\frac12}^1 \left| \frac{d}{dr} ( r \mathcal{L} g)   \right|^2  \frac{1}{r} \, dr \right)^{\frac14}  \left( \int_{\frac12}^1 \left| \mathcal{L} g \right|^2 r \, dr       \right)^{\frac14}.
\end{equation}
Furthermore, if $g\in C^4([0,1])$, with $\mathcal{L} g (0) = 0$ and $\frac{d}{dr}( r \mathcal{L}g) \in C[0, 1]$, then one has
\begin{equation}\label{3-3-1-20}
\int_0^1 \left| \frac{d}{dr} ( r \mathcal{L} g )   \right|^2 \frac{1}{r} \, dr  \leq C \left(\int_0^1 |\mathcal{L} g|^2 r \, dr \right)^{\frac12} \left( \int_0^1 |\mathcal{L}^2 g |^2 r \, dr   \right)^{\frac12} + C \int_0^1 |\mathcal{L}g|^2 r\, dr ,
\end{equation}
\begin{equation}\label{3-3-1-20-1}
\ba
 \left| \frac{d}{dr} (r \mathcal{L}g)(1)  \right|^2
 \leq & 4 \int_{\frac12}^1 \left|  \frac{d}{dr} ( r \mathcal{L} g)   \right|^2  \frac{1}{r}\, dr\\
  &\quad    + 8 \left(  \int_{\frac12}^1 | \mathcal{L}^2 g |^2  r\, dr   \right)^{\frac12}  \left(
\int_{\frac12}^1 \left| \frac{d}{dr} ( r \mathcal{L} g) \right|^2 \frac{1}{r} \, dr  \right)^{\frac12},
\ea
\end{equation}
and
\be \label{estlinfty}
|\mathcal{L} g (1)  |  + \left| \frac{d}{dr} ( r \mathcal{L} g) (1)  \right| \leq C \left( \int_0^1| \mathcal{L} g|^2 r \, dr +  \int_0^1 | \mathcal{L}^2 g |^2 r \, dr   \right)^{\frac12}.
\ee

If, in addition, $g(0)=g(1)=0$, then one has
\begin{equation}\label{estLinfty1}
\left|\frac{d}{dr}(rg)(1)\right|\leq  2\sqrt{3} \left( \int_0^{1}\left|g\right|^2 r \, dr \right)^{\frac18} \left( \int_0^{1}\left|\mathcal{L}g\right|^2 r \, dr \right)^{\frac38}.
\end{equation}
\end{lemma}


The following lemma is a variant of Hardy-Littlewood-P\'olya type inequality.
\begin{lemma}\label{lemmaHLP}
Suppose that $g\in C^1([0,1])$ satisfies $g(0)=0$. It holds that
\be \label{HLP-2}
\int_0^1 |g|^2 r \, dr \leq C \int_0^1 \left| \frac{d(r g) }{dr}   \right|^2 \frac{1-r^2}{r} \, dr .
\ee
\end{lemma}

The following lemma is about two weighted interpolation inequalities, which are quite similar to \cite[(3.28)]{M}.
\begin{lemma}\label{weightinequality} Let $g \in C^2[0, 1]$, then one has
\begin{equation} \label{weight1} \ba
\int_0^1 |g|^2r \, dr  \leq & C \left(\int_0^1 (1-r^2)|g|^2 r \, dr\right)^{\frac23} \left(\int_0^1  \left|\frac{d}{dr}(rg)\right|^2 \frac{1}{r} \, dr\right)^{\frac13} \\
&\ \ \ \  + C \int_0^1 (1-r^2)|g|^2 r\, dr
\ea
\end{equation}
and
\be \label{weight2} \ba
\int_0^1 \left| \frac{d}{dr}(rg) \right|^2 \frac1r \, dr &  \leq C \left( \int_0^1 \frac{1-r^2}{r} \left| \frac{d}{dr} ( rg)  \right|^2 \, dr \right)^{\frac23} \left( \int_0^1 |\mathcal{L} g|^2 r \, dr  \right)^{\frac13} \\
&\ \ \ \ \ \ \ \ \ \ \ + C \int_0^1 \frac{1-r^2}{r} \left| \frac{d}{dr} ( rg)  \right|^2 \, dr .
\ea
\ee
\end{lemma}

\begin{lemma}\label{sobolev}
Assume that $\alpha > 0$. There exists a constant $C$, such that for every $g\in H^1(B_1^4(0))$, one has
\be \nonumber
\int_{B_1^4(0)} |g|^2 \, dx \leq C \left[ \int_{B_1^4(0) } |\nabla g|^2 \, dx + \alpha \int_{\partial B_1^4(0)} |g|^2 \, dS   \right].
\ee
\end{lemma}

We collect some basic properties of the modified Bessel functions of the first kind in the following lemma. The proof can be found in \cite{Baricz, Ifantis, Watson}.
\begin{lemma}\label{lemBessel}
Let $I_1(z)$ be the modified Bessel function of the first kind, i.e. it is the solution of the problem
\eqref{eqBessel1}.
Assume that $0 < x <y$, it holds that
\be \label{Bessel1}
e^{x - y} \frac{x}{y} < \frac{I_1(x)}{I_1 (y)} < e^{x - y} \left(\frac{y}{x}\right)^{1/2}.
\ee
Furthermore, for every $x>0$, it holds that
\be\label{Bessel1-5}
\frac{x}{2}\leq  I_1(x)  \leq \frac{x}{2} \cosh x
\ee
and
\be \label{Bessel2}
 0\leq I_1^{\prime}(x)  \leq I_1(x)+\frac{I_1(x)}{x}.
\ee
\end{lemma}

The integrals of the modified Bessel functions of the first kind are estimated in the following lemma. The proof can be found in \cite{WX1}.
\begin{lemma}\label{AlemBessel2}
It holds that
\be \label{A-96}
\int_0^1 | I_1(|\xi| r)|^2 r\, dr
\leq  C \min\{1, |\xi|^{-1}\} (I_1(|\xi|))^2
\ee
and
\be \label{A-97}
 \int_0^1 \left| \frac{d}{dr} \left(r I_1 (|\xi| r) \right) \right|^2 \frac1r \, dr
\leq  C \max\{1, |\xi|\} (I_1(|\xi|))^2.
\ee
Furthermore, one has
\be\label{A-98}
\int_0^1 | \mathcal{L}I_1(|\xi| r)|^2 r\, dr
\leq  C  \min\{1, |\xi|^{-1}\} \xi^4 (I_1(|\xi|))^2
\ee
and
\be\label{A-99}
\int_0^1 \left| \frac{d}{dr} \left( r \mathcal{L} I_1(|\xi|r)\right)  \right|^2 \frac1r \, dr
\leq  C  \max\{1, |\xi|\} \xi^4 (I_1(|\xi|))^2.
\ee
\end{lemma}


\section{Existence of solutions for the linearized problem}\label{sec-ex}
In this section, we use  Galerkin method to  prove the existence of solutions to the problem \eqref{2-0-8} for each fixed $n$. The proof here is quite similar to that for the problem with no slip boundary conditions in \cite{WX1}.

\subsection{A priori estimates of $\psi_n$ in terms of $f_n$}
First, we give some a priori estimates of $\psi_n$ in terms of $f_n$.
\begin{lemma}\label{lemma4-1}
Let $\psi_n$ be a smooth solution of the problem \eqref{2-0-8}, then it holds that
\be \label{4-1}
(\Phi |n| )^2 \int_0^1 |\psi_n|^2 r \, dr \leq C \int_0^1 |f_n|^2 r \, dr
\ee
and
\be \label{4-2}
\int_0^1 |\mathcal{L} \psi_n|^2 r \, dr + n^2 \int_0^1 \left|\frac{d}{dr}(r \psi_n) \right|^2 \frac1r \, dr + n^4 \int_0^1 |\psi_n|^2 r \, dr
\leq C \int_0^1 |f_n|^2 r \, dr.
\ee
\end{lemma}

\begin{proof}
The equation \eqref{2-1-7} gives that
\be \nonumber
n \int_0^1 \frac{\bar{U}(r)}{r} \left| \frac{d}{dr} (r \psi_n)   \right|^2 \, dr + n^3 \int_0^1 \bar{U}(r) |\psi_n|^2 r \, dr = - \Im \int_0^1 f_n \overline{\psi_n} r\, dr.
\ee
By Lemma \ref{lemmaHLP} and the fact that $\bar{U}(r) \geq \frac{\Phi}{\pi}(1 - r^2)$, one has
\be \label{4-5} \ba
& \Phi |n|^2 \int_0^1\Phi |\psi_n|^2 r  +  \frac{\bar{U}(r)}{r} \left| \frac{d}{dr}(r \psi_n) \right|^2 +  |n|^2 \bar{U}(r) |\psi_n|^2 r \, dr
 \leq   C \int_0^1 |f_n|^2 r \, dr.
\ea
\ee
Taking the estimate \eqref{4-5} into \eqref{2-1-5} and applying Lemma \ref{lemma1} yield
\be \label{4-6} \ba
& \int_0^1 |\mathcal{L} \psi_n|^2 r  + 2n^2  \left|  \frac{d}{dr} (r \psi_n) \right|^2 \frac1r  + n^4  |\psi_n|^2 r \, dr +  \alpha \left| \frac{d}{dr}(r \psi_n)(1) \right|^2 \\
\leq & \int_0^1 |f_n| |\psi_n| r\, dr + C \Phi |n| \int_0^1 \left|  \frac{d}{dr}(r \psi_n) \right| |r\psi_n| \, dr \\
\leq & \int_0^1 |f_n| |\psi_n| r\, dr + \frac12 \int_0^1 |\mathcal{L} \psi_n|^2 r \, dr + C (\Phi |n|)^2 \int_0^1 |\psi_n|^2 r \, dr .
\ea
\ee
Then the estimate \eqref{4-2} follows from \eqref{4-5} and Lemma \ref{lemma1}. Hence the proof of Lemma \ref{lemma4-1} is completed.
\end{proof}

Using the similar idea as in the proof of Lemma \ref{lemma4-1}, one can prove the following high order estimates for $\psi_n$.
\begin{lemma}\label{lemma4-2} Assume that $\psi_n$ is a smooth solution to the problem \eqref{2-0-8}, then it holds that
\be \label{4-11}
n^4 \int_0^1 |\mathcal{L} \psi_n|^2 r  + n^2  \left|\frac{d}{dr}(r \psi_n) \right|^2 \frac1r + n^4 |\psi_n|^2 r \, dr
\leq C (1 + \Phi) \int_0^1 |f_n|^2 r \, dr
\ee
and
\be \label{4-12}
\int_0^1 |\mathcal{L}^2 \psi_n|^2 r \, dr \leq C (1 + \Phi^2) \int_0^1 |f_n|^2 r \, dr .
\ee
\end{lemma}

\begin{proof}
Multiplying the equation \eqref{4-6} by $n^2$ gives
\be \label{4-13} \ba
& n^2 \int_0^1 |\mathcal{L} \psi_n|^2 r \, dr + 2n^4 \int_0^1 \left|\frac{d}{dr}(r \psi_n)\right|^2 \frac1r \, dr + n^6 \int_0^1 |\psi_n|^2 r \, dr \\
\leq & n^2 \int_0^1 |f_n| |\psi_n| r \, dr + C \Phi |n|^3 \int_0^1 \left| \frac{d}{dr}(r \psi_n)  \right| |r \psi_n|\, dr \\
\leq & n^2 \int_0^1 |f_n| |\psi_n| r \, dr + n^4 \int_0^1 \left| \frac{d}{dr}(r \psi_n) \right|^2 \frac1r \, dr + C \Phi^2 n^2 \int_0^1 |\psi_n|^2 r \, dr .
\ea
\ee
According to Lemma \ref{lemma4-1}, one has
\be \label{4-14}
n^2 \int_0^1 |\mathcal{L} \psi_n|^2 r \, dr + n^4 \int_0^1 \left|\frac{d}{dr}(r \psi_n)\right|^2 \frac1r \, dr + n^6 \int_0^1 |\psi_n|^2 r \, dr
\leq C \int_0^1 |f_n|^2 r \, dr.
\ee

Similarly, multiplying \eqref{4-6} by $n^4$ yields
\be \label{4-15}
\ba
& n^4 \int_0^1 |\mathcal{L} \psi_n|^2 r \, dr + 2n^6 \int_0^1 \left|\frac{d}{dr}(r \psi_n)\right|^2 \frac1r \, dr + n^8 \int_0^1 |\psi_n|^2 r \, dr \\
\leq & n^4 \int_0^1 |f_n| |\psi_n| r \, dr + C \Phi |n|^5 \int_0^1 \left| \frac{d}{dr}(r \psi_n)  \right| |r \psi_n|\, dr \\
\leq & n^4 \int_0^1 |f_n| |\psi_n| r \, dr + C \Phi n^4 \int_0^1 \left| \frac{d}{dr}(r \psi_n) \right|^2 \frac1r \, dr + C \Phi  n^6 \int_0^1 |\psi_n|^2 r \, dr .
\ea
\ee
It follows from  \eqref{4-14} that  one has \eqref{4-11}.

Multiplying \eqref{2-0-8} by $\mathcal{L}^2 \overline{\psi_n}$ and integrating over $[0, 1]$ yield
\be \label{4-16} \ba
& in \int_0^1 \bar{U}(r) \mathcal{L} \psi_n \mathcal{L}^2 \overline{\psi_n} r\, dr -  in^3 \int_0^1 \bar{U}(r) \psi_n \mathcal{L}^2 \overline{\psi_n} r \, dr \\
& \ \ \ - \int_0^1 |\mathcal{L}^2 \psi_n|^2 r - 2n^2  \mathcal{L}\psi_n \mathcal{L}^2 \overline{\psi_n} r+  n^4  \psi_n \mathcal{L}^2 \overline{\psi_n} r\, dr = \int_0^1 f_n \mathcal{L}^2 \overline{\psi_n} r\, dr .
\ea \ee
It follows from the estimate \eqref{4-14}  and Cauchy-Schwarz inequality that
\be \label{4-17} \ba
\left|  n \int_0^1 \bar{U}(r)  \mathcal{L} \psi_n \mathcal{L}^2 \overline{\psi_n} r \, dr \right|
\leq &  C \Phi^2 n^2 \int_0^1 |\mathcal{L} \psi_n|^2 r \, dr + \frac18 \int_0^1 |\mathcal{L}^2 \psi_n|^2 r \, dr \\
\leq & C \Phi^2 \int_0^1 |f_n|^2 r \, dr + \frac18 \int_0^1 |\mathcal{L}^2 \psi_n|^2 r \, dr
\ea
\ee
and
\be \label{4-18} \ba
& \left|  n^3 \int_0^1 \bar{U}(r) \psi_n \mathcal{L}^2 \overline{\psi_n} r \, dr  \right| \leq C \Phi^2 n^6 \int_0^1 |\psi_n|^2 r \, dr + \frac18 \int_0^1 |\mathcal{L}^2 \psi_n|^2 r \, dr \\
\leq & C\Phi^2 \int_0^1 |f_n|^2 r \, dr + \frac18 \int_0^1 |\mathcal{L}^2 \psi_n|^2 r \, dr.
\ea \ee
Similarly, it holds that
\be  \label{4-19} \ba
& 2n^2 \int_0^1 \mathcal{L} \psi_n \mathcal{L}^2 \overline{\psi_n} r  + n^4 \psi_n \mathcal{L}^2 \overline{\psi_n} r\, dr
\leq  C(1 + \Phi) \int_0^1 |f_n|^2 r \, dr + \frac14 \int_0^1 |\mathcal{L}^2 \psi_n|^2 r \, dr .
\ea
\ee
Therefore, combining all the estimates \eqref{4-16}-\eqref{4-19} gives \eqref{4-12}. This completes the proof of Lemma \ref{lemma4-2}.
\end{proof}

\subsection{Existence of solutions to \eqref{2-0-8}}
 First, let us  introduce two function spaces.
\begin{definition}\label{def0}
  Denote
$$X= \left\{ \varphi \in C^\infty([0, 1])\ \left|  \ba  & \  \varphi(1) = 0, \ \ \mathcal{L}\varphi (1) = - \alpha \varphi^{\prime}(1),\ \ \lim_{r \rightarrow 0+} \mathcal{L}^k \varphi(r) = 0, \\
& \ \lim_{r \rightarrow 0+} \, \frac{d}{dr} ( r \mathcal{L}^k \varphi)(r) = 0,\  k\in \mathbb{N} \ea  \right. \right\}.$$
Let $X_0$ be the completion of $C^\infty([0,1])$ under the norm
$$\|\varphi\|_{X_0} := \left( \int_0^1 |\varphi|^2  r \, dr \right)^{\frac12 }. $$
Let $X_4$ be the closure of $X$ with respect to the following $X_4$-norm,
$$\|\varphi\|_{X_4} : = \displaystyle \left[
\int_0^1  \left(  |\varphi|^2 + |\mathcal{L} \varphi |^2 + |\mathcal{L}^2 \varphi|^2\right) r \, dr \right]^{\frac12} .$$
\end{definition}

In order to apply Galerkin method, one needs to construct an orthonormal basis for $X_0$. Our basic strategy is  to seek
a basis in the function space $X_4$. First, let us give some properties of functions in  $X_4$.

\begin{lemma}\la{lemma3-3-1}
 Assume that $\varphi$ is a function in $X_4$. Then $\varphi,$ $\frac{d}{dr}(r\varphi ) ,$  $\mathcal{L}\varphi$,
$\frac{d}{dr} ( r \mathcal{L} \varphi) $ $\in C([0, 1])$,
\begin{equation}\label{3-3-1-1}
\varphi (0) = \varphi (1) = \mathcal{L} \varphi (0) = 0, \ \ \ \mathcal{L}\varphi(1) = - \alpha \frac{d}{dr}(r \varphi)(1),
\end{equation}
and
\begin{equation} \label{3-3-1-2}
\lim_{r \rightarrow 0 +} \frac{1}{\sqrt{r}} \frac{d}{dr} ( r\varphi ) = \lim_{r \rightarrow 0 + }
\frac{1}{\sqrt{r}} \frac{d}{dr} ( r \mathcal{L} \varphi ) = 0.
\end{equation}
Moreover, there exists a positive constant $C$ independent of $\varphi$ such that
\be \la{3-3-1-3}
\int_0^1 \frac{|\varphi|^2}{r} \, dr + \int_0^1 \left|  \frac{d\varphi}{dr} \right|^2  r \, dr \leq C \int_0^1 | \mathcal{L} \varphi|^2  r \, dr
\ee
and
\be \la{3-3-1-4}
 \int_0^1 \left| \frac{d}{dr}( r \mathcal{L} \varphi) \right|^2 \frac1r \, dr + \int_0^1 \frac{|\mathcal{L}\varphi |^2}{r} \, dr
+ \int_0^1 \left| \frac{d}{dr} \mathcal{L} \varphi \right|^2  r \, dr \leq C \|\varphi \|_{X_4}^2.
\ee
\end{lemma}

\begin{proof}
Suppose that $\{ \varphi_j  \} \subseteq X $ is a sequence which converges to $\varphi$ in $X_4$. It follows from the proof of
Lemma \ref{lemma1} that
\be  \label{3-3-1-5}
\int_0^1 \left| \frac{d}{dr} ( r \varphi_j)  \right|^2  \frac{1}{r} \, dr
\leq  \int_0^1 \left| \mathcal{L} \varphi_j  \right|^2  r \, dr
\leq  \|\varphi_j \|_{X_4}^2.
\ee
With the aid of homogeneous boundary conditions $\varphi_j(0) = \varphi_j (1) =0$, it holds that
\be \label{3-3-1-6} \ba
\int_0^1 \left| \frac{d}{dr} (r \varphi_j)  \right|^2 \frac1r \, dr &=
\int_0^1 \left( \frac{d\varphi_j}{dr} + \frac{\varphi_j}{r} \right) \left( \frac{d \overline{\varphi_j} }{dr} + \frac{\overline{\varphi_j}}{r}  \right) r \, dr \\
& = \int_0^1 \left( \left| \frac{d \varphi_j }{dr}  \right|^2 r  + \frac{|\varphi_j |^2}{r} \right) \, dr .
\ea \ee
This, together with \eqref{3-3-1-5}, gives
\be \nonumber
\int_0^1   \left( \left| \frac{d\varphi}{dr} \right|^2 r + \frac{|\varphi|^2 }{r} \right) \, dr  = \int_0^1 \left|  \frac{d}{dr}(r \varphi) \right|^2 \frac1r \, dr
\leq \|\varphi\|_{X_4}^2.
\ee

For every $r \in [0, 1]$, one has
\be \nonumber
| r \varphi_j (r) | = \left|  \int_0^r \frac{d}{ds} [s \varphi_j (s)]  ds       \right|
\leq   r\left(  \int_0^1  \left|  \frac{d}{ds}[s \varphi_j (s)] \right|^2 \frac{1}{s} \, ds \right)^{\frac12}  .
\ee
This yields
\begin{equation}\label{3-3-1-7}
\sup_{r \in [0, 1] } | \varphi_j (r) | \leq  \left(  \int_0^1  \left|  \frac{d}{ds}[s \varphi_j (s)] \right|^2 \frac{1}{s} \, ds \right)^{\frac12}  \leq  \|\varphi_j \|_{X_4}.
\end{equation}
The inequality \eqref{3-3-1-7} also holds for $\varphi_j - \varphi_k$. Hence $\varphi \in C([0, 1])$, and
\begin{equation}\nonumber
\varphi(0) = \lim_{n \rightarrow + \infty} \varphi_j (0) = 0,
\ \ \ \ \ \varphi(1) = \lim_{j \rightarrow + \infty}  \varphi_j (1) = 0.
\end{equation}

It follows from Lemma \ref{lemmaA2} that
\be \label{3-3-1-8}
\left| \frac{d(r \varphi_j )}{dr} (1)    \right| \leq C \left( \int_0^1 |\mathcal{L} \varphi_j|^2 r \, dr\right)^{\frac12}.
\ee
For every $r \in ( 0 , 1]$,
\begin{equation}\nonumber \ba
\left|  \frac{d(r \varphi_j)}{dr}  \frac{1}{r}     \right|
& = \left| \frac{d(r \varphi_j) }{dr}(1) - \int_r^1 \mathcal{L} \varphi_j \, ds             \right| \\
& \leq C  \left( \int_0^1 |\mathcal{L} \varphi_j|^2 r \, dr\right)^{\frac12} + \left[ \int_0^1 | \mathcal{L} \varphi_j |^2  s \, ds     \right]^{\frac12}  \left[ \int_r^1 \frac{1}{s} \, ds   \right]^{\frac12}.
\ea
\end{equation}
This implies
\begin{equation}\label{3-3-1-10}
\left| \frac{d (r \varphi_j )}{dr} \right|  \leq C r (1 +  |\ln r |^{\frac12} ) \left[ \int_0^1 | \mathcal{L} \varphi_j|^2  s \, ds \right]^{\frac12}.
\end{equation}
Hence taking the limit for $j$ gives
\begin{equation} \label{3-3-1-11}
 \frac{d (r \varphi) }{dr} \in C([0, 1]) \ \ \ \mbox{ and}  \  \ \ \lim_{r \rightarrow 0 +} \frac{1}{\sqrt{r}}
\frac{d(r \varphi)}{dr} = 0.
\end{equation}

For every $0 \leq r \leq 1$, one has
\be\nonumber
 |  r \mathcal{L} \varphi_j (r) | = \left|  \int_0^r \frac{d}{ds} ( s \mathcal{L} \varphi_j ) \, ds        \right|
\leq C \left[  \int_0^r  \left|   \frac{d}{ds} ( s \mathcal{L} \varphi_j ) \right|^2  \frac{1}{s} \, ds   \right]^{\frac12} r .
\ee
This, together with  \eqref{3-3-1-20} in Lemma \ref{lemmaA2}, implies
\begin{equation}\label{3-3-1-25}
|\mathcal{L} \varphi_j (r) | \leq C \left[ \int_0^1 \left| \frac{d}{dr}(s \mathcal{L} \varphi_j )  \right|^2 \frac1s \, ds  \right]^{\frac12} \leq C \|\varphi_j \|_{X_4}.
\end{equation}
Hence taking the limit for $\varphi_j $ gives
\be \nonumber
\mathcal{L}\varphi \in C([0, 1])\ \ \ \ \lim_{r\rightarrow 0+} \mathcal{L} \varphi(r) = 0,\ \ \ \ \mbox{and}\ \ \ \mathcal{L} \varphi(1) =-
\alpha \frac{d}{dr} (r\varphi)(1).
\ee

Similarly, one has
\be \label{3-3-1-28} \ba
\left| \frac{1}{r} \frac{d}{dr} ( r \mathcal{L} \varphi_j )  \right|   = &\,\, \left|
\frac{d}{dr}( r \mathcal{L} \varphi_j ) (1)  - \int_r^1 \frac{d}{ds} \left[
\frac{1}{s} \frac{d}{ds} ( s \mathcal{L} \varphi_j ) \right] \, ds   \right| \\
 \leq &\,\, \left| \frac{d}{dr} ( r \mathcal{L} \varphi_j )(1)  \right|  +  \left( \int_0^1 |\mathcal{L}^2 \varphi_j |^2  s \, ds \right)^{\frac12}
 | \ln r|^{\frac12} \\
 \leq &\,\, C \|\varphi_j \|_{X_4}  \left( 1 + |\ln r|^{\frac12} \right),
\ea
\ee
where  the inequality \eqref{estlinfty} in Lemma \ref{lemmaA2} has been used to get the last inequality.
Therefore, taking the limit for $\varphi_j $ as $j\to \infty$ yields
\begin{equation}\nonumber
\frac{d}{dr}( r \mathcal{L} \varphi ) \in C([0, 1])\ \ \ \ \ \mbox{and}\ \ \ \ \lim_{r\rightarrow 0+} \frac{1}{\sqrt{r}} \frac{d}{dr}
( r \mathcal{L} \varphi ) = 0.
\end{equation}

By virtue of \eqref{3-3-1-25}-\eqref{3-3-1-28}, one can get that
\be \la{3-3-1-30}
\left| r \mathcal{L} \varphi(r) \right|
\leq \int_0^r \left|  \frac{d}{ds} \left[ s \mathcal{L} \varphi(s)   \right]    \right| \, ds
\leq C \int_0^r \left( s + s |\ln s|^{\frac12}  \right) \, ds  \|\varphi\|_{X_4}
\leq C r^{\frac74} \|\varphi\|_{X_4}
\ee
and
\be \la{3-3-1-31}
\left|  \frac{d}{dr} ( \mathcal{L} \varphi) (r) \right| = \left| \frac{1}{r} \frac{d}{dr}( r \mathcal{L} \varphi) - \frac{\mathcal{L} \varphi}{r}  \right|
\leq C \left(  r^{-\frac14} + |\ln r|^{\frac12} \right) \|\varphi \|_{X_4}.
\ee
Hence it holds that
\be \la{3-3-1-32}
\int_0^1 \left|  \frac{d}{dr} (  r \mathcal{L} \varphi) \right|^2 \frac1r + \frac{|\mathcal{L}\varphi|^2 }{r}  +
\left| \frac{d}{dr} (\mathcal{L} \varphi)  \right|^2  r \, dr \leq C \|\varphi \|_{X_4}.
\ee
This finishes the proof of Lemma \ref{lemma3-3-1}.
\end{proof}

Based on Lemma \ref{lemma3-3-1}, we have the following compactness result.


\begin{lemma}\la{lemma3-3-3}
$X_4$ is compactly embedded into $X_0$.
\end{lemma}

\begin{proof} Assume that $\{\varphi_j \}$ is a bounded sequence in $X_4$. Owing to Lemma \ref{lemma3-3-1}, it holds that
\be \nonumber
 \int_0^1 \left|  \frac{d \varphi_j }{dr}    \right|^2  r \, dr + \int_0^1  |\varphi_j  |^2  \frac1r \, dr
\leq C \|\varphi_j \|_{X_4}^2 .
\ee
Therefore, if  $\varphi_j $ is regarded as a radially symmetric function defined on $B_1(0) \subset \mathbb{R}^2$,
then $ \{ \varphi_j \}$ is a bounded sequence in $H^1(B_1(0))$.  It is well-known that $H^1(B_1(0))$ is compact in $L^2(B_1(0))$. Hence there is a subsequence of $\{\varphi_j \}$ (still labelled by $\{ \varphi_j \}$) and a radially symmetric function
$\varphi$, such that $\{ \varphi_j \}$ converges to $\varphi$ in $L^2(B_1(0))$.  Hence, $\{\varphi_j \} $ converges to $\varphi$ in $X_0$, which completes the proof of Lemma \ref{lemma3-3-3}.
\end{proof}

Before investigating the eigenfunction of the differential operator $\mathcal{L}^2$, we first give the existence of solutions for the associated PDE.
\begin{lemma}\la{lemma3-3-2}
Given $g\in X_0$, the following problem
\begin{equation}\label{3-3-1}
\left\{  \ba
&\mathcal{L}^2 \varphi = g \ \ \ \mbox{in}\ (0, 1),\\
&\varphi (0) = \mathcal{L} \varphi (0) =\varphi(1) = 0, \\
& \mathcal{L}\varphi (1) = - \alpha \varphi^{\prime}(1),
\ea
\right.
\end{equation}
has a unique solution $\varphi \in X_4$, and it holds that
\begin{equation}\nonumber
\|\varphi \|_{X_4}   \leq C \|g\|_{X_0},
\end{equation}
where the constant $C$ is independent of $g$.
\end{lemma}

\begin{proof}
{\it Step 1. Existence.}
We do not construct the solution to the problem \eqref{3-3-1} directly. Instead, we consider the following boundary value problem for a fourth order equation on $\overline{B_1^4}(0)\subset \mathbb{R}^4$,
\be \la{3-3-2}  \left\{ \ba
& \Delta^2_4 \phi = G\ \ \ \ \ \mbox{in}\ B_1^4(0),\\
& \phi = 0\ \ \ \ \ \ \ \mbox{on}\ \partial B_1^4(0), \\
& \Delta_4 \phi = - \alpha \frac{\partial \phi}{\partial \Bn}\ \ \ \ \mbox{on}\ \partial B_1^4(0),
\ea \right. \ee
where $B_1^4(0)$ is the unit ball centered at the origin in $\mathbb{R}^4$,
\be \nonumber
\Delta_4= \sum_{i=1}^4 \partial_{x_i}^2 \quad \text{and}\quad G(x_1, x_2, x_3, x_4) = g(r)/r\ \ \ \ \mbox{with}\ r = \left(\sum_{i=1}^4 x_i^2 \right)^{ \frac12 }.
\ee

By Lax-Milgram theorem, there exists a unique solution $\phi \in H^2(B_1^4(0))$ to the problem \eqref{3-3-2}. According to the regularity theory for elliptic equations (\cite{ADN}), the unique solution $\phi \in H^4(B_1^4(0) )$ satisfies
\be \label{3-3-3}
\|\phi\|_{H^4(B_1^4(0) )} \leq C \| G\|_{L^2(B_1^4(0))} = C \|g\|_{X_0}.
\ee
 Since $G$ is a radially symmetric function, due to the rotational invariance  of $\Delta_4$ and the uniqueness of solution to the problem \eqref{3-3-2},
$\phi$ is also a radially symmetric function, i.e.,
$\phi(x_1, x_2, x_3, x_4) = \phi(r)$.

Let $\varphi(r) = r \phi(r)$. It can be verified that
\be \nonumber \ba
\mathcal{L} \varphi & = \left(  \frac{d^2}{dr^2} + \frac{1}{r} \frac{d}{dr} - \frac{1}{r^2}       \right) ( r\phi)   = r \left(  \frac{d^2}{dr^2} + \frac{3}{r} \frac{d}{dr}             \right) \phi
 = r \Delta_4 \phi.
\ea\ee
Similarly,
\be \nonumber
\mathcal{L}^2 \varphi = \mathcal{L} ( \mathcal{L} \varphi) = \mathcal{L} (r \Delta_4 \phi) = r \Delta_4^2 \phi =g.
\ee
Clearly, at $r=1$, one has
\be \nonumber
\varphi(1) = \phi(1) =0,\ \ \ \ \ \ \mathcal{L} \varphi (1) = \Delta_4 \phi (1) =- \alpha \left[ r \frac{\partial \phi}{\partial r } \right](1) = -  \alpha \frac{d \varphi }{d r }(1).
\ee
Moreover,
\begin{equation}\label{3-3-9} \ba
\|\varphi\|_{X_4}^2  = &\,\, \int_0^1 \left(  |\mathcal{L}^2 \varphi |^2  + |\mathcal{L} \varphi |^2  +
|\varphi|^2 \right)  r \, dr
 =  \int_0^1 \left(     | \Delta_4^2 \phi |^2 + |\Delta_4 \phi |^2 + |\phi|^2     \right)  r^3 \, dr \\
 \leq &\,\, C \|\phi \|_{H^4(B_1^4(0))}^2
 \leq  C \|g\|_{X_0}^2.
\ea
\end{equation}

In fact, the function $G$ can be approximated by a sequence of smooth radially symmetric functions $\{ G_n \}$ in $L^2(B_1^4(0))$. The corresponding solutions $\{ \phi_n \}$ belong to $C^\infty(\overline{B_1^4(0)})$. This implies
$\{ \varphi_n = r \phi_n \}\subseteq X$. Hence, the solution $\varphi$ can be approximated by $\{ \varphi_n \}\subseteq X$ under
$X_4$-norm, thus $\varphi \in X_4$ and  $\varphi$ is a solution to the problem \eqref{3-3-1}.

{\it Step 2. Uniqueness.} Suppose $\varphi_1$ and $\varphi_2$ are both solutions to \eqref{3-3-1} in $X_4$. Taking
$\varphi_1 - \varphi_2$ as a test function,  by virtue of Lemma \ref{lemma3-3-1},
\be\nonumber
\ba
0= &\,\,\int_0^1 \mathcal{L}^2 ( \varphi_1 - \varphi_2) \overline{(\varphi_1 - \varphi_2)}  r \, dr \\
=&\,\, \displaystyle -  \int_0^1 \frac{1}{r} \frac{d}{dr} ( r \mathcal{L} \varphi_1  - r \mathcal{L}\varphi_2)
 \frac{d}{dr} \left[   \overline{ r (\varphi_1 - \varphi_2) }      \right] \, dr
+ \frac{d}{dr}( r \mathcal{L} \varphi_1 - r\mathcal{L} \varphi_2 )(1)  (\overline{\varphi_1 - \varphi_2})(1)\\
&\ \ \ \ \ \
- \lim_{r \rightarrow 0+ } \frac{d}{dr} ( r \mathcal{L} \varphi_1 - r \mathcal{L} \varphi_2)(r)
\lim_{r \rightarrow 0+} (\overline{ \varphi_1 - \varphi_2} )(r)\\
 =&\,\, - \displaystyle  \int_0^1 \frac{1}{r} \frac{d}{dr} ( r \mathcal{L} \varphi_1  - r \mathcal{L}\varphi_2)
 \frac{d}{dr} \left[   \overline{ r (\varphi_1 - \varphi_2) }      \right] \, dr \\
=&\,\, \int_0^1 | \mathcal{L} (\varphi_1 - \varphi_2) |^2  r \, dr
- ( \mathcal{L} \varphi_1 - \mathcal{L} \varphi_2) (1)  \frac{d}{dr}[ r (\overline{\varphi_1 - \varphi_2} ) ](1) \\
&\ \ \ \ \
+  \lim_{r \rightarrow 0+ } ( \mathcal{L} \varphi_1 - \mathcal{L} \varphi_2)(r)  \lim_{r\rightarrow 0+}
\frac{d}{dr}[r( \overline{\varphi_1 - \varphi_2} )]\\
 =&\,\,  \int_0^1 | \mathcal{L} (\varphi_1 - \varphi_2) |^2  r \, dr + \alpha \left| \frac{d}{dr}[r(\varphi_1 - \varphi_2)](1) \right|^2.
\ea
\end{equation}
Hence one has
\begin{equation}\nonumber
\int_0^1  |\varphi_1 - \varphi_2|^2  r \, dr  \leq C \int_0^1  |\mathcal{L} (\varphi_1 - \varphi_2)|^2  r \, dr = 0.
\end{equation}
This implies the uniqueness of solutions to \eqref{3-3-1} and thus completes the proof of Lemma \ref{lemma3-3-2}.
\end{proof}

Now we are ready to show the existence of orthonormal basis for $X_0$.
\begin{pro}\label{theorem3-3-4}
There exists an orthonormal basis $\{\varphi_n\}\subset X_4$ for $X_0$.
\end{pro}
\begin{proof}
Suppose that $\varphi$ is the unique solution to \eqref{3-3-1}. Define the solution operator as $\varphi = \mathcal{M} g$. Since
\be \nonumber
\int_0^1 \mathcal{M}g  \bar{g}  r  \, dr    =  \int_0^1 \varphi  \mathcal{L}^2 \overline{\varphi}  r \, dr
= \int_0^1 |\mathcal{L} \varphi|^2  r \, dr + \alpha \left| \frac{d}{dr} (r \varphi)(1) \right|^2 \geq 0 ,
\ee
$\mathcal{M}$ is a symmetric operator on $X_0$.
It follows from Hilbert-Schmidt theory for the symmetric operators (\cite{Lax}) and Lemma \ref{lemma3-3-3} that there exists an orthonormal basis of $X_0$. In fact,
the basis consists of the eigenfunctions of the operator $\mathcal{M}$.
\end{proof}

Now for every $n \in \mathbb{Z}$, the existence of a solution $\psi_n $ can be obtained by the  standard Galerkin approximation method together with the a priori estimates. Since all the a priori estimates hold for the approximate solutions, they also hold for the solution $\psi_n $. The uniqueness of the solution to \eqref{2-0-8} can also be established by a priori estimates.

Let $H_{r, n}^4(0, 1)$ denote the closure of $X$ under the $H_{r, n}^4$-norm which is defined as follows,
\be \nonumber
\|\varphi\|_{H_{r, n}^4(0, 1)} := \int_0^1 ( |\mathcal{\varphi}|^2 + n^4 |\varphi|^2 + |\mathcal{L} \varphi |^2)r\, dr + \int_0^1 (|\mathcal{L}^2\varphi |^2 + n^4 |\mathcal{L}\varphi|^2 + n^8 |\varphi|^2  ) r\, dr.
\ee

\begin{pro} \label{existence-stream} Assume that  $f_n(r) \in L_r^2(0, 1)$. There exists a unique solution $\psi_n \in H_{r, n}^4(0, 1)$ to the linear system \eqref{2-0-8},  and a positive constant $C$, which is independent of $f_n$,  $\Phi$ , $\alpha$ and $n$, such that
\begin{equation}\nonumber
\|\psi_n \|_{H^4_{r, n}(0, 1 )} \leq C (1 + \Phi) \|f_n \|_{L^2_r(0, 1)}.
\end{equation}
\end{pro}


\subsection{Regularity of the velocity field} \label{sec-reg}
In this section,   the regularity of the velocity field is obtained.

Let
\be \nonumber
\Bv^*_n = v^r_n e^{inz} \Be_r  + v^z_n e^{inz} \Be_z= in \psi_n e^{inz} \Be_r - \frac1r \frac{d}{dr}(r \psi_n) e^{inz} \Be_z,
\ \ \text{and}\ \  \Bo_n^\theta = (\mathcal{L} - n^2) \psi_n e^{inz} \Be_\theta .
\ee
 First,  the $L^2(\Omega)$-bound of $\Bv^*_n$ and $\Bo^\theta_n $ can be estimated as follows.
\begin{lemma}\la{lemma3-4-2}There exists a constant $C$, independent of $n$, such that
\begin{equation}\nonumber
\| \Bv^*_n \|_{L^2(\Omega)} + \|\Bo^\theta_n \|_{L^2(\Omega)} \leq C \|\psi_n \|_{H_{r, n}^4(0, 1 )}.
\end{equation}
\end{lemma}

\begin{proof}
Note that $ v^r_n =  i n   \psi_n $ and $ v^z_n = - \frac1r \frac{d}{dr}(r \psi_n  ) $. One has
\be \la{3-4-21}
\|v^r_n \|_{L^2(\Omega)}^2 =  2\pi  \int_0^1 n^2 | \psi_n |^2  r  \, dr
\leq C \|\psi_n\|_{H_{r, n}^4(0, 1) }^2
\ee
and
\be \la{3-4-22}
\|v^z_n\|_{L^2(\Omega)}^2 =  2\pi \int_0^1 \left| \frac{d}{d r} ( r \psi_n )       \right|^2  \frac1r  \, dr
\leq C \|\psi_n \|_{H_{r, n}^4(0, 1 )}^2.
\ee
It follows from \eqref{3-4-21} and \eqref{3-4-22} that
\be \la{3-4-23}
\|\Bv^*_n\|_{L^2(\Omega)} \leq C \|\psi_n \|_{H_{r, n}^4(0, 1)}.
\ee

Since $\omega^\theta_n = ( \mathcal{L} - n^2) \psi_n $,  it holds that
\begin{equation}\nonumber
  \int_0^1 |\omega_n^\theta|^2 r \, dr \leq \int_0^1 |(\mathcal{L} - n^2)\psi_n|^2 r \, dr
\leq C \|\psi_n \|_{H_{r, n}^4(0, 1)}^2.
\end{equation}
Therefore, one has
\begin{equation}\la{3-4-25}
\|\Bo^\theta_n \|_{L^2(\Omega)} \leq C \|\psi_n \|_{H_{r, n}^4(0, 1)}.
\end{equation}
This, together with \eqref{3-4-23}, proves Lemma \ref{lemma3-4-2}.
\end{proof}

Next,  as same as \cite[Lemma 5.4]{WX1}, one has $H^2(\Omega)$-estimate for $\Bo^\theta_n$.
\begin{lemma}\label{lemma3-4-4}
There exists a positive constant $C$ independent of $n$ such that
\begin{equation}\label{globalomega}
\|\Bo^\theta_n \|_{H^2(\Omega)} \leq C \|\psi_n \|_{H_{r, n}^4(0, 1)}.
\end{equation}
\end{lemma}
The  proof of Lemma \ref{lemma3-4-4} is exactly the same as that of   \cite[Lemma 5.4]{WX1}, so we omit the details here.

Now  the regularity of $\Bv^*_n$ can be improved to $H^3(\Omega)$.

\begin{lemma}\label{theorem3-4-5} There exists a positive constant $C$ independent of $n$ such that
\be \label{H3est}
\|\Bv^*_n \|_{H^3(\Omega) } \leq C \|\psi_n \|_{H_{r, n}^4(0, 1)} .
\ee
\end{lemma}

\begin{proof}
For every $0< r<1$, define $\Omega_r = (B_1(0) \setminus B_r(0) ) \times \mathbb{R}$.
Straightforward computations give
\be \label{vorteq2}
\Delta \Bv^*_n = -{\rm curl}~\Bo^\theta_n \ \ \ \ \mbox{in}\ \Omega_r.
\ee
 In fact,  the equation \eqref{vorteq2} holds on the whole domain $\Omega$. Suppose that $\Bp \in C^\infty_0(\Omega)$ is a  vector-valued function defined on $\Omega$,
\be \la{3-4-61} \ba
&\,\,\int_{-\pi}^{+\pi} \int_{B_1(0)\setminus B_r(0)} \Bv^*_n \cdot \Delta \Bp \, dx dy dz \\
  = &\,\, \int_{-\pi }^{+ \pi } \int_{B_1(0) \setminus B_r(0)} \Delta \Bv^*_n \cdot \Bp \, dxdydz
+  \int_{-\pi }^{+\pi } \int_{\partial B_r (0) } \frac{\partial \Bv^*_n }{\partial \Bn } \cdot \Bp \, dSdz \\
&\ \ \ \ \ \ \ \ - \int_{-\pi }^{+\pi } \int_{\partial B_r (0) } \Bv^*_n \cdot \frac{\partial \Bp }{\partial \Bn} \, dSdz.
\ea \ee
On $\partial B_r(0) \times \mathbb{R}$,  it holds that
\be \nonumber
\frac{\partial \Bv^*_n}{\partial \Bn } = \frac{d v^r}{d r}e^{inz} \Be_r
+ \frac{d v^z}{d r}e^{inz} \Be_z = in \frac{d\psi_n }{dr}   e^{inz} \Be_r - \mathcal{L}\psi_n e^{inz}
\Be_z.
\ee
Therefore, one has
\be \nonumber \ba
&\,\, \left|   \int_{-\pi }^{+ \pi } \int_{\partial B_r(0)}  \frac{\partial \Bv^*_n }{\partial \Bn} \cdot \Bp \, dSdz       \right| \\
 \leq &\,\,   \int_{-\pi }^{+ \pi } \int_{\partial B_r(0)}  |n|  \left|  \frac{d \psi_n }{d r} \right|   \left| \frac{\partial \Bp }{\partial z}  \right| \, dSdz +
 \int_{-\pi }^{+ \pi } \int_{\partial B_r(0)} | \mathcal{L}\psi_n |  |\Bp | \, dSdz\\
 \leq &\,\, C(n)  \sup_{\Omega} | \partial_z \Bp |  \cdot r \left| \frac{d  \psi_n }{d r}(r) \right|    + C \sup_{\Omega} | \Bp | \cdot r | \mathcal{L} \psi_n (r) | .
\ea
\ee
It follows from the proof of Lemma \ref{lemma3-3-1} that one has
\be \nonumber
 \left|  r \frac{d \psi_n}{d r} \right|
= \left|  \frac{ d}{d r}(r \psi_n) - \psi_n  \right|
\leq C (r |\ln r|^{\frac12} + r^{\frac34} )   \|\psi_n \|_{H_{r, n}^4 (0, 1) }
\ee
and
\be \nonumber
| r \mathcal{L} \psi_n (r) | \leq C r^{\frac74} \|\psi_n  \|_{H_{r, n}^4(0, 1)}.
\ee
Hence it holds that
\be \la{3-4-65}
\lim_{r \rightarrow 0+ } \left|   \int_{-\pi }^{+ \pi } \int_{\partial B_r(0)}  \frac{\partial \Bv^*_n  }{\partial \Bn} \cdot \Bp \, dSdz       \right| = 0.
\ee
Similarly, one has
\be \la{3-4-66}
\lim_{r \rightarrow 0+ } \left| \int_{-\pi }^{+\pi } \int_{\partial B_r (0) } \Bv^*_n \cdot \frac{\partial \Bp }{\partial \Bn} \, dSdz \right| = 0.
\ee
Taking the limit of \eqref{3-4-61} as $r\rightarrow 0+ $ yields
\be \nonumber
\Delta \Bv^*_n = - {\rm curl}~\Bo^\theta_n ,\ \ \ \mbox{in}\ \Omega.
\ee

It follows from the regularity theory for the elliptic equation \eqref{vorteq2}(\cite{GT}) and the trace theorem for axisymmetric functions that one has
\be \label{3-4-67} \ba
\| \Bv^*_n \|_{H^3(\Omega) } & \leq C \|\Bo^\theta_n \|_{H^2(\Omega)} +  C \|\Bv^*_n\|_{H^{\frac52} (\partial \Omega)}  \\
& \leq C \|\Bo^\theta_n \|_{H^2(\Omega )} + C \| \Bv^*_n \|_{H^1(\Omega)} +  C \|\partial_z \Bv^*_n \|_{H^1(\Omega)}
+ C \|\partial_z^2 \Bv^*_n \|_{H^1(\Omega)} .
\ea
\ee
Herein,
\be \label{3-4-68} \ba
 \|\Bv^*_n\|_{H^1(\Omega)}^2
 =  \|\omega_n^\theta \|_{L^2(\Omega)}^2 +  \|\Bv_n^*\|_{L^2(\Omega)}^2    \leq C \|\psi\|_{H_{r, n}^4(0, 1)}^2.
\ea
\ee
Similarly, one has
\be \label{3-4-69}  \ba
& \| \partial_z \Bv^*_n \|_{H^1(\Omega)}^2 + \|\partial_z^2 \Bv^*_n \|_{H^1(\Omega)}^2 \\
 \leq &  \|\partial_z \Bo_n^\theta \|_{L^2(\Omega)}^2 +  \|\partial_z \Bv_n^*\|_{L^2(\Omega)}^2 +  \|\partial_z^2 \Bo^\theta_n \|_{L^2(\Omega)}^2
 +  \|\partial_z^2 \Bv_n^* \|_{L^2(\Omega)}^2
 \\ \leq & C \|\psi_n\|_{H_{r, n}^4(D)}.
\ea \ee
The estimates \eqref{3-4-67}--\eqref{3-4-69} together with Lemma \ref{lemma3-4-4} completes  the proof of Lemma \ref{theorem3-4-5}.
\end{proof}

\begin{pro}\label{back} Let $\psi_n$ and $v^\theta_n$ be the solutions obtained in Propositions \ref{existence-stream} and \ref{swirl}, the corresponding velocity field $\Bv= \sum_{n \in \mathbb{Z}} \Bv^*_n + v_n^\theta e^{inz} \Be_z $ is a strong solution to the problem \eqref{2-0-1}-\eqref{slipBC1}, and satisfies the estimates
\be \label{back1}
\|\Bv^*\|_{H^3(\Omega)} \leq C (1 + \Phi ) \| \BF\|_{H^1(\Omega)} \ \ \ \ \text{and}\ \ \
\  \|\Bv^\theta\|_{H^2(\Omega) } \leq C  \|\BF\|_{L^2(\Omega)}.
\ee
\end{pro}

\begin{proof} Since $\psi_n$ is a solution to \eqref{2-0-8}, following similar argument as in the proof of Lemma \ref{theorem3-4-5}, one has
\be \label{vorticityeq}
{\rm curl}~\left( (\bBU \cdot \nabla ) \Bv^*_n + (\Bv^*_n \cdot \nabla) \bBU   \right) - {\rm curl}~(\Delta \Bv^*_n) = {\rm curl}~\BF^*_n\ \ \ \mbox{in}\ \Omega,
\ee
where $\BF^*_n = F^r_n \Be_r + F^z_n \Be_z$. Hence there exists some function $P_n$ with $\nabla P_n \in L^2(\Omega)$ such that
\be \label{ns-new}
 (\bBU \cdot \nabla ) \Bv^*_n + (\Bv^*_n \cdot \nabla) \bBU  - \Delta \Bv^*_n + \nabla P_n = \BF^*_n \ \ \ \mbox{in}\ \Omega.
\ee
Consequently,  $\Bv^*$ is a strong solution to the problem \eqref{2-0-1} and \eqref{slipBC1}.  According to Lemma \ref{theorem3-4-5} and Proposition \ref{existence-stream}, one has
\be \nonumber \ba
 \|\Bv^* \|_{H^3(\Omega)}^2 & \leq \sum_{n \in \mathbb{Z}} \|\Bv^*_n \|_{H^3(\Omega)}^2 \leq
  C \sum_{n \in \mathbb{Z}} \|\psi_n\|_{H_{r, n}^4 (0, 1)}^2  \leq C (1 + \Phi^2 )\sum_{n \in \mathbb{Z}} \|f_n \|_{L_r^2(0, 1)}^2 \\
   & \leq C (1 + \Phi^2 ) \|\BF^*\|_{H^1(\Omega)}^2.
\ea \ee
This, together with the result in Proposition \ref{swirl} for swirl velocity, finishes the proof of Proposition \ref{back}.
\end{proof}


{\bf Acknowledgement.}
The research of Wang was partially supported by NSFC grant 11671289. The research of  Xie was partially supported by  NSFC grants 11971307 and 11631008,   Natural Science Foundation of Shanghai 21ZR1433300, and Institute of Modern Analysis-A Frontier Research Center of Shanghai. The authors would like to thank Professors Congming Li and Zhouping Xin for helpful discussions.

\end{document}